\documentclass{amsart}

\usepackage{amsmath}
\usepackage{amsfonts}
\usepackage{amssymb}
\usepackage{stmaryrd}
\usepackage{amsrefs}
\usepackage{amsthm}
\usepackage{centernot}
\usepackage{graphicx}
\usepackage{subfigure}
\usepackage{fullpage}
\usepackage[page]{appendix}

\usepackage{tikz}
\usepackage{mathrsfs}
\usetikzlibrary{decorations.shapes}

\usepackage{float}
\usepackage{wrapfig}

\newtheorem{theorem}{Theorem}[section]
\newtheorem{lemma}[theorem]{Lemma}
\newtheorem{proposition}[theorem]{Proposition}

\theoremstyle{definition}
\newtheorem{definition}[theorem]{Definition}
\newtheorem{example}[theorem]{Example}
\newtheorem{assumption}[theorem]{Assumption}

\theoremstyle{remark}
\newtheorem{remark}[theorem]{Remark}

\newcommand{\Prob}{\ensuremath{\mathbb{P}}}
\newcommand{\Exp}{\ensuremath{\mathbb{E}}}
\newcommand{\R}{\ensuremath{\mathbb{R}}}
\newcommand{\expect}{\Exp}
\newcommand{\Def}{\ensuremath{:=}}
\newcommand{\lawof}{\ensuremath{\mathcal{L}}}
\newcommand{\Var}{\ensuremath{\operatorname{Var}}}
\newcommand{\diag}{\ensuremath{\operatorname{diag}}}
\newcommand{\Hess}{\ensuremath{\operatorname{Hess}}}
\newcommand{\Id}{\ensuremath{\operatorname{Id}}}
\newcommand*\Ent[1]{{\ensuremath{\operatorname{Ent}_{#1}}}}
\newcommand{\Cov}{\ensuremath{\operatorname{Cov}}}
\newcommand{\tendsto}{\ensuremath{\rightarrow}}
\newcommand{\tr}{\ensuremath{\operatorname{tr}}}
\newcommand{\Lp}{\ensuremath{\mathbf{L}}}

\newcommand*\cmt[1]{\ensuremath{X_{x^{#1},A}}}
\newcommand*\tps[1]{\ensuremath{\tilde{\mathcal{A}}_{2#1,n}}}
\newcommand*\tpoly[2]{{\ensuremath{q_{#2}^{\bar w_{#1}}}}}
\newcommand*\pvsteps[2]{{\ensuremath{S_{\shortuparrow}^{\bar w_{#1}}(#2)}}}
\newcommand*\phsteps[2]{{\ensuremath{S_{\shortrightarrow}^{\bar w_{#1}}(#2)}}}
\newcommand*\mystrut[1]{\vrule width0pt height0pt depth#1\relax}

\newcommand*\jeparam[1]{\ensuremath{ { n_{#1} - n }}} 
\newcommand*\jeparamsum{\ensuremath{{ n_{1} + n_{2} - 2n} }}

\begin{document}

\title{Global Fluctuations for Linear Statistics of $\beta$-Jacobi Ensembles}

\author{Ioana Dumitriu}
\address{Department of Mathematics, University of Washington}
\email{dumitriu@math.washington.edu}
\author{Elliot Paquette}
\address{Department of Mathematics, University of Washington}
\email{paquette@math.washington.edu}
\thanks{Both authors acknowledge the support of the NSF through CAREER award DMS-0847661.  The second author was additionally supported by NSF grant DMS-0806024.}
\date{\today}
\maketitle

\begin{abstract}
We study the global fluctuations for linear statistics of the form
$\sum_{i=1}^n f(\lambda_i)$ as $n \rightarrow \infty$, for 
$C^1$ functions $f$, and
$\lambda_1, \ldots, \lambda_n$ being the eigenvalues of a (general)
$\beta$-Jacobi ensemble \cite{KillipNenciu, EdelmanSutton}. The fluctuation from the mean ($\sum_{i=1}^n
f(\lambda_i) - \Exp \sum_{i=1}^n f(\lambda_i)$) is given asymptotically by a Gaussian process. 

We compute the covariance matrix for the process and show
that it is diagonalized by a shifted Chebyshev polynomial basis; in
addition, we analyze the deviation from the predicted mean for polynomial test functions, 
and we obtain a law of large numbers. 
\end{abstract}

\section{Introduction}

Global fluctuations for linear statistics, also known as central limit
theorems, have  been of interest to the random matrix community for
almost as long as the limiting properties of empirical spectral distributions (also known
sometimes as laws of large numbers). A variety of models and
eigenvalue distributions have been studied from this point of view, starting with the
classical Gaussian and Wishart matrices \cite{MehtaBook, Muirhead}, generalizations
thereof (Wigner and Wishart-like matrices) \cite{BaiSilverstein, LytovaPastur, CabanalDuvillard,Guionnet, ORourke, SinaiSoshnikov, Soshnikov}, tridiagonal models \cite{DumitriuEdelman,Popescu}, 
different eigenvalue potentials \cite{Johansson}, $\beta$-ensembles \cite{Forrester95, Killip}, classical compact groups \cite{Soshnikov00, DiaconisEvans}, banded matrices \cite{Guionnet, AndersonZeitouni}, permutations \cite{BenArousDang} and so on. The methods of approach range from
the classical method of moments \cite{DumitriuEdelman,AndersonZeitouni}, to free
probability \cite{CapitaineCasalis,Guionnet,KusalikMingoSpeicher,MingoNica} and stochastic calculus \cite{CabanalDuvillard}. 

To put it more concretely, we are interested in the following
problem. A linear statistic of an $n \times n$ matrix $A$ with
eigenvalues $\lambda_1, \lambda_2, \ldots, \lambda_n$ is a
functional of the form
\[
\mathcal{F}(A) := \sum_{i=1}^n f(\lambda_i)~,
\]
where $f$ is a function (we sometimes refer to them as \emph{test functions}) belonging to a certain class (which, depending
on the ensemble to whom $A$ belongs, may be as restrictive as the
class of polynomials, or as wide as $\mathbb{L}^2$). The first issue at hand
is to calculate the limit of $\frac{1}{n} \mathcal{F}(A)$ as $n \rightarrow
\infty$ (in case this exists), in other words, to find the limiting
empirical spectral distribution for the eigenvalues of $A$ (also known
as the law of large numbers). The second
issue is to examine the fluctuation from the mean, e.g., study
$$X_{f, A} := \mathcal{F}(A) - \Exp \mathcal{F}(A) ~,$$ perhaps under a
suitable scaling, and prove that $X_{f,A}$ converges in
distribution, here to a centered Gaussian variable whose variance depends on
$f$. The term ``global'' in ``global fluctuations'' refers to the fact that all eigenvalues contribute similarly to
$\mathcal{F}(A)$. 

The Jacobi ensemble (also known as Double Wishart) is one of 
many on which such studies have been performed. They
were introduced in connection with the MANOVA procedure
of statistics for measuring the likelihood of a multivariate linear model
\cite{BaiSilversteinBook,Muirhead}, and found to be of interest in quantum conductance and
log-gas theory \cite{Beenakker, ForresterBook}. One can describe them through their eigenvalue
distributions
\begin{eqnarray} \label{bj_distr}
d\mu_J (\lambda_1, \ldots, \lambda_n)& := & \tfrac1Z \prod_ i \lambda_i^{\tfrac{\beta}{2}\left[ n_1 - n + 1\right] -1}(1-\lambda_i)^{\tfrac{\beta}{2}\left[ n_2 - n + 1\right] -1} \prod_{i <j } |\lambda_i - \lambda_j |^{\beta} ~,
\end{eqnarray}
where $Z=Z(n,n_1,n_2,\beta)$ is a normalization constant. In full
generality, $\beta>0$, while $n_1$ and $n_2$ need not be positive integers; in fact,
the only constraints (which relate to the integrability of the
measure) are that $n_1, n_2 \geq n-1$.

In the case that $\beta \in \{1,2,4\}$ and $n_1, n_2 \in \mathbb{N},$ they admit full
matrix models (as $J = W_1^{1/2} (W_1+W_2)^{-1} W_1^{1/2}$, where $W_1,
W_2$ are Wishart matrices, hence the ``double Wishart''
name. For an extensive study of the $\beta =1$ case, as well as a clear
exposition of how these models arose, we refer to
\cite{Muirhead}; the other cases ($\beta = 2,4$) can be dealt with
similarly. 

Recently, it was shown that in these ``classical'' cases a different kind of
model can be constructed, starting from random projections,
rather than random Wishart matrices; or, equivalently, that ``chopping
off'' an appropriate corner of a unitary Haar matrix will yield a matrix
whose singular values, squared, are distributed according to 
\eqref{bj_distr} (discovered in \cite{Collins}, rediscovered in \cite{EdelmanSutton}). 

The greatest generality is achieved by the tridiagonal model \cite{KillipNenciu,EdelmanSutton}, which covers any $\beta>0$, and
removes the condition that $n_1, n_2 \in \mathbb{N}$. We give the
model below (hereafter referred to as the Edelman-Sutton model, as it appears most
clearly in their work \cite{EdelmanSutton}). Given the matrix $B_{\beta}$
defined as
\begin{eqnarray} \label{mmodel1}
 B_\beta & = & \begin{pmatrix}
c_ns_{n-1}'       &                 &                  &           & \\
-s_{n-1}c_{n-1}'  & c_{n-1}s_{n-2}' &                  &           & \\
                  &-s_{n-2}c_{n-2}' & c_{n-2}s_{n-3}'  &           & \\
                  &                 & \ddots           &  \ddots   & \\
                  &                 &                  &  -s_1c_1' & c_1 \\
\end{pmatrix}~,
\end{eqnarray}
with the variables $c_i, ~s_i$, $i=1, \ldots, n$, and $c_j', ~s_j'$, $j=1, \ldots,
(n-1)$ obeying the distribution laws and relationships
\begin{align} \label{mmodel2}
\left\{ c_1, c_2, \ldots, c_n, c_1', c_2', \ldots, c_{n-1}'\right
\} &\text{ mutually independent, } \\
c_i \sim \sqrt{\operatorname{Beta}(\tfrac{\beta}{2}(n_1 - n + i),\tfrac{\beta}{2}(n_2 - n + i))} ~~&\text{and}~~
c_j' \sim \sqrt{\operatorname{Beta}(\tfrac{\beta}{2}j,\tfrac{\beta}{2}(n_1 + n_2 - 2n + 1 + j))} \\
s_i = \sqrt{1-c_i^2} ~~&\text{and}~~ s_j' = \sqrt{1-c_j'^2}~~,
\end{align}
the eigenvalues of $A=B_\beta B_\beta^T$ are distributed according to
\eqref{bj_distr} (see \cite{EdelmanSutton}).

We are interested in the behavior of $X_{f,A}$ as
$n \tendsto \infty$ with $(n_1 + n_2 - 2n)$ growing linearly in $n$ and
with $\beta$ fixed. 
This is the only scaling regime in which the limiting spectral distribution is truly Jacobi.  

If either $n_1 \gg n$ or $n_2 \gg n$, in the case when $\beta = 1,2,4$, the Wishart matrices in the full models have $W_1 \approx \beta n_1 I_n$, respectively, $W_2 \approx \beta n_2 I_n$. For example if $n_2 \gg n$ and $n_2 \gg n_1$, this heurestic predicts that the Double Wishart model behaves like
\[
W_1 ( W_1 + W_2)^{-1} \approx
W_1 ( W_1 + \beta n_2 Id)^{-1} \approx W_1 / (\beta n_2),
\]
so that appropriately rescaling, Wishart behavior should appear
.  These heurestics are studied rigorously in Jiang~\cite{Jiang}. (The symmetric regime, $n_1 \gg n$ and $n_1 \gg n_2$, predicts Wishart behavior with a huge shift in eigenvalues.) 

  Conversely, in the sublinear growth cases, i.e. where $(n_1 + n_2 - 2n) \ll n,$ the Jacobi ensemble takes on behavior that looks much more like the classical compact groups.  This connection is explicit for $\beta=1,4$ and fixed values of $n_1 - n$ and $n_2 - n$ (see Proposition 3.1 of \cite{JohanssonCompactGroups}). 
These heurestics predict the correct limiting spectral distributions as well.  In the superlinear case, the limiting spectral distribution is a point mass (easily seen also from \ref{mmodel2}, which shows that the matrix $B_{\beta}B_{\beta}^T$is very close to a mulitple of the identity), while in the sublinear case, the limiting spectral distribution is the arcsine law. These statements about the limiting spectral distributions are straightforward exercises following the approach of Trotter~\cite{Trotter}.  We sketch this approach in the proof of the following proposition.
\begin{proposition}
\label{jacobi_llns}
Let $f$ be a continuous test function on $[0,1].$
\begin{enumerate}
\item If $\jeparamsum = o(n),$ then 
\[
\frac{1}{n}\sum_{i=1}^n f(\lambda_i) \to_{\Prob} \frac{1}{\pi}\int_0^1 \frac{f(x)}{\sqrt{x(1-x)}}~dx.
\]
\item If $n_1/n \to p$ and $n_2/n \to q,$ then
\[
\frac{1}{n}\sum_{i=1}^n f(\lambda_i) \to_{\Prob} \int_0^1 f(x)~d\mu(x),
\]
where $\mu$ has density
\[
d\mu(x) := \frac{p+q}{2\pi }\frac{\sqrt{-(x-\lambda_{-})(x-\lambda_{+})}}{x(1-x)}\mathbf{1}_{[\lambda_{-},\lambda_{+}]}~dx,
\]
and
\[
\lambda_{\pm} := \left[\sqrt{\tfrac{p}{p+q}(1-\tfrac{1}{p+q})} \pm \sqrt{\tfrac{1}{p+q}(1-\tfrac{p}{p+q})} \right]^2.
\]
\item If $\jeparamsum = \omega(n)$ and if $(\jeparam{1})/ (\jeparamsum) \to \lambda,$ then
\[
\frac{1}{n}\sum_{i=1}^n f(\lambda_i) \to_{\Prob} f(\lambda).
\]
\end{enumerate}
\end{proposition}

\begin{proof}
Regardless of the scales of $\jeparam{1}$ and $\jeparam{2},$ the limiting eigenvalue distribution can be understood by computing $A_\infty = B_\infty B_\infty^T.$  (Note that on taking the $\beta$ parameter to infinity, the $\operatorname{Beta}(\beta x, \beta y)$ variable in the matrix model converges in probability to $\frac{x}{x+y}.$  Replacing the $\operatorname{Beta}$ variables by these limits in $B_\beta$ gives the matrix $B_\infty.$)

By applying Stirling's approximation, it can be shown that there is a constant $C$ depending only on $\beta$ so that
\[
\Exp \left| c_i - \sqrt{\frac{\jeparam{1} + i }{\jeparamsum+2i}}\right|^2 \leq \frac{C}{i}.
\]
A similar bound holds for $c_i'$ and for $c_is_i.$  Applying all these bounds, it follows that
\begin{equation}
\label{trotter_estimate}
\Exp \left\| B_\beta B_\beta^T - B_\infty B_\infty^T \right\|_F^2 = O(\log n).
\end{equation}
From the fundamental realization of Trotter~\cite{Trotter}, any $o(n)$ bound on the expected-square Frobenius norm suffices to show that the ESDs of two matrix models are converging together as $n \to \infty.$   

It is now elementary to check that the limiting spectral distribution for $B_\infty B_\infty^T$ is that which is stated in the theorem in the sublinear and superlinear cases.  In the linear case, we compute the limiting distribution by way of the Jacobi differential recurrence formula, which we do in proving Theorem~\ref{moment_palindromy} (see \eqref{differential_recurrence}).
\end{proof}

In our study of the linear scaling regime, we apply a wide array of methods, starting with the method of moments (which often boils down to path-counting), special functions (orthogonal polynomial) theory and generating functions, as well as one important result from the work of Anderson and Zeitouni \cite{AndersonZeitouni} (more details in Section \ref{sec:extension}). 

As mentioned in the introductory paragraph, the study of global fluctuations of linear statistics for random
matrices spans a wide literature, and covers a broad spectrum of
models. We will only mention here a few works that are either closely related in scope, in model, or
those that have served as inspiration for our study. 

The method of moments, introduced by Wigner himself \cite{Wigner, Wigner58} and used for proving central limit theorems for polynomials of Wishart matrices by Jonsson \cite{Jonsson}, has been employed with great success by Sinai and Soshnikov \cite{SinaiSoshnikov}, Soshnikov \cite{Soshnikov}, P\'ech\'e and Soshnikov \cite{PecheSoshnikov}, etc., to obtain both central limit theorems for traces of large powers of random matrices and universality results for the fluctuations of the extremal eigenvalues in the case of Wigner and Wishart matrices. The method of moments has also been used by Dumitriu and Edelman \cite{DumitriuEdelman} to calculate the fluctuations in the case of $\beta$-Hermite and $\beta$-Laguerre ensembles (generalizations of the Gaussian and central Wishart ensembles for $\beta = 1,2,4$), in the case of polynomial functions. It is also one essential ingredient in the work of Anderson-Zeitouni \cite{AndersonZeitouni} on band matrices. 

It is worth mentioning that the method of moments is essentially equivalent in spirit (though not necessarily in form) to the Stieltjes transform methods used by Bai and Silverstein (e.g., \cite{BaiSilverstein}) to calculate central limit theorems for generalized Wishart matrices; for a good reference on the methodology involved, we recommend \cite{BaiSilversteinBook}. 

Another method for computing fluctuations of linear statistics
involves a stochastic calculus approach introduced by
Cabanal-Duvillard \cite{CabanalDuvillard} to prove a central limit theorem for Wishart
matrices in the case $\beta=2$; stochastic calculus was also used by
Guionnet \cite{Guionnet} in computing fluctuations for a class of band
matrices and sample covariance matrices, and by Guionnet and Zeitouni \cite{GuionnetZeitouniImproved} to calculate large
deviations for a wide class of random matrices. 

Other approaches to calculating fluctuations for linear functionals
for $\beta$-ensembles include the Capitaine and Casalis work \cite{CapitaineCasalis},
which, through free probability, obtains results for both Wishart and
Jacobi (Double Wishart) matrices in the case $\beta = 2$. The later
work of Kusalik, Mingo, and Speicher \cite{KusalikMingoSpeicher} builds on \cite{CapitaineCasalis} and on
results obtained by Mingo and Nica \cite{MingoNica} to obtain fluctuations
(second-order asymptotics) for random matrices (also in the case
$\beta = 2$). Finally, Chatterjee \cite{Chatterjee} has introduced the Stein method to
computing central limit theorems for a wide class of random matrices,
for analytic potentials. 

Specifically in the case of $\beta$-Jacobi ensembles, for an
``extremal'' class of $\beta$-Jacobi ensembles (when  $n_1 =
o(\sqrt{n_2})$ and $n = o( \sqrt{n_2})$), as mentioned before, Jiang \cite{Jiang} has
established a series of important results, among which are the
calculations of fluctuations, through approximation methods.

 For all $\beta$-Jacobi ensembles of fixed parameters, Killip \cite{Killip}
 proved that the fluctuations of \emph{macroscopic} statistics obey a CLT; this result is similar 
to the one we obtain, but in the case
that $f = \chi_I$ where $I$ is a (fixed, independent of $n$)
 finite union of intervals in $[0,1]$ and under a different normalization.
It is unclear how Killip's result changes if the parameters of the ensemble scale with $n,$ which is the regime studied here.  In addition, while our method does not allow us to obtain any results for discontinuous functions, it seems that going in the opposite direction -- using Killip's results to obtain fluctuation theorems for smooth functions -- would need \emph{microscopic} statistics, i.e. where the lengths 
 of the intervals shrink with $n$.  As Killip notes, the microscopic regime is much more difficult and is not covered in \cite{Killip}.

Last but by no means least, we would like to mention that the most
comprehensive results for linear functionals in the case of
$\beta$-ensembles found in the literature have been obtained by
Johansson \cite{Johansson}. The fluctuations obtained in \cite{Johansson} are true for
any $\beta>0$, in the case of Hermitian matrices, for a large class of (polynomial) potentials, and for a
large class of functions $f$ (in its full generality, Johansson's work
is applicable to $\mathcal{H}^{17/2}$ functions, where
$\mathcal{H}^{\alpha}$ stands for the corresponding Sobolev
space). The methods are analytical and make heavy use of potential
theory. In addition to the fluctuations, Johansson was also able to
obtain the \emph{deviation} from the mean (second-order asymptotics), for the same class of functions.

Johansson's results subsume the work \cite{DumitriuEdelman} in the case of $\beta$-Hermite matrices (general $\beta$, fixed potential $V(x) = x^2$), and have served as a ``moral'' (albeit not technical) inspiration to us in our quest. 

\subsection{Our results}

Our purpose in this paper is to calculate the global fluctuations for
$\beta$-Jacobi ensembles, for as large a class of functions $f$ as
possible.  By using concentration properties of the Jacobi ensemble and making use of a
theorem by Anderson and Zeitouni \cite{AndersonZeitouni}, cited below, 
we were able to obtain the fluctuations for all $\beta$ in the case of $C^1$ test
functions on $[0,1].$  We only obtain the deviation from the mean for polynomial test functions, and conjecture the deviation should extend to a larger class of functions.

Our asymptotic analysis will occur in the proportional scaling regime, and so we will make the following assumptions on the growth of $n_1$ and $n_2.$ 
\begin{assumption} 
\label{model_assumptions}
\text{Let $n_1 = pn$ and $n_2 = qn$ for some fixed $p,q \geq 1$ having $p+q > 2.$ }
\end{assumption}
Chebyshev polynomials are an essential ingredient to our proof, both for their analytic properties and their combinatorial
ones.  We define the shifted Chebyshev polynomials of the first kind, $\Gamma$, by
\[
\Gamma_n(x) = 
2T_n\left(\frac{ 2x - \lambda_{+} - \lambda_{-}}{\lambda_+ -\lambda_-}\right)
\]
where $T_n$ are the standard Chebyshev polynomials of the first kind, satisfying $T_n(\cos\theta) = \cos n\theta.$  By making a change of variables, it immediately follows that $\{\Gamma_n(x)\}_{n=0}^\infty$ are a complete orthonormal system for $\tilde{L}^2(\Omega),$ the weighted $L^2$ space induced by the inner product
\[
\left<f,g\right> = \frac{1}{2\pi}\int_{\lambda_{-}}^{\lambda_{+}} f(x)g(x)\frac{1}{\sqrt{\left(\lambda_+ - x\right)\left(x - \lambda_{-}\right)}}dx.
\]
Using this inner product, we define the Chebyshev coefficients
\begin{equation}
\label{chebyshev_coefficient}
\hat f(n) = \left<f, \Gamma_n\right> = \frac{1}{2\pi} \int_{\lambda_{-}}^{\lambda_{+}} f(x)\Gamma_n(x)\frac{1}{\sqrt{\left(\lambda_+ - x\right)\left(x - \lambda_{-}\right)}}dx.
\end{equation}

Our main result is given below. 

\begin{theorem}
\label{clt}
Let $A$ be an $n \times n$ $\beta$-Jacobi matrix, with $n_1, n_2$ satisfying Assumption \ref{model_assumptions}.  For any fixed $k \in \mathbb{N}$, the $k$-tuple $(X_{\Gamma_1, A}, \ldots, X_{\Gamma_k, A})$ 
converges in distribution to the $k$-tuple of independent centered normal variables $(Y_1, \ldots, Y_k)$, where
$Y_i$ has variance $\tfrac{2}{\beta} i.$ Further, for any $f$ continuously differentiable on $[0,1],$ the variable $X_{f,A}$ converges in
distribution to a centered normal variable $Y_f$, with variance given by
\[
\sigma_f^2 := \tfrac{2}{\beta} \sum_{n=1}^\infty n|\hat f (n)|^2,
\]
where $\hat f(n)$ is the $n^{th}$ Chebyshev coefficient, defined as in
\eqref{chebyshev_coefficient}.  
\end{theorem}

\begin{remark}
\label{variancecomparison}
In analogy with Fourier series on the unit circle, it is alluring to consider the condition above for $f$ as requiring one half a derivative, in the $L^2$ sense; we would expect for
\[
\tau_f^2 = \sum_{n=1}^\infty n^2 |\hat f(n)|^2
\]
to behave like the square-$L^2$ norm of $f',$ and this can be easily established.  Precisely,
\[
\tau_f^2 = \frac{1}{2\pi} \int_{\lambda_{-}}^{\lambda_{+}} |f'(x)|^2 \sqrt{ (\lambda_{+} - x)(x-\lambda_{-}) } dx, 
\]
where the proof follows from the identity $T_n'(x) = nU_{n-1}(x),$ with $U$ the Chebyshev polynomial of the second kind, and the orthonormality of $\{U_n\}$ with the weight $\sqrt{1-x^2}.$  Since $\tau^2_f \geq \sigma^2_f$, given the $C^1$ condition for $f$ on $[\lambda_{-}, \lambda_{+}]$, 
the variance in the Theorem~\ref{clt} is finite.
\end{remark}
\vspace{0.1in}
\begin{remark} Note that the case when $p=q=1$ is not covered. This is the case when neither one of the exponents of the ensemble grows to $\infty$; the method of proof collapses since one of the main ingredients, the ability to get uniform tail bounds for entries of the matrix is no longer true at the ``bottom right'' corner of the matrix, and as such the errors can no longer be accurately estimated by the same means. We present the results of some numerical simulations for this case in Section~\ref{sec:numerics}.  We also note that the theorem is proven by Johansson in the $\beta=2$ case by methods of orthogonal polynomial theory~\cite{JohanssonCompactGroups}.
\end{remark}
\vspace{0.1in}

Our second result concerns the deviation from the mean, and is restricted to polynomial functions. 


%
\begin{theorem}
\label{jacobi_expectation}
For any polynomial $\phi,$ 
\[
\Exp \tr(\phi(A)) = n\int_{\lambda_{-}}^{\lambda_{+}}\phi(x)d\mu(x) + (\tfrac2\beta-1)\int_{\lambda_{-}}^{\lambda_{+}} \phi(x) d\nu(x) + o(\tfrac1n),
\] 
where $\mu$ is as defined in Theorem~\ref{jacobi_llns} and $\nu$ is the signed measure with density
\[
d\nu := \tfrac14\delta_{\lambda_{-}}
+\tfrac14\delta_{\lambda_{+}}
-\frac{1}{2\pi\sqrt{-(x-\lambda_{+})(x-\lambda_{-})}}\mathbf{1}_{(\lambda_{-},\lambda_{+})}~dx.
\]
\end{theorem}


To structure of the paper follows the method of proof, which takes the following steps:
\begin{itemize}
\item[Step 1.] Prove a ``central limit theorem'' for polynomials;
\item[Step 2.] Find the class of polynomials which diagonalizes the covariance matrix for the resulting Gaussian process;
\item[Step 3.] Use concentration techniques to show that $C^1[0,1]$ linear statistics can be approximated by polynomial test functions in such a way that the variance of the difference of the two is small for all $n.$
\item[Step 4.] Prove that the approximation works asymptotically. 
\end{itemize}

The rest of the paper is structured as follows: after a reparameterization of the model (Section \ref{sec:reparameterization}), Section \ref{sec:polynomial_clt} covers Step 1 in the above ``recipe'': show that the fluctuations are Gaussian when the test functions are the monomials.  The proof extends the mechanism that was employed in \cite{DumitriuEdelman} for the $\beta$-Hermite and $\beta$-Laguerre ensembles.  In Section \ref{sec:covariance} we show that the limiting covariance is diagonalized in shifted Chebyhsev basis; the method employed is original and has to do with the generating
function of the covariance matrix. Section \ref{sec:extension} contains the proof
that the matrix model satisfies the necessary conditions
to apply the Anderson-Zeitouni theorem.  Section \ref{sec:expectation} contains the proof of Theorem \ref{jacobi_expectation} (calculating the
deviation from the mean for analytic functions). 
Section \ref{sec:numerics} contains experimental results for the case
that $p=q=1.$  Finally, we have
included three Appendices. Appendix \ref{sec:symmetric},
which is the longest of the three, contains the symmetric function theory
results necessary for the calculation of the deviation (Section
\ref{sec:expectation}); more explicitly, it contains the proof that the series
expansion of the functional $\mathcal{F}(A)$ for monomial$f$ has a
``palindromic'' quality (the mechanism here is similar to
the one employed in \cite{DumitriuEdelman}). Appendix \ref{sec:beta} shows the existence of a Poincar\'e inequality for Beta variables that is stronger than what can be proven using general log-concave theory.  Finally Appendix \ref{sec:rootbeta} shows a theorem of independent interest, which we proved in the course of an unsuccessful attempt to obtain our main result by a different approximation method: that ``square root of beta'' variables can be coupled to Gaussian variables in such a way as to have small variance.

\subsection{Reparameterization} 
\label{sec:reparameterization}

While the parameters given naturally arise in the full matrix model
(which exists only for $\beta = 1,2,4$), e.g., as the size ratios of the two
Wishart matrices involved, we choose to work with a slightly different
set of parameters for the purposes of this problem. 
Define parameters $a$ and $b$ by
\[
a \Def \frac{1}{p+q} ~~~\text{and}~~~ b \Def \frac{p}{p+q}.
\]
As we shall see, $a$ and $b$ allow us to express the results in a
``cleaner'', perhaps more natural form. They expose symmetries of the
asymptotics, which are invariant under the involution $a \mapsto 1-b,$ $b \mapsto 1-a.$

For the regime of consideration of Theorem~\ref{clt} the parameters $a$ and $b$ take on values in the triangle $0 < a < \tfrac12,$ and $a < b < 1-a.$  The limiting spectral distribution will have support given by
\[
\lambda_{\pm} = \left[\sqrt{b(1-a)} \pm \sqrt{a(1-b)} \right]^2.
\]

The reciprocal expression $\tfrac 2\beta$ appears frequently, with
some terms having polynomial dependence upon it.  Thus in the proofs
we have used $\alpha$ in place of $\tfrac 2\beta$.  The Jacobi
ensemble density, with these parameters, is expressed as
\begin{equation}
\label{jacobidensity}
d\mu_J(\lambda_1, \ldots, \lambda_n) = \tfrac1Z \prod_ i \lambda_i^{\tfrac{n}{\alpha}\left[\tfrac ba - 1\right] +\tfrac1\alpha-1}(1-\lambda_i)^{\tfrac{n}{\alpha}\left[\tfrac{1-b}{a} - 1\right] +\tfrac1\alpha-1} \prod_{i <j } |\lambda_i - \lambda_j |^{\tfrac 2\alpha}. 
\end{equation}
The tridiagonal matrix model with these parameters is given $A = B_\beta B_\beta^t$ where
\begin{equation}
\label{tridiagonal_model}
B_\beta = \begin{pmatrix}
c_ns_{n-1}'       &                 &                  &           & \\
-s_{n-1}c_{n-1}'  & c_{n-1}s_{n-2}' &                  &           & \\
                  &-s_{n-2}c_{n-2}' & c_{n-2}s_{n-3}'  &           & \\
                  &                 & \ddots           &  \ddots   & \\
                  &                 &                  &  -s_1c_1' & c_1 \\
\end{pmatrix},
\end{equation}
\begin{align*}
\left\{ c_1, c_2, \ldots, c_n, c_1', c_2', \ldots, c_{n-1}'\right
\} &\text{ mutually independent, } \\
c_i \sim \sqrt{\operatorname{Beta}(\tfrac{nb}{\alpha a}+\alpha^{-1}(i-n), \tfrac{n(1-b)}{\alpha a}+\alpha^{-1}(i-n)} ~~&\text{and}~~
c_i' \sim \sqrt{\operatorname{Beta}(\alpha^{-1}i, \tfrac{n}{\alpha a}+\alpha^{-1}(i-2n+1)}, \\
s_i = \sqrt{1-c_i^2} ~~&\text{and}~~ s_i' = \sqrt{1-c_i'^2}.
\end{align*}

\section{Polynomial Fluctuations}
\label{sec:polynomial_clt}
\subsection{Traces of Powers and Path Counting}
\label{sec:path_counting}
When the linear statistic $f$ is a polynomial, it can be computed
explicitly using powers of the matrix model.  By linearity, this
reduces to the study of monomials $\tr (A^k),$ and by the
tridiagonality of $A,$ there is a simple combinatorial expansion for
this trace.  In particular, these traces can be expressed in terms of certain lattice paths.
In this section we will study these lattice paths and develop their combinatorial properties.
We will use these combinatorial properties to compute the covariance of the limiting Gaussian
process for polynomial test functions.  Their properties are not needed for the proof that the
limiting fluctuations are Gaussian. 

\begin{definition} An \emph{alternating bridge} is a lattice path from $(0,0)$ to $(2k,0)$ using only the steps $(1,1),$ $(1,0),$ and $(1,-1)$ none of whose odd steps travel up and none of whose even steps travel down.  Let $\mathcal{A}_{2k}$ denote the collection of all such lattice paths.  Likewise, let $\mathcal{L}_k$ denote the collection of all lattice paths of length $k$ without the alternating property.
\end{definition}

\begin{remark} These paths bear some similarity to the \emph{alternating Motzkin Paths} which have been used to study the Laguerre Ensemble~\cite{DumitriuThesis}.  These paths differ in that Motzkin paths are restricted to stay above the $x$-axis, while these are allowed to go above and below the axis.
\end{remark}

For a lattice path $\bar w$ starting at $(0,k)$ with sequence of vertical coordinates $\{ w_0=k, w_1, w_2, \ldots \}$ and an $n \times n$ matrix $M$, define $M_{\bar w}$ to be the product
\[
M_{\bar w} = \prod_{i=0} M_{w_{2i},w_{2i+1}} M^T_{w_{2i+1},w_{2i+2}}
= \prod_{i=0} M_{w_{2i},w_{2i+1}} M_{w_{2i+2},w_{2i+1}},
\]
provided that all $n \leq w_i \leq 1.$  If the lattice path $\bar w$ walks off the edge of the matrix, in the sense that either some $w_i > n$ or $w_i < 1,$ then define $M_{\bar w} = 0.$ 

\begin{example}
A lattice path $\bar w$ and its associated product $M_{\bar w}.$

\begin{center}
\begin{tikzpicture}
\draw[step=1cm,gray, very thin] (-0.5, -2.5) grid (6.5, 0.5);
\draw (0,0) --
      (1,-1) --   
      (2,-1) --   
      (3,-2) --  
      (4,-1) --  
      (5,-1) --  
      (6,0);
\draw (-0.1cm, -2 cm) -- (0.1cm, -2 cm) node[anchor=east] {$4~$};
\draw (-0.1cm, -1 cm) -- (0.1cm, -1 cm) node[anchor=east] {$5~$};
\draw (-0.1cm, 0 cm) -- (0.1cm, 0 cm) node[anchor=east] {$6~$};
\draw[xshift=10cm,yshift=-0.5cm] node[text width=6cm] 
	{
		Provided the matrix $M$ is at least $6 \times 6,$ this lattice path $\bar w$ would produce the product $M_{\bar w} = M_{6,5}M^T_{5,5}M_{5,4}M^T_{4,5}M_{5,5}M^T_{5,6}$.
	};
\end{tikzpicture}
\end{center}

\end{example}
\noindent Expanding the trace,
\[
\tr A^k = \sum_{i=1}^n \left[\left(B_\beta B_\beta^T\right)^k\right]_{i,i}.
\]
The diagonal entries $[(B_\beta B_\beta^T)^k]_{i,i}$ can be written in terms of alternating bridges,
since for all $1 \leq i \leq n,$
\[
\left[\left(B_\beta B_\beta^T\right)^k\right]_{i,i} = \sum_{\bar w \in \mathcal{A}_{2k}} \left(B_\beta\right)_{\bar w+i},
\]
where $\bar w + i$ is the lattice path $\bar w$ shifted up by $i$.  For convenience, define $\tps{k}$ to be all alternating bridges that are shifted up to start at coordinates between $1$ and $n;$ we will refer to these lattice paths as \emph{tridiagonal trace paths}.  In terms of these paths, we can write the trace of a power of a matrix as
\[
\tr A^k = \sum_{\bar w \in \tps{k}} A_{\bar w}.
\]

When $n$ is large and $k$ is fixed, each $A_{\bar w}$ is approximated
by a substantially simpler quantity: every entry 
in a $2k \times 2k$ principal submatrix on the diagonal 
of $A$ is strongly approximated by a deterministic tridiagonal band matrix
 (c.f. Lemmas~\ref{rt_order} and~\ref{rt_order}).  Thus, endow an alternating bridge with a weight by giving each horizontal edge weight $x$ and each inclined edge weight $y.$  Define the weight of the bridge to be the product of the weights of its edges, and define $p_k(x,y)$ to be the sum of all the weights over all the paths in $\mathcal{A}_{2k}$.  If we let $h(\bar w)$ denote the number of horizontal steps taken by path $\bar w,$ then
\[
p_k(x,y) = \sum_{\bar w \in \mathcal{A}_{2k}} x^{h(\bar w)}y^{2k - h(\bar w)}.
\]
We are interested in finding the exponential generating
function for these $p_k,$ i.e. we will compute
\[
\mathscr{P}(t) = \sum_{k=0}^\infty \frac{t^k}{k!} p_k(x,y),
\]
and show that 
\begin{equation}
\label{pkexpgf}
\mathscr{P}(t) = e^{t(x^2+y^2)}I_0(2xyt),
\end{equation}
where $I_0$ is the modified Bessel function of the first kind.

These polynomials exhibit some nice combinatorial
properties.  Suppose that a path $\bar w \in \mathcal{A}_{2k}$ has $i$
up-steps.  Because the path returns to $0,$ it must also have $i$
down-steps.  Down-steps must be placed in odd positions, and up-steps
must be placed in even positions; as a result, the placement of the
up-steps is independent from the placement of the down-steps.  Thus,
there are exactly ${ k \choose i}{k \choose k- i}$ paths in
$\mathcal{A}_{2k}$ having $2i$ inclined steps.  Note, this argument
also shows that the number of inclined steps must be even.  Consequently, the number of horizontal steps is even as well, and we have shown
\begin{equation}
\label{pkevendegree}
p_k(x,y) = \sum_{ l = 0 }^k { k \choose l }^2 x^{2l}y^{2(k-l)} = y^{2k} {}_2F_1(-k,-k,1; (x/y)^{2k}).
\end{equation}
For definitions and properties of the hypergeometric function
${}_2F_1$, see \cite[page 556]{AS}.  As a consequence, we are able to compute the size of $\mathcal{A}_{2k}$ by simply evaluating this polynomial at $x=y=1,$
\[
\left| \mathcal{A}_{2k} \right| = \sum_{i=0}^{k} {k \choose i}^2 = {2k \choose k}.
\]

While the alternating structure naturally lends itself to describing traces of $A,$ there is another way to view $\mathcal{A}_{2k}$ which lends itself better to computing $\mathscr{P}(t).$  If $\bar w = w_1w_2\cdots w_{2k-1}w_{2k},$ for steps $w_i,$ then the concatenation of the steps $w_{2i-1}w_{2i}$ is one of $(2,1)$, $(2,-1)$ or $(2,0).$  Moreover, if it is either of the first two, then by the alternating structure, $w_{2i-1}w_{2i}$ must have been $(1,0)(1,1)$ or $(1,-1)(1,0)$ respectively.  If it was a horizontal step, then there are two possibilities, either $(1,0)(1,0)$ or $(1,-1)(1,1).$  

\begin{definition}
By concatenating pairs of steps, alternating bridges $\bar w$ are in bijective correspondence with lattice paths in $\mathcal{L}_k$ \emph{whose horizontal steps are $2$-colored}.  Let those horizontal steps corresponding to $(1,0)(1,0)$ be colored \emph{red}, and let those horizontal steps corresponding to $(1,-1)(1,1)$ be colored \emph{blue}.
\end{definition}

\begin{example}
Two alternating bridges with the overlaid $\mathcal{L}_k$ path.

\vspace{.25cm}

\begin{tikzpicture}
\def\xoff{8cm}
\draw[step=1cm,gray, very thin] (-0.5, -0.5) grid (6.5, 2.5);
\draw (0,0) --
      (1,0) --   
      (2,1) --   
      (3,1) --  
      (4,1) --  
      (5,0) --  
      (6,0);
\draw[thick,dotted] (0,0) -- (2,1);
\draw[thick,decorate,decoration=triangles,red] (2,1) -- (4,1);
\draw[thick,dotted] (4,1) -- (6,0);
\draw (-0.1cm, 2 cm) -- (0.1cm, 2 cm) node[anchor=east] {$2~$};
\draw (-0.1cm, 1 cm) -- (0.1cm, 1 cm) node[anchor=east] {$1~$};
\draw (-0.1cm, 0 cm) -- (0.1cm, 0 cm) node[anchor=east] {$0~$};
\draw[xshift=\xoff, step=1cm,gray, very thin] (-0.5, -0.5) grid (6.5, 2.5);
\draw[xshift=\xoff] (0,0) --
      (1,0) --   
      (2,1) --   
      (3,0) --  
      (4,1) --  
      (5,0) --  
      (6,0);
\draw[xshift=\xoff,thick,dotted] (0,0) -- (2,1);
\draw[xshift=\xoff,thick,decorate,decoration=crosses,blue] (2,1) -- (4,1);
\draw[xshift=\xoff,thick,dotted] (4,1) -- (6,0);
\draw[xshift=\xoff] (-0.1cm, 2 cm) -- (0.1cm, 2 cm) node[anchor=east] {$2~$};
\draw[xshift=\xoff] (-0.1cm, 1 cm) -- (0.1cm, 1 cm) node[anchor=east] {$1~$};
\draw[xshift=\xoff] (-0.1cm, 0 cm) -- (0.1cm, 0 cm) node[anchor=east] {$0~$};
\end{tikzpicture}

\vspace{.15cm}

\begin{center}
\fbox{
\footnotesize
\begin{minipage}[c]{0.83\linewidth}
\tikz {\path[thick,draw=black,dotted] 
        (0,0) -- (1,0); } Inclined Step.~~~~~~~~~~
\tikz {\path[thick,draw=red,
        decorate,decoration=triangles] 
        (0,0) -- (1,0);} Red Step,$(1,0)(1,0).$~~~~~~~~~~
\tikz {\path[thick,draw=blue,
        decorate,decoration=crosses] 
	(0,0) -- (1,0);} Blue Step, $(1,-1)(1,1).$
\end{minipage}}
\end{center}

\vspace{.25cm}

%
%
%
\end{example}

\begin{lemma}
\label{even_many_steps}
Let $\bar w \in \tps{k}$ be given, and define $\pvsteps{}{m}$ to be the number of times $\bar w$ walks from height $m$ to height $m+1$ or back, and let $\phsteps{}{m}$ be the number of times that $\bar w$ walks horizontally at height $m.$  Both $\phsteps{}{m}$ and $\pvsteps{}{m}$ are even. 
\end{lemma}
\begin{proof}
Let $\bar u$ be the colored lattice path from $\mathcal{L}_{k}$ that corresponds to $\bar w.$  Let $v$ be the number of steps that $\bar u$ makes between height $m$ and height $m+1$ and back.  Because $\bar u$ returns to its starting height, $v$ is even.  Let $R$ be the number of red horizontal steps (i.e. those resulting from a $(1,0)(1,0)$ pattern) that $\bar u$ makes at height $m,$ and let $B$ be the number of blue horizontal steps (those resulting from a $(1,-1)(1,1)$ pattern) that $\bar u$ makes at height $m+1.$  Because $\phsteps{}{m} = v + 2R$ and $\pvsteps{}{m} = v + 2B$, both are always even.
\end{proof}

The correspondence between colored $\mathcal{L}_k$ and $\mathcal{A}_{2k}$
allows the polynomials $p_k(x,y)$ to be represented in a third way.
We will define the weight of an \emph{uncolored} path $p \in \mathcal{L}_k$ to equal the
sum of the weights over all alternating bridges $\bar w$ to which its
colorings correspond.  Suppose that an alternating bridge $\bar w$ is in
correspondence with a colored path $p$, one with $r$ red edges and $b$
blue edges. 
Recall that $h(p)$ is the number of horizontal steps the path takes,
and therefore the weight of $\bar w$ is $(xy)^{k - h(p)}
x^{2r}y^{2b}.$  There are ${h(p) \choose r}$ ways of placing the $r$
red edges on the path (after which the placement of the $b$ blue edges is
determined).   
As the possible colorings of a fixed path $p$ are in
bijective correspondence with $\{1,0\}^{h(p)},$ it follows that
the sum of the weights corresponding to all different colorings of a
given path $p$ is $(xy)^{k - h(p)}(x^2+y^2)^{h(p)}.$  In conclusion, $p_k(x,y)$ can be written as
\[
p_k(x,y) = \sum_{ p \in \mathcal{L}_k} (xy)^{k - h(p)}(x^2+y^2)^{h(p)}.
\]

The subset of the lattice paths $\mathcal{L}_k$ that fixes a given horizontal edge is in bijective correspondence with $\mathcal{L}_{k-1}$, simply by removing the given edge. By inclusion-exclusion, it follows immediately that the lattice paths in $\mathcal{L}_k$ that have no horizontal steps are counted by 
\[
|\mathcal{L}_k| - {k \choose 1}|\mathcal{L}_{k-1}| 
+ {k \choose 2}|\mathcal{L}_{k-2}|
- {k \choose 3}|\mathcal{L}_{k-3}| + \cdots =
\begin{cases}
{k \choose \tfrac{k}{2}} & k \text{ even } \\
0 & k \text{ odd }.
\end{cases}
\] 
The correspondence between $\mathcal{L}_k$ with a fixed horizontal edge and $\mathcal{L}_{k-1}$ decreases the statistic $h(p)$ by exactly $1$, and so this inclusion-exclusion formula carries over to $p_k$ as
\[
p_k(x,y) - {k \choose 1}(x^2+y^2)p_{k-1}(x,y) 
+ {k \choose 2}(x^2+y^2)^2p_{k-2}(x,y) + \cdots =
\begin{cases}
(xy)^k {k \choose \tfrac{k}{2}} & k \text{ even } \\
0 & k \text{ odd }.
\end{cases}
\] 
This recurrence can be recast in terms of the exponential generating function $\mathscr{P}(t)$ to read
\[
\mathscr{P}(t)e^{-t(x^2+y^2)} = 1 
+ \frac{x^2y^2t^2}{2!}{2 \choose 1}
+ \frac{x^4y^4t^4}{4!}{4 \choose 2}
+ \frac{x^6y^6t^6}{6!}{6 \choose 3} + \cdots
= I_0(2xyt).
\]
Thus, we have shown \eqref{pkexpgf},
\[
\mathscr{P}(t) = e^{t(x^2+y^2)}I_0(2xyt).
\]
Working with this function proves to be somewhat complicated, and it will be convenient to instead use the Laplace transform of $\mathscr{P}(t).$  Let $\Lp_t[f(t)](\omega)$ denote the Laplace transform in the variable $t$
\[
\Lp_t[f(t)](\omega) = \int_0^\infty{ e^{-\omega t} } f(t) dt.
\]
When applicable, $\Lp_{s,t}$ will denote the Laplace transform in both
variables.  The calculation of the Laplace transform of
$\mathscr{P}(t)$ is simplified greatly by some elementary properties
of the Laplace transform and the known Laplace transforms of modified
Bessel functions.  All of these properties are available for reference
in \cite[Chapter 29]{AS}; properties of the modified Bessel functions
are available in \cite[Chapter 9]{AS}.  The Laplace transform of the modified Bessel functions $I_n$ is given by 
\begin{equation}
\label{bessel_laplace}
\Lp_t[ I_n(ct) ](\omega) = \frac{c^n}{\left(\omega + \sqrt{\omega^2 - c^2}\right)^n}\frac{1}{\sqrt{\omega^2-c^2}},~~~~~ \omega > c.
\end{equation}
If for some real value of $\omega_0,$ the Laplace transform is finite, then for any $\omega$ in the half plane $\Re \omega > \omega_0,$ the Laplace transform is finite.  Further, the transform satisfies the following identities
\begin{align}
\Lp_t[ e^{kt} f(t)](\omega) &= \Lp_t[ f(t) ](\omega - k), 
\label{laplace_identity1}
\\
\Lp_t[ t f(t)](\omega) &= -\frac{d}{d\omega}\Lp_t[ f(t) ](\omega).
\label{laplace_identity2}
\end{align}

We will show that \emph{a priori}, the Laplace transform of $\mathscr{P}(t)$ is finite in the half plane $\Re \omega > (x+y)^2.$  This follows as $I_n(2xyt)$ satisfies the simple estimate
\[
0 \leq I_n(2xyt) \leq e^{2xyt},
\]
for $t >0, 2xy>0,$ and thus
\[
0 \leq \mathscr{P}(t) \leq e^{t(x+y)^2}.
\]
Identity~\eqref{laplace_identity1} makes computing the Laplace transform of $\mathscr{P}(t)$ a simple substitution into \eqref{bessel_laplace}, as
\[
\Lp_t[\mathscr{P}(t)](\omega) 
=\Lp_t[e^{t(x^2+y^2)}I_0(2xyt)](\omega) 
=\Lp_t[I_0(2xyt)](\omega-x^2-y^2) 
= \frac{1}{\sqrt{(\omega -x^2-y^2)^2 - 4x^2y^2}}.
\] 
Using~\eqref{laplace_identity2}, it is possible to compute the Laplace transform of $\partial_x\mathscr{P}(t),$ which arises later.
\begin{lemma}
\label{derivative_characteristic}
\[\Lp_t[\partial_x\mathscr{P}(t)](\omega) = 
\frac{2x(\omega + y^2 - x^2)}{\left(\left(\omega-x^2-y^2\right)^2-4x^2y^2\right)^{\tfrac32}}, ~~~~ \omega > (x+y)^2. \]
\end{lemma}
\begin{proof}
This is a straightforward application of \eqref{laplace_identity1}, \eqref{laplace_identity2} and the identity $I_0(t)'  = I_1(t).$ 
\begin{align*}
\Lp_t[\partial_x\mathscr{P}(t)](\omega)
&= \Lp_t[2xte^{t(x^2+y^2)}I_0(2xyt) + 2yte^{t(x^2+y^2)}I_1(2xyt)](\omega) \\
&= -\partial_\omega\Lp_t[2xe^{t(x^2+y^2)}I_0(2xyt) + 2ye^{t(x^2+y^2)}I_1(2xyt)](\omega) \\
&= -\partial_\omega\left[\frac{2x}{\sqrt{\tilde{\omega}^2 - 4x^2y^2}} + \frac{2y}{\sqrt{ \tilde{\omega}^2 - 4x^2y^2}}\frac{2xy}{\tilde{\omega} + \sqrt{\tilde{\omega}^2 - 4x^2y^2}} \right], \\
\intertext{where $\tilde{\omega}$ is $\omega - x^2 - y^2$. Thus}  
\Lp_t[\partial_x\mathscr{P}(t)](\omega)
&= \frac{2x(\tilde{\omega} + 2y^2)}{\left(\tilde\omega^2-4x^2y^2\right)^{\tfrac32}} \\
&= \frac{2x(\omega + y^2 - x^2)}{\left(\left(\omega-x^2-y^2\right)^2-4x^2y^2\right)^{\tfrac32}}, ~~~~ \omega > (x+y)^2.
\end{align*}
\end{proof}

\begin{remark}
In a manner of speaking, we have circuitously arrived at the regular
generating function for $p_k(x,y),$  since it is possible to deduce
the generating function from the exponential generating function by
way of the Laplace transform, as follows. Let $\mathscr{P}^R(t)$ denote the generating function,
\[
\mathscr{P}^R(t) = \sum_{k=0}^\infty t^k p_k(x,y).
\]
The effect of taking the Laplace transform on an exponential generating function can be understood using the Gamma function.
\begin{align*}
\Lp_t[\mathscr{P}(t)](\omega)
&=\int_0^\infty{ e^{-\omega t} } \mathscr{P}(t) dt \\
&=\int_0^\infty{ e^{-\omega t} } \sum_{k=0}^\infty \frac{t^k}{k!}p_k(x,y) dt. \\
\intertext{ The order of summation and integration can be interchanged because $\frac{t^k}{k!}p_k(x,y)$ is always positive for $t>0$, $x,y \in \mathbf{R}$, }
\Lp_t[\mathscr{P}(t)](\omega)
&=\sum_{k=0}^\infty \int_0^\infty{ e^{-\omega t} } \frac{t^k}{k!}p_k(x,y)dt. \\
\intertext{ Make the change of variables $s = \omega t,$ so that}
\Lp_t[\mathscr{P}(t)](\omega)
&=\sum_{k=0}^\infty \int_0^\infty{ e^{-s} } \frac{s^{k}}{\omega^{k+1} k!}p_k(x,y)ds \\
&=\sum_{k=0}^\infty \omega^{-k-1}p_k(x,y)ds  \\
&=\mathscr{P}^R(\omega^{-1})\omega^{-1}. 
\end{align*}
Thus, putting everything together,
\[
\mathscr{P}^R(\omega) = \frac{1}{\sqrt{\left(1 - \omega(x^2+ y^2)\right)^2 - 4\omega^2x^2y^2}}.
\]
\end{remark}

\subsection{Asymptotic Normality of Fluctuations}

We show in this subsection that polynomial test functions asymptotically have jointly normal fluctuations.  This is the first component of Theorem~\ref{clt}, and we summarize the precise claim in the following proposition.
\begin{proposition}
\label{polynomial_clt}
Let $A$ be an $n \times n$ $\beta$-Jacobi matrix, with parameters as described in Section~\ref{sec:reparameterization}.  For any fixed $k \in \mathbb{N},$ the $k$-tuple $\left( \cmt{1}, \cmt{2}, \ldots, \cmt{k} \right)$ converges in distribution to a centered multivariate normal random variable.
\end{proposition}
The method of proof will be the computation of the moments.  Recall that a multivariate normal variable has mixed moments characterized by the Wick formula, which we will state precisely.
\begin{proposition}
\label{wick}
A centered random vector $\left(Z_1, Z_2, \ldots, Z_k\right)$ is a multivariate normal if and only if for each word $m \in \left[ k \right]^l,$ the mixed moments satisfy
\[
\expect \prod_{i \in m} Z_i = \begin{cases}
0 & \text{if $l$ is odd,} \\
\sum_{G} \prod_{\{a,b\} \in \mathcal{E}(G)} \expect Z_{m_{a}} Z_{m_{b}} & \text{if $l$ is even,}
\end{cases}
\]
where the sum is over all graphs $G$ that are perfect matchings on the vertices $[k],$ and where $\mathcal{E}(G)$ is the edge set of this graph.
\end{proposition}

To prove Proposition~\ref{polynomial_clt}, it suffices to show that all the mixed moments asymptotically obey the Wick formula.  Thus, our first goal is to show that the moments have the correct form.
\begin{proposition}
\label{mixedmoment_structure}
For a fixed word $m \in \left[ k \right]^l,$ 
\[
\expect \prod_{i \in m} \cmt{i} = \begin{cases}
O(n^{-1/2}) & \text{if $l$ is odd,} \\
\sum_{G} \prod_{\{a,b\} \in \mathcal{E}(G)} \expect \cmt{m_{a}} \cmt{m_{b}} +O(n^{-1/2})& \text{if $l$ is even,}
\end{cases}
\]
where the sum is over all graphs $G$ that are perfect matchings on the vertices $[k],$ and where $\mathcal{E}(G)$ is the edge set of this graph.
\end{proposition}
This nearly proves Proposition~\ref{polynomial_clt}, but it remains to show that the covariances have a limit.  We will delay this proof as we will identify the limiting covariance explicitly, and we begin in the direction of proving Proposition~\ref{mixedmoment_structure}.  In the sequel, fix some word $m \in \left[ k \right]^l.$  We will write the mixed moment indicated by $m$ in a way that exposes its asymptotically relevant terms.  The first step is to write the mixed moment in terms of tridiagonal trace paths.  

\begin{align}
\expect \prod_{u\in m} \cmt{k}
&=
\expect \prod_{u \in m} \left[\sum\nolimits_{\bar w \in \tps{u}} A_{\bar w} - \expect A_{\bar w} \right] \nonumber \\
\label{reduction0}
&=
\sum_{\bar w_1, \ldots, \bar w_{l}} \expect \prod_{i=1}^l \left[ A_{\bar w_i} - \expect A_{\bar w_i}\right],
\end{align} 
where the sum is over all tridiagonal trace paths $\left(\bar w_1, \ldots, \bar w_l\right) \in \tps{m_1} \times \tps{m_2} \times \cdots \times \tps{m_l}.$

Each nonzero random variable $A_{\bar w}$ is a product of terms of matrix entries.  More specifically, by Lemma~\ref{even_many_steps} trace paths visit each matrix entry an even number of times, and so $A_{\bar w}$ is a polynomial in the random variables $\{c_i^2\}$ and $\{(c_i')^2\}.$  Thus for each tridiagonal trace path $\bar w_i$ for which $A_{\bar w_i} \centernot{\equiv} 0,$ it is possible to define random variables $\tpoly{i}{j}$ with $1 \leq j \leq 2n-1$ so that
\begin{enumerate}
\item \tpoly{i}{j} is a polynomial in $c_j^2$ for $1 \leq j \leq n;$
\item \tpoly{i}{j+n} is a polynomial in $(c_j')^2$ for $1 \leq j \leq n-1;$
\item $A_{\bar w_i} = \prod_{j=1}^{2n-1} \tpoly{i}{j};$
\item The smallest nonzero coefficient of each $\tpoly{i}{j}$ is $1.$
\end{enumerate}
We will write $\tpoly{i}{j}(x)$ for the corresponding polynomial in $x,$ while when no argument is provided, we mean the random variable defined above.   This decomposition breaks a random variable $A_{\bar w_i}$ into a product of independent random variables.  Further, each polynomial has the form $\tpoly{i}{j}(x) = x^{a_{i,j}}(1-x)^{b_{i,j}}$ for some non-negative integer powers.  Note, however, that most of these polynomials are identically $1$.

We will use these polynomials to alternately express the difference $A_{\bar w} - \expect A_{\bar w}.$   Specifically, we telescope in the following way.
\begin{align}
A_{\bar w_i} - \expect A_{\bar w_i} 
&= 
\prod_{j=1}^{2n-1} \left[ ( \tpoly{i}{j} - \expect \tpoly{i}{j} ) + \expect \tpoly{i}{j} \right] - \prod_{j=1}^{2n-1}\expect \tpoly{i}{j} \notag \\
\label{reduction1}
&= \sum_{\substack{S \subset [2n-1] \\ S \neq \emptyset}} \left[ 
\prod_{j \in S} ( \tpoly{i}{j} - \expect \tpoly{i}{j} )
\prod_{j \not\in S} \expect \tpoly{i}{j}
\right].
\end{align}
In this last step we omit the empty set precisely because it is the term canceled by $\expect A_{\bar w_i}.$  

Note that in \eqref{reduction0} we require a product of $l$ of these terms.  Thus, by applying the \eqref{reduction1} multiple times, we can write
\begin{equation}
\label{reduction2}
\prod_{i=1}^l\left[
A_{\bar w_i} - \expect A_{\bar w_i} \right]
=
\sum_{S_1 \ldots S_l}
\prod_{i=1}^l 
\prod_{j \in S_i}( \tpoly{i}{j} - \expect \tpoly{i}{j} )
\prod_{j \not\in S_i}\expect \tpoly{i}{j},
\end{equation}
where it is important to note that the sum is over \emph{nonempty} subsets of $[2n-1].$ 

In expectation, we will see that each difference term $\tpoly{i}{j} - \expect \tpoly{i}{j}$ that appears in the product contributes a factor of $n^{-1/2},$ and thus that the magnitude of \eqref{reduction2} is at most $O(n^{-l/2}).$  To show this, we require the ability to estimate moments of the terms that appear in the right hand side of~\eqref{reduction2}.  This is expressed in the following lemma.
\begin{lemma}
\label{beta_moment_order}
Fix a polynomial $q(x) = x^{a_1}(1-x)^{a_2},$ and fix an $n \in \mathbb{N}.$  There is a constant $C = C(m,a_1,a_2)$ so that
\begin{align*}
\max_{1 \leq i \leq n} \expect \left|q(c_i^2) - \expect q(c_i^2)\right|^m &\leq C n^{-m/2}, \text{ and } \\
\max_{1 \leq i \leq n-1} \expect \left|q((c_i')^2) - \expect q((c_i')^2)\right|^m &\leq C n^{-m/2}.
\end{align*}
\end{lemma}
\begin{proof}
In the current parameterization, we recall that $c_i^2$ and $(c_i')^2$ are mutually independent Beta random variables with parameters
\begin{align*}
c_i^2 &\sim \operatorname{Beta}(\tfrac{nb}{\alpha a}+\alpha^{-1}(i-n), \tfrac{n(1-b)}{\alpha a}+\alpha^{-1}(i-n)), \text{ and } \\
(c_i')^2 &\sim \operatorname{Beta}(\alpha^{-1}i, \tfrac{n}{\alpha a}+\alpha^{-1}(i-2n+1)).
\end{align*}
The primary tool in this proof is the Poincar\'e inequality for Beta random variables.  From Lemma~\ref{beta_poincare}, a Beta variable $X \sim \operatorname{Beta}(p_1,p_2)$ satisfies a Poincar\'e inequality
\[
\Var f(X) \leq \frac{1}{4(p_1+p_2)} \expect \left| f'(X)\right|^2,
\]
for any Lipschitz function $f$ on $[0,1].$  Let $\mathcal{M}$ denote the collection of all Beta variables appearing in the matrix model.  We note that for all these variables, the sum of their parameters is at least
\(
\frac{n}{\alpha}\left[ \frac{1}{a} - 2 \right]. 
\)
By hypothesis on the parameters of the matrix, $a < 1/2,$ and thus there is a constant $C$ so that
\[
\max_{X \in \mathcal{M}} \sup_{\|f\|_{Lip}<\infty} \left[\frac{\Var f(X)}{\expect \left| f'(X)\right|^2}\right] \leq \frac{C}{n}.
\]
Further, by applying each of these inequalities to $q(X)$ for any $X \in \mathcal{M},$ we see that for any Lipschitz $f,$ 
\[
\Var f(q(X)) \leq \frac{C}{n} \expect \left| f'(q(X)) q'(X)\right|^2.
\]
Note that $|q'(x)| \leq \left(a_1 + a_2\right)$ on $[0,1],$ and thus
\[
\Var f(q(X)) \leq \frac{C(a_1+a_2)^2}{n} \expect \left| f'(q(X))\right|^2,
\]
for all Lipschitz functions on the interval and any $X \in \mathcal{M}.$  It is well known that a Poincar\'e inequality implies exponential integrability (see~\cite{BorovkovUtev}).  Precisely,
\[
\expect \exp \left[ \frac{\left| g(X) - \expect g(X)\right|\sqrt{n}}{12(a_1+a_2)\sqrt{C}}\right] \leq 2,
\]
for every $X \in \mathcal{M}.$  By expanding the exponential in its series, the claim follows.
\end{proof}

As a consequence of Lemma~\ref{beta_moment_order}, it is possible to estimate the contribution of any product of terms as in \eqref{reduction2}.  

\begin{lemma}
\label{trace_path_order}
There is a constant $C=C(l, \max_{1\leq i \leq l} m_i)$ so that for any $l$-tuple $\left(\bar w_1, \ldots, \bar w_l\right) \in \tps{m_1} \times \tps{m_2} \times \cdots \times \tps{m_l},$ 
\[
\left|
\expect
\prod\nolimits_{i=1}^l\left[
A_{\bar w_i} - \expect A_{\bar w_i} \right]
\right|
\leq Cn^{-l/2}.
\]
Furthermore, the dominant contribution is given by
\[
D_{(\bar w_i)_i}  \Def
\sum_{s_1 \ldots s_l}
\prod_{i=1}^l \left[ ( \tpoly{i}{s_i} - \expect \tpoly{i}{s_i} )
\prod_{j \not{=} s_i}\expect \tpoly{i}{j}\right],
\]
with the sum over all $l$-tuples $(s_1,\ldots, s_l) \in [l]^{2n-1},$ and
\[
\left|
\expect
\prod_{i=1}^l\left[
A_{\bar w_i} - \expect A_{\bar w_i} \right]
-\expect D_{(\bar w_i)_i}
\right| \leq Cn^{-(l+1)/2}.
\]
\end{lemma}
\begin{proof}
We recall \eqref{reduction2}:
\begin{equation*}
\prod_{i=1}^l\left[
A_{\bar w_i} - \expect A_{\bar w_i} \right]
=
\sum_{S_1 \ldots S_l}
\prod_{i=1}^l 
\prod_{j \in S_i}( \tpoly{i}{j} - \expect \tpoly{i}{j} )
\prod_{j \not\in S_i}\expect \tpoly{i}{j},
\end{equation*}
where the sum is over nonempty subsets $S_i \subset [2n-1].$  Taking expectations, most of these of summands will be $0.$  This is because for each word $\bar w_i,$ there are at most $4m_i$ nontrivial polynomials $\tpoly{i}{j}$, where $\bar w_i \in \tps{m_i}.$  Thus, there are at most $2^{4m_1}2^{4m_2}\cdots 2^{4m_l}$ nonzero summands of the form
\begin{equation}
\label{termand}
P_{S_1,\ldots,S_{l}} \Def
\expect
\prod_{i=1}^l 
\prod_{j \in S_i}( \tpoly{i}{j} - \expect \tpoly{i}{j} )
\prod_{j \not\in S_i}\expect \tpoly{i}{j},
\end{equation}
and thus it suffices to show the desired bound for an arbitrary term such as this.  From each $S_i,$ pick an arbitrary $j_i.$  Each $\tpoly{i}{j}$ is a random variable supported on $[0,1],$ and thus both $|\tpoly{i}{j} - \expect \tpoly{i}{j}| \leq 1$ and $|\expect \tpoly{i}{j}| \leq 1.$  Therefore, the term in~\eqref{termand} can be bounded by
\[
\left|P_{S_1,\ldots,S_{l}} \right| \leq \expect \left|\prod_{i=1}^l
( \tpoly{i}{j_i} - \expect \tpoly{i}{j_i} )\right|
\leq \frac{1}{l}\sum_{i=1}^l \expect \left|\tpoly{i}{j_i} - \expect \tpoly{i}{j_i}\right|^l, 
\]
where we have applied the arithmetic-geometric mean inequality.  By applying Lemma~\ref{beta_moment_order}, we conclude that there is a constant $C$ that depends only on $\max_{1\leq i \leq l} m_i$ and $l$ so that
\[
\left|P_{S_1,\ldots,S_{l}} \right| \leq C n^{-l/2}.
\]
Summing over all possible nonzero summands, the first conclusion follows.  Note that the same argument shows that if $\sigma \Def |S_1| + |S_2| + \cdots + |S_l| > l,$ then the same argument (with the same constant no less) shows
\[
\left|P_{S_1,\ldots,S_{l}} \right| \leq C n^{-\sigma/2},
\]
from which the second conclusion follows.
\end{proof}

Having established these bounds, we introduce the notion of a dependency graph.  
\begin{definition}
For any tuple of tridiagonal trace paths $\left( \bar w_1, \bar w_2, \ldots, \bar w_l\right),$ define the \emph{dependency graph} $\mathcal{G}$ to be a graph with vertex set $[l]$ and $i \not\leftrightarrow j$ if and only if $A_{\bar w_i}$ and $A_{\bar w_j}$ are functions of mutually independent random variables. 
\end{definition}
\noindent The family of vector variables
\[
\Xi \Def \left\{\left(A_{\bar w_i} \right)_{i \in S}\right\}_S,
\]
where $S$ ranges over all connected components of $\mathcal{G},$ is a mutually independent family of random variables.  The importance of these connected components is that there are very few $l$-tuples of tridiagonal trace paths that have few connected components in their dependency graph.  Moreover, it is possible to estimate exactly how many trace paths have such dependency graphs.  This motivates the following definition.

\begin{definition} For any $\chi \in \{1,2,\ldots, \lfloor l/2 \rfloor\},$ let $\mathcal{B}_\chi$ be the collection of all $l$-tuples in $\tps{m_1} \times \tps{m_2} \times \cdots \times \tps{m_l}$ whose dependency graphs have $\chi$ connected components and no isolated vertices.  For any such word tuple of words, let $\mathcal{E} = \mathcal{E}(\bar w_1, \ldots, \bar w_l)$ denote the edge set of the dependency graph.
\end{definition}

When $l$ is even, $\mathcal{B}_{l/2}$ is the collection of all $l$-tuples of trace paths whose dependency graphs are perfect matchings.  With this definition, we can count the number of $l$-tuples of trace paths having a particular number of connected components.
 
\begin{lemma}
\label{dgraph_count}
For any $\chi \in \mathbb{N},$ there is a constant $C=C(\chi, \max_{1\leq i \leq l} m_i)$ so that $\left| \mathcal{B}_\chi \right| \leq C n^{\chi}.$
\end{lemma}
\begin{proof}
This ultimately stems from the observation that there are only finitely many entries in the matrix that depend on a given entry.  Thus, once any arbitrary trace path in a connected component has been chosen, the remainder of the trace paths must start nearby.  Formally, we begin by bounding the number of ways to construct a connected component on $s$ vertices.

Without loss of generality, suppose these $s$-tuples are chosen from $\tps{m_1} \times \tps{m_2} \times \cdots \times \tps{m_s}.$  As we would like choices having a connected dependency graph, we overcount by first choosing a desired spanning tree and then filling out the graph.  As there are only $s^{s-2}$ such spanning trees, we lose at most a constant factor. 

Let $M = \max_{1\leq i \leq l} m_i,$ and choose the first trace path in the tuple arbitrarily; there are $\left|\tps{m_1}\right|$ possible choices for this path.  Traversing the vertices of the tree in a depth first search, each vertex traversed must depend on the previously chosen path $\bar w_{prev} \in \tps{m_{prev}}.$  This forces the choice of $\bar w_{new} \in \tps{m_{new}}$ to have that $A_{\bar w_{new}}$ depends on $A_{\bar w_{prev}},$ and thus the starting point of $\bar w_{new}$ must be no more than $m_{new} + m_{prev}$ steps from the starting point of the previous.  Thus there are at most $4M \left|\mathcal{A}_{2M}\right|$ ways to choose the new path.  This bound holds for every vertex explored in the depth first search, and we arrive at the bound that there are at most $\left[4M \left|\mathcal{A}_{2M}\right|\right]^s\cdot n$ ways to choose trace paths having dependency graph spanned by a given tree.

Summing over all possible partitions of $l$ with $\chi$ parts, i.e. all multisets of naturals $\{s_i\}$ so that $s_1 + s_2 + \cdots + s_\chi = l,$ and choosing components of these sizes for each, we arrive at the bound that there is a constant $C$ so that $\left| \mathcal{B}_\chi \right| \leq C n^{\chi}.$

\end{proof}

It is now possible to identify the asymptotically relevant portions of an arbitrary mixed moment, and hence prove Proposition~\ref{mixedmoment_structure}.

\begin{proof}[Proof of Proposition~\ref{mixedmoment_structure}]
In terms of the notation $B_\chi,$ we recall \eqref{reduction0} and rewrite it as
\begin{align}
\expect \prod_{u\in m} \cmt{k}
&=
\sum_{\bar w_1, \ldots, \bar w_{l}} \expect \prod_{i=1}^l \left[ A_{\bar w_i} - \expect A_{\bar w_i}\right] \notag \\
\label{reduction4}
&=
\sum_{\chi=1}^{\lfloor l/2 \rfloor} \sum_{(\bar w_1, \ldots, \bar w_{l}) \in B_\chi} \expect \prod_{i=1}^l \left[ A_{\bar w_i} - \expect A_{\bar w_i}\right],
\end{align} 
noting that this sum contains no $l$-tuples of words with isolated vertices in their dependency graphs, as these vanish identically on taking expectations.  By Lemma~\ref{trace_path_order}, there is a constant $C_1$ sufficiently large that 
\[
\left|\expect \prod_{i=1}^l \left[ A_{\bar w_i} - \expect A_{\bar w_i}\right]\right| \leq C_1n^{-l/2},
\]
for every word in the sum.  Also, by Lemma~\ref{dgraph_count} there is a constant $C_2$ sufficiently large that for all $1 \leq \chi \leq l/2,$ $|B_\chi| \leq C_2 n^{\chi}.$  It is immediate that if $l$ is odd, then by \eqref{reduction4},
\[
\left|\expect \prod_{u\in m} \cmt{k}\right|
\leq \sum_{\chi=1}^{\lfloor l/2 \rfloor} C_1n^{-l/2} C_2 n^{\chi} = O(n^{-1/2}).
\]

If $l$ is even, however, then applying the same bound to terms for which $\chi < l/2,$
\begin{align}
\expect \prod_{u\in m} \cmt{k}
&= \sum_{(\bar w_1, \ldots, \bar w_{l}) \in B_{l/2}} \expect \prod_{i=1}^l \left[ A_{\bar w_i} - \expect A_{\bar w_i}\right] + O(n^{-1/2}) \notag \\
\label{reduction5}
&= \sum_{(\bar w_{i})_i \in B_{l/2}} ~~\prod_{\{a,b\} \in \mathcal{E} }\expect  \left[ A_{\bar w_a} - \expect A_{\bar w_a}\right]  \left[ A_{\bar w_b} - \expect A_{\bar w_b}\right] + O(n^{-1/2}).
\end{align}
It only remains to show that the Wick word has the same form, i.e. it should be shown that
\begin{equation}
\label{wickword0}
W \Def \sum_{G} \prod_{\{a,b\} \in \mathcal{E}(G)} \expect \left[\cmt{m_{a}} \cmt{m_{b}}\right],
\end{equation}
where $G$ ranges over all perfect matchings of $[l]$, has the same asymptotically relevant terms as~\eqref{reduction5}.  We recall \eqref{reduction0}, due to which we may rewrite
\[
W = \sum_{G} ~~~ \prod_{\{a,b\} \in \mathcal{E}(G)}~~~ \sum_{ \tps{m_a} \times \tps{m_b} } \expect \left[ A_{\bar w_a} - \expect A_{\bar w_a}\right]  \left[ A_{\bar w_b} - \expect A_{\bar w_b}\right],
\]
where the inner sum may be taken over all pairs of $l$-tuples.  For a fixed perfect matching $G,$ every possible tuple $(\bar w_1, \ldots, \bar w_l)$ is represented exactly once. After commuting the inner sum and the product, we may write
\[
W = \sum_{(\bar w_1, \ldots, \bar w_l)} \sum_{G} \prod_{\{a,b\} \in \mathcal{E}(G)} \expect \left[ A_{\bar w_a} - \expect A_{\bar w_a}\right]  \left[ A_{\bar w_b} - \expect A_{\bar w_b}\right].
\]
As before, we may ignore $l$-tuples whose dependency graphs have an isolated vertex, and thus we write
\[
W = \sum_{\chi=1}^{l/2} \sum_{(\bar w_i)_i \in B_\chi} \sum_{G} \prod_{\{a,b\} \in \mathcal{E}(G)} \expect \left[ A_{\bar w_a} - \expect A_{\bar w_a}\right]  \left[ A_{\bar w_b} - \expect A_{\bar w_b}\right].
\]
We will bound the contribution of terms having $\chi < l/2,$ and we note that there is a constant $C_3$ so that for any pairing $G$ and any tuple of paths $(\bar w_i)_i,$
\[
\left|\prod_{\{a,b\} \in \mathcal{E}(G)} \expect \left[ A_{\bar w_a} - \expect A_{\bar w_a}\right]  \left[ A_{\bar w_b} - \expect A_{\bar w_b}\right] \right| \leq C_3 n^{-l/2},
\] 
which follows from applying Lemma~\ref{trace_path_order}.  Writing $C_4 = (2l)!/2^l/l!$ for the number of perfect matchings on $[l],$ we have
\[
\sum_{\chi=1}^{l/2-1} \sum_{(\bar w_i)_i \in B_\chi} \sum_{G} \prod_{\{a,b\} \in \mathcal{E}(G)} \left|\expect \left[ A_{\bar w_a} - \expect A_{\bar w_a}\right]  \left[ A_{\bar w_b} - \expect A_{\bar w_b}\right]\right| \leq \sum_{\chi=1}^{l/2-1} C_2 n^\chi \cdot C_4 \cdot C_3 n^{-l/2} = O(n^{-1/2}).
\]
For each tuple of words $(\bar w_i)_i \in B_{l/2},$ there is exactly one choice of pairing $G$ so that so that the product is nonzero, and thus
\[
W = \sum_{(\bar w_i)_i \in B_{l/2}} \prod_{\{a,b\} \in \mathcal{E}} \expect \left[ A_{\bar w_a} - \expect A_{\bar w_a}\right]  \left[ A_{\bar w_b} - \expect A_{\bar w_b}\right] + O(n^{-1/2}),
\]
which completes the proof on comparison with~\eqref{reduction5}.
\end{proof}

\subsection{Computing the Covariance}
We now turn to showing that all possible the pairwise covariances $\Cov(\cmt{k},\cmt{l})$ have limits and produce an expression for that limiting covariance.  We will use $C_{k,l}$ to denote the covariance we eventually show to be the limit.  These covariances can be described in terms of the polynomials $p_k(x,y)$ introduced in Section~\ref{sec:path_counting}.  The exact form of the covariance is given by an integral against a parameter $\sigma.$  In terms of $\sigma,$ define the expressions
\begin{equation}
\label{xydef0}
x = \Def  \frac{\sqrt{(b+ \sigma)(1-a + \sigma)}}{ 1+2\sigma} ~, ~~\mbox{and}~~
y = \Def  \frac{\sqrt{(1-b+\sigma)(a+\sigma)}}{ 1+2\sigma}.
\end{equation}
The matrix $C_{k,l}$ for $k,l \geq 1$ can now be defined by
\begin{equation}
\label{raw_limiting_covariance}
C_{k,l} \Def
\frac{\alpha}{4}\int_{-a}^0\frac{1}{1+2\sigma}\left[\left(\partial_x p_k\partial_x p_m + \partial_y p_k \partial_y p_m\right)(1-x^2-y^2) -\left(\partial_x p_k\partial_y p_m + \partial_y p_k \partial_x p_m\right)(2xy) \right]d\sigma.
\end{equation}
\begin{remark}
\label{covariancesigns}
In this form, the integrand is separated into positive and negative
parts.  We can check that $x^2 + y^2 < 1$ for all $-a \leq \sigma \leq 0.$   Furthermore, because $p_k$ have all positive coefficients, $x$ is nonnegative, and $y$ is nonnegative, it follows that
\begin{align*}
\left(\partial_x p_k\partial_x p_m + \partial_y p_k \partial_y p_m\right)(1-x^2-y^2) &\geq 0~, \text{ and }\\
\left(\partial_x p_k\partial_y p_m + \partial_y p_k \partial_x p_m\right)(2xy) &\geq 0,
\end{align*}
for all $-a \leq \sigma \leq 0.$
To check that $x^2 + y^2 < 1,$ we clear the denominator and expand the terms to show that this is equivalent to 
\[
b(1-a) + (1-b)a < 1 + 2\sigma + 2\sigma^2.
\]
The quadratic on the right is increasing for $-1/2 < \sigma < 0,$ and thus to show the inequality, it suffices to show that 
\[
b(1-a) + (1-b)a = a + b(1-2a) < 1 - 2a + 2a^2.
\]
Using that $1-2a > 0$ and $b < 1-a,$ the inequality follows.
\end{remark}
Our primary purpose in this section is to prove the following Proposition.
\begin{proposition}
\label{covariance_limit}
For each fixed $k,l \in \mathbb{N},$ as $n\tendsto \infty,$
\[
\Cov(\cmt{k},\cmt{l}) = \expect \left[\cmt{k} \cmt{l}\right] = C_{k,l} + O(n^{-1/2}).
\]
\end{proposition}
Note that combining this Proposition with Proposition~\ref{mixedmoment_structure}, we have proven Proposition~\ref{polynomial_clt}.  We turn immediately towards proving Proposition~\ref{covariance_limit}.  We recall that by \eqref{reduction0}, we have
\[
\expect \left[ \cmt{k} \cmt{l} \right]=
\sum_{\bar w_k, \bar w_{l}} \expect \left[ A_{\bar w_k} - \expect A_{\bar w_k}\right]\left[ A_{\bar w_l} - \expect A_{\bar w_l}\right].
\]
By Lemma~\ref{dgraph_count}, there is a constant $K_\chi$ so that there are at most $K_\chi \cdot n$ such words.  Applying the second portion of Lemma~\ref{trace_path_order}, we have that there is a constant $K_{k \vee l}$ so that
\[
\left|\expect \left[\cmt{k} \cmt{l} \right]
- \sum\nolimits_{\bar w_k,\bar w_l} \expect D_{(\bar w_k, \bar w_l)}
\right| \leq K_\chi n \cdot K_{k \vee l}\cdot n^{-3/2},
\] 
where we recall that $D_{(\bar w_k, \bar w_l)}$ is given by
\[
D_{(\bar w_k, \bar w_l)} =
\sum_{s_k, s_l}~
\prod_{i \in \{k,l\}}~ \left[ ( \tpoly{i}{s_i} - \expect \tpoly{i}{s_i} )
\prod\nolimits_{j \not{=} s_i}\expect \tpoly{i}{j}\right],
\]
with the sum over all choices of $s_k, s_l \in [2n-1].$  Thus, it suffices to analyze the quantity $\sum\nolimits_{\bar w_k,\bar w_l} \expect D_{(\bar w_k, \bar w_l)}$ and show it has the desired limit.  Note that by the construction of $\tpoly{i}{s_i},$ each of $\tpoly{k}{s_k}$ and $\tpoly{l}{s_l}$ are independent if $s_k \neq s_l,$ and thus we have
\[
\expect \left[ \cmt{k} \cmt{l} \right]
=
\sum\nolimits_{\bar w_k,\bar w_l} 
\sum_{t=1}^{2n-1}~
\expect
\left[ \tpoly{k}{t} - \expect \tpoly{k}{t} \right]
\left[ \tpoly{l}{t} - \expect \tpoly{l}{t} \right]
\left[\prod\nolimits_{j \not{=} t}\expect \tpoly{k}{j}\expect \tpoly{l}{j}\right]
+O(n^{-1/2}).
\] 
We define $r_t$ so that
\begin{equation}
\label{rt_def}
r_t \Def 
\sum\nolimits_{\bar w_k,\bar w_l} 
\expect
\left[ \tpoly{k}{t} - \expect \tpoly{k}{t} \right]
\left[ \tpoly{l}{t} - \expect \tpoly{l}{t} \right]
\left[\prod\nolimits_{j \not{=} t}\expect \tpoly{k}{j}\expect \tpoly{l}{j}\right],
\end{equation}
and note that by commuting sums in the previous equation, we have
\begin{equation}
\label{cmt_to_rt}
\expect \left[ \cmt{k} \cmt{l} \right]
=
\sum_{t=1}^{2n-1}~
r_t + O(n^{-1/2}).
\end{equation}
Let $\{z_i\}_{i=1}^{2n-1}$ be the enumeration of all the Beta variables in $\mathcal{M},$ where $z_i = (c_i)^2$  for $1 \leq i \leq n$ and $z_i = (c_{i-n}')^2$ when $n+1 \leq i \leq 2n-1.$  This makes each $\tpoly{i}{j}$ a polynomial in $z_j.$  The first step in the analysis amounts to using Taylor approximation to pull the expectations inside the $\tpoly{i}{j}$ polynomials.
\begin{lemma}
\label{rt_order}
There is a constant $K = K(k,l)$ so that for all $1 \leq t \leq 2n-1,$
\[
|r_t| \leq Kn^{-1}.
\]  
Moreover, it is possible to identify the dominant contribution $r_t^D$, which is given by
\[
r_t^D := 
\Var(z_t)
\sum\nolimits_{\bar w_k,\bar w_l} 
\left[ \tpoly{k}{t}'( \expect z_t) \right]
\left[ \tpoly{l}{t}'( \expect z_t) \right]
\left[\prod\nolimits_{j \not{=} t}\tpoly{k}{j}(\expect z_j) \tpoly{l}{j}(\expect z_j)\right],
\]
and which has
\[
\left| 
r_t
- r_t^D
\right| \leq Kn^{-3/2}.
\]
\end{lemma}
\begin{proof}
The first claim follows from Lemma~\ref{beta_moment_order} and from the fact that the number of trace paths that depend on $z_t$ is bounded by some $K=K(k,l).$  The second claim will follow from Taylor approximation.  For any polynomial $\tpoly{k}{j}$ or $\tpoly{l}{j},$ it is possible to bound the maximums of the derivatives over $[0,1]$ in terms of $k$ and $l.$  Each polynomial has the form $\tpoly{i}{j}(x) = x^{a^1_{i,j}}(1-x)^{a^2_{i,j}},$ and hence its first and second derivatives can be bounded by $a^1_{i,j}+a^2_{i,j}$ and $(a^1_{i,j}+a^2_{i,j})^2.$  These parameters $a^1$ and $a^2$ can in turn be bounded by $i,$ to yield
\[
\max_{x \in [0,1]} \left| \tpoly{i}{j}'\right| \leq 4i~~\text{and}~~
\max_{x \in [0,1]} \left| \tpoly{i}{j}''\right| \leq (4i)^2,
\]
for either $i \in \{k,l\}.$ These imply that the $0^{th}$ order approximation has error
\[
\left|\tpoly{i}{j}(z_j) - \tpoly{i}{j}( \expect z_j)\right| \leq 4i \left|z_j - \expect z_j\right|,
\]
and the $1^{st}$ order approximation has error
\[
\left|\tpoly{i}{j}(z_j) - \tpoly{i}{j}( \expect z_j) -\tpoly{i}{j}'(\expect z_j) (z_j - \expect z_j)\right| \leq 8i^2 \left|z_j - \expect z_j\right|^2.
\]
We recall the definition of $r_t,$ which was given by
\[
r_t = 
\sum\nolimits_{\bar w_k,\bar w_l} 
\underbrace{\mystrut{2.5ex}
\expect
\left[ \tpoly{k}{t} - \expect \tpoly{k}{t} \right]
\left[ \tpoly{l}{t} - \expect \tpoly{l}{t} \right]}_{\text{(i)}}
\underbrace{\mystrut{2.5ex}
\left[\prod\nolimits_{j \not{=} t}\expect \tpoly{k}{j}\expect \tpoly{l}{j}\right]}_{\text{(ii)}}.
\]
Using $1^{st}$ order approximation for term $(i),$ we bound 
\begin{equation}
\label{rt_ibound}
D_{(i)} \Def
\left|\expect
\left[ \tpoly{k}{t} - \expect \tpoly{k}{t} \right]
\left[ \tpoly{l}{t} - \expect \tpoly{l}{t} \right]
-
\expect \left[\tpoly{k}{t}'(\expect z_t)\tpoly{l}{t}'(\expect z_t) (z_t - \expect z_t)^2\right]
\right| \leq K_1n^{-3/2},
\end{equation}
with the constant implicitly depending on $k,$ $l,$ and the constants assured by Lemma~\ref{beta_moment_order}.
Using the $0^{th}$ order approximation for term $(ii),$ we will bound the difference between $(ii)$ and its approximation.  This will be done by replacing each $z_j$ by $\expect z_j$ one term at a time.  As there are at most $2k + 2l$ non-constant polynomials $\tpoly{k}{j}$ and $\tpoly{l}{j}$, this reduces bounding $(ii)$ to bounding, for any fixed $u,$
\[
\Delta_u \Def \left[\tpoly{k}{u}(\expect z_u)\tpoly{l}{u}(\expect z_u)-\expect\tpoly{k}{u}(z_u)\expect \tpoly{l}{u}(z_u)\right]
\prod_{\substack{j < u \\ j \neq t}}\tpoly{k}{j}(\expect z_j)\tpoly{l}{j}(\expect z_j)
\prod_{\substack{j > u \\ j \neq t}}\expect \tpoly{k}{j}(z_j)\expect \tpoly{l}{j}(z_j).
\]
Recalling that all $\tpoly{i}{j}$ are almost surely less than $1,$ this can be bounded by
\[
\left| 
\Delta_u \right| \leq  \left|\tpoly{k}{u}(\expect z_u)\tpoly{l}{u}(\expect z_u)-\expect\tpoly{k}{u}(z_u)\expect \tpoly{l}{u}(z_u)\right| \leq (4k+4l)\expect\left| z_u -\expect z_u\right| \leq K_2n^{-1/2}.
\]
These bounds applied to the difference of $(ii)$ and its approximation show
\begin{equation}
\label{rt_iibound}
D_{(ii)} \Def
\left|
\prod\nolimits_{j \not{=} t}\expect \tpoly{k}{j}(z_j)\expect \tpoly{l}{j}(z_j)
-\prod\nolimits_{j \not{=} t}\tpoly{k}{j}(\expect z_j)\tpoly{l}{j}(\expect z_j)
\right|
\leq \sum_{ \substack{1 \leq u \leq 2n-1 \\ \tpoly{k}{u}\tpoly{l}{u} \neq 1 } } \left|\Delta_u\right| \leq (2k+2l) \cdot K_2 n^{-1/2}.
\end{equation}
By combining Lemma~\ref{beta_moment_order} with Cauchy-Schwarz, one has that $(i)$ is at most $K_3 n^{-1}.$ 
Therefore, we can combine both of~\eqref{rt_ibound} and~\eqref{rt_iibound} to show
\begin{align*}
\left| r_t - r_t^D \right| 
&\leq 
\sum\nolimits_{\bar w_k,\bar w_l} 
\left|(i)\right|
\left|D_{(ii)}\right|
+
\left|D_{(i)}\right|
\left|
\left[\prod\nolimits_{j \not{=} t}\tpoly{k}{j}(\expect z_j) \tpoly{l}{j}(\expect z_j)\right]
\right| \\
&\leq 
\sum\nolimits_{\bar w_k,\bar w_l} 
K_3 n^{-1} \cdot (2k+2l) \cdot K_2 n^{-1/2}
+
K_1 n^{-3/2} \cdot 1.
\end{align*}
As the sum is only over paths that depend upon $t,$ the proof is complete.
\end{proof}
All the expectations in $r_t^D$ are approximately equal to one of two values, $\expect z_t$ and $\expect z_{t+n}$ (or $t-n$ in the case $t > n$), on account of the trace paths being forced to overlap.  Thus, this can be expressed in terms of the polynomials $p_k(x,y)$ for values of $t$ for which the trace paths are sufficiently far from the matrix edge.  The values of $x$ and $y$ are given in terms of the expectations of matrix entries.  Put $s(t) = t$ if $ 1 \leq t \leq n$, and put $s(t) = t - n$ if $n +1 \leq t \leq 2n -1.$  The values of $x$ and $y$ are given by
\begin{equation}
\label{xydef}
x(t) \Def \sqrt{\expect \left[c_s^2(1-(c_s')^2)\right]}~~~\text{and}~~~
y(t) \Def \sqrt{\expect \left[(c_s')^2(1-(c_s)^2)\right]}.
\end{equation}
Note that these are not exactly the expressions for $x$ and $y$ given in~\eqref{xydef0}, but we will show that these two quantities are strongly related.  In what follows, we unequivocally mean the $x$ and $y$ given in \eqref{xydef}.

\begin{lemma}
\label{rtd_limit}
Define $\xi_t^D$ to be
\[
\xi_t \Def r_t^D -  
\frac{1}{4}
\begin{cases}
\Var( c_s^2)\left[
\left(\frac{x\partial_xp_k(x,y)}{\expect c_s^2} - \frac{y\partial_yp_k(x,y)}{1-\expect c_s^2})\right)
\left(\frac{x\partial_xp_m(x,y)}{\expect c_s^2} - \frac{y\partial_yp_m(x,y)}{1-\expect c_s^2})\right)
\right] & 1 \leq t \leq n \\
\Var( c_{s}'^2)\left[
\left(\frac{y\partial_yp_k(x,y)}{\expect c_{s}'^2} - \frac{x\partial_xp_k(x,y)}{1-\expect (c_s')^2})\right)
\left(\frac{y\partial_yp_m(x,y)}{\expect c_{s}'^2} - \frac{x\partial_xp_m(x,y)}{1-\expect (c_s')^2})\right)
\right] & n+1 \leq t \leq 2n-1. 
\end{cases}
\]
There is a constant $K=K(k,l)$ so that for all $k+l \leq t \leq n-k-l$ and $n+k+l \leq t \leq 2n-k-l-1,$ $\left| \xi_t^D \right| \leq Kn^{-2}.$
\end{lemma}
\begin{proof}
We show the proof for $1 \leq t \leq n.$  The proof for $t > n$ is identical.  We recall that $r_t^D$ is given by
\[
r_t^D = 
\Var(z_t)
\sum\nolimits_{\bar w_k,\bar w_l} 
\left[ \tpoly{k}{t}'( \expect z_t) \right]
\left[ \tpoly{l}{t}'( \expect z_t) \right]
\left[\prod\nolimits_{j \not{=} t}\tpoly{k}{j}(\expect z_j) \tpoly{l}{j}(\expect z_j)\right].
\] 
This splits nicely as $r_t^D = \Var(z_t) M_t(\bar w_k) M_t(\bar w_l),$ where we define
\[
M_t(\bar w_i) \Def
\sum\nolimits_{\bar w_i} 
\left[ \tpoly{i}{t}'( \expect z_t) \right]
\left[\prod\nolimits_{j \not{=} t}\tpoly{i}{j}(\expect z_j)\right].
\]
This $M_t(\bar w_i)$ is essentially computable from just two expectations, $\expect z_t$ and $\expect z_{t+n}.$ Letting
\[
M_t(\bar w_i)^D \Def
\sum\nolimits_{\bar w_i} 
\left[ \tpoly{i}{t}'( \expect z_t) \right]
\left[\prod\nolimits_{\substack{j \not{=} t \\ 1\leq j \leq n}}\tpoly{i}{j}(\expect z_t)\tpoly{i}{j+n}(\expect z_{t+n})\right],
\]  
we show that $\left|M_t(\bar w_i) - M_t^D(\bar w_i)\right|$ is $O(n^{-1}).$
We will require the formulae for $\expect z_t = \expect c_t^2$ and $\expect z_{t+n}= \expect (c_{t}')^2,$ and so we recall the precise distributions of these entries,
\[
c_i^2 \sim \operatorname{Beta}(\tfrac{nb}{\alpha a}+\alpha^{-1}(i-n), \tfrac{n(1-b)}{\alpha a}+\alpha^{-1}(i-n)) ~~\text{and}~~
(c_i')^2 \sim \operatorname{Beta}(\alpha^{-1}i, \tfrac{n}{\alpha a}+\alpha^{-1}(i-2n+1)).
\]
Their expectations are given by
\begin{align}
\label{beta_expectations}
\expect z_t  = \expect c_t^2&= \frac{\tfrac{nb}{\alpha a}+\alpha^{-1}(t-n)}{\tfrac{n}{\alpha a}+\alpha^{-1}(2t-n)} = \frac{b - a + 2a\tfrac{t}{n}}{1 - 2a + 2a\tfrac{t}{n}} \\
\expect z_{t+n} = \expect (c_{t}')^2&= \frac{\alpha^{-1}t}{\tfrac{n}{\alpha a}+\alpha^{-1}(2t-n)} = \frac{a\tfrac{t}{n}}{1 - 2a + 2a\tfrac{t}{n}}. 
\end{align}
Each of these expectations, as a function of $t,$ is uniformly Lipschitz continuous over $0 \leq t \leq n$ with constant $K_1 \cdot n^{-1}$ for some $K_1$ depending only on the ensemble parameters.  By the same method used in the proof of Lemma~\ref{rt_order}, it is straightforward to show that there is a constant $K_2=K_2(k,l)$ so that  
\[
\left| 
M_t(\bar w_i) - M_t^D(\bar w_i)\right| \leq K_2 n^{-1}.
\]
We recall the notation of Lemma~\ref{even_many_steps}, where we defined $\phsteps{i}{t}$ to be the number of horizontal steps of $\bar w_i$ from level $i$ to $i$ and $\pvsteps{i}{t}$ to be the number of steps of $\bar w_i$ from level $i$ to $i+1$ or vice versa.  The polynomial $\tpoly{i}{t}$ may be identified precisely in terms of these counts.  Recalling the matrix model \eqref{tridiagonal_model}, the variables $c_t$ and $s_t$ appear only in the $t^{th}$ row from the bottom of the matrix.  It follows that
\[
\tpoly{i}{t}(c_t^2) = c_t^{\phsteps{i}{t}}{(1-c_t^2)}^{\pvsteps{i}{t}/2},
\]
and thus, differentiating,
\begin{align}
\tpoly{i}{t}(z)' &= 
\frac{\phsteps{i}{t}z^{{\phsteps{i}{t}}/{2}-1}(1-z)^{{\pvsteps{i}{t}}/{2}}}{2}
-
\frac{\pvsteps{i}{t}z^{{\phsteps{i}{t}}/{2}}(1-z)^{{{\pvsteps{i}{t}}/{2}-1}}}{2} \notag\\
\label{tpoly_derivative}
&=
\frac{\phsteps{i}{t}\tpoly{i}{t}(z)}{2z}-\frac{\pvsteps{i}{t}\tpoly{i}{t}(z)}{2(1-z)}. 
\end{align}
We now relate $M_t^D(\bar w_i)$ to expressions containing $p_i(x,y).$  The essential realization is that
\begin{equation}
\label{xpart}
\sum\nolimits_{\bar w_i}
\phsteps{i}{t}
\left[\prod\nolimits_{{1\leq j \leq n}}\tpoly{i}{j}(\expect z_t)\tpoly{i}{j+n}(\expect z_{t+n})\right]
= \sum\nolimits_{\bar w \in \mathcal{A}_{2i}} h(\bar w)x^{h(\bar w)}y^{2i - h(\bar w)}
= x\partial_x p_i(x,y),
\end{equation} 
where $h(\bar w_i)$ is the number of horizontal steps $\bar w_i$ makes, and $x$ and $y$ are defined earlier.  This is a direct consequence of the bijection between paths $\bar w_1 \in \tps{i}$ that have a single marked horizontal edge at level $t$ and paths $\bar w_2 \in \mathcal{A}_{2i}$ having a single marked horizontal edge.  This is given by the map that simply vertically shifts $\bar w_1$ to start at $0;$ note that this is invertible on account of the mark being forced to lie at level $t$.  For $n-k-l \geq t \geq k+l,$ every summand on the left hand side of~\eqref{xpart} is exactly the summand given on the right when identifying paths via this bijection (note that for $t$ too close to the matrix edge, some of the paths on the left hand side will be $0$, destroying the identity).  Similar reasoning shows
\begin{equation}
\label{ypart}
\sum\nolimits_{\bar w_i}
\pvsteps{i}{t}
\left[\prod\nolimits_{{1\leq j \leq n}}\tpoly{i}{j}(\expect z_t)\tpoly{i}{j+n}(\expect z_{t+n})\right]
=y\partial_y p_i(x,y).
\end{equation}
By combining \eqref{tpoly_derivative},~\eqref{xpart},and~\eqref{ypart}, it follows that
\[
M_t^D(\bar w_i) = \frac{x\partial_x p_i(x,y)}{2\expect z_t} - \frac{y\partial_y p_i(x,y)}{2\expect[1-z_t]}. 
\]
The conclusion of the lemma follows more or less immediately.  By Lemma~\ref{beta_moment_order}, the variance of $z_t$ can be controlled by $K_3 n^{-1},$ with $K_3$ depending only on the matrix parameters.  The moduli of $M_t(\bar w_i)$ and $M_t(\bar w_i)^D$ can be controlled by some $K_4=K_4(k,l),$ and so
\[
\left| \xi_t \right| = \left| \Var(z_t) M_t(\bar w_k) M_t(\bar w_l)
-\Var(z_t) M_t^D(\bar w_k) M_t^D(\bar w_l)\right| \leq K_3 n^{-1} \cdot 2K_4\cdot K_2 n^{-1},
\]
completing the proof.
\end{proof}
On account of the variance being of the order of $n^{-1},$ summing these expressions takes the form of a Riemann sum.  We thus conclude the proof of the limiting covariance formula by showing that this Riemann sum converges to the integral given by $C_{k,l}.$ 
\begin{proof}[Proof of Proposition~\ref{covariance_limit}]
By Lemma~\ref{rt_order} and \eqref{cmt_to_rt}, 
\[
\expect \cmt{k} \cmt{l} = \left[\sum\nolimits_{t=1}^{2n-1}{r_t^D} \right]+ O(n^{-1/2}).
\]
For $1 \leq t \leq n,$ Lemma~\ref{rtd_limit} shows that
\begin{equation}
\label{rtd_rman_sum}
\sum_{t=1}^{n}{r_t^D} = 
\sum_{t=k+l}^{n-k-l} 
\frac{\Var( c_t^2)}{4}\left[
\left(\frac{x\partial_xp_k(x,y)}{\expect c_t^2} - \frac{y\partial_yp_k(x,y)}{1-\expect c_t^2}\right)
\left(\frac{x\partial_xp_m(x,y)}{\expect c_t^2} - \frac{y\partial_yp_m(x,y)}{1-\expect c_t^2}\right)
\right]
+ O(n^{-1/2}).
\end{equation}
We will show that the variance of these Beta variables is of order $n^{-1}.$  To concisely describe the integrand that results in the limit, put $\tau$ to be the variable over which the integral is taken, and define $e(\tau)$ and $e'(\tau)$ as
\begin{equation}
\label{e_def}
e(\tau) \Def \frac{b - a + 2a\tau}{1 - 2a + 2a\tau}~~~\text{and}~~~
e'(\tau) \Def \frac{a\tau}{1 - 2a + 2a\tau}~~~0\leq \tau \leq 1,
\end{equation}
so that for $\tau = t/n,$ $e(\tau) = \expect{c_t^2}$ and $e'(\tau) = \expect(c_t')^2$ (see \eqref{beta_expectations}).  We will reuse the notation $x$ and $y$ by putting
\begin{equation}
\label{xyt_def}
x(\tau) \Def \sqrt{e(\tau)(1-e'(\tau))}~~~\text{and}~~~
y(\tau) \Def \sqrt{e'(\tau)(1-e(\tau))}.
\end{equation}
This definition is now consistent with \eqref{xydef0}, after making a change of variables.  We recall the variances of these Beta variables,
\begin{align}
\label{beta_variances}
\Var c_t^2 &= \frac{\alpha}{n}\frac{\left(\tfrac{b}{a}+\tfrac{t}{n} - 1\right)\left(\tfrac{1-b}{a}+\tfrac{t}{n} - 1\right)}{\left( 
\tfrac{1}{a}+\tfrac{2t}{n} - 2\right)^2 \left(\tfrac{1}{a}+\tfrac{2t}{n} - 2+\tfrac{\alpha}{n}\right)}
= \frac{\alpha a}{n}\frac{\left({b}-a +a\tfrac{t}{n} \right)\left({1-b-a}+a\tfrac{t}{n}\right)}{\left(1-2a+2a\tfrac{t}{n} \right)^3} + O(n^{-2}) \\
\Var (c_t')^2 &= \frac{\alpha}{n}\frac{\left(\tfrac{t}{n}\right)\left(\tfrac{1}{a}+\tfrac{t}{n} - 2\right)}{\left( 
\tfrac{1}{a}+\tfrac{2t}{n} - 2\right)^2 \left(\tfrac{1}{a}+\tfrac{2t}{n} - 2+\tfrac{\alpha}{n}\right)}
= \frac{\alpha a}{n}\frac{\left(a\tfrac{t}{n}\right)\left({1-2a}+a\tfrac{t}{n}\right)}{\left(1-2a+2a\tfrac{t}{n} \right)^3} + O(n^{-2}), \notag
\end{align}
where we may choose the constants in the error terms to depend only on the ensemble parameters (and not $t$).  By virtue of the $\frac{\alpha}{n}$ factor, the sum $\sum_{t=1}^n r_t^D$ takes the form of a Riemann sum.  The integrand, exposed on the right hand side of~\eqref{rtd_rman_sum}, is Lipschitz continuous in $t/n,$ and thus the convergence of the Riemann sum to the integral occurs with rate $O(n^{-1}).$  This shows
\begin{equation}
\label{rawint_a}
\sum_{t=1}^n r_t^D = \frac{\alpha a}{4} \int_{0}^1 
\frac{\left({b}-a +a\tau \right)\left({1-b-a}+a\tau\right)}{\left(1-2a+2a\tau \right)^3}
\left(\frac{x\partial_xp_k}{e(\tau)} - \frac{y\partial_yp_k}{1-e(\tau)}\right)
\left(\frac{x\partial_xp_m}{e(\tau)} - \frac{y\partial_yp_m}{1-e(\tau)}\right)
d\tau + O(n^{-1/2}).
\end{equation}
Applying the same reasoning to $n+1 \leq t \leq 2n-1,$ it follows that
\begin{equation}
\label{rawint_b}
\sum_{t=n+1}^{2n-1} r_t^D = \frac{\alpha a}{4} \int_{0}^1 
\frac{\left(a\tau \right)\left({1-2a}+a\tau\right)}{\left(1-2a+2a\tau \right)^3}
\left(\frac{y\partial_yp_k}{e'(\tau)} - \frac{x\partial_xp_k}{1-e'(\tau)}\right)
\left(\frac{y\partial_yp_m}{e'(\tau)} - \frac{x\partial_xp_m}{1-e'(\tau)}\right)
d\tau + O(n^{-1/2}).
\end{equation}
The sum of these two integrals~\eqref{rawint_a} and \eqref{rawint_b} and the associated error bounds show that the limiting covariance exists, and their sum provides an expression for the limit.  The remainder of the proof will show that this expression can be alternately expressed in the form given by $C_{k,l}$ (defined in \eqref{raw_limiting_covariance}).  The primary difference is a change of variables.  Take $\sigma = a(\tau - 1).$  The integrals become
 \begin{align}
\label{rawsint_a}
\sum_{t=1}^n r_t^D &= \frac{\alpha}{4} \int_{-a}^0 
\frac{e(\sigma)(1-e(\sigma))}{\left(1 + 2\sigma \right)}
\left(\frac{x\partial_xp_k}{e(\sigma)} - \frac{y\partial_yp_k}{1-e(\sigma)}\right)
\left(\frac{x\partial_xp_m}{e(\sigma)} - \frac{y\partial_yp_m}{1-e(\sigma)}\right)
d\sigma + O(n^{-1/2}) \\
\label{rawsint_b}
\sum_{t=n+1}^{2n-1} r_t^D &= \frac{\alpha}{4} \int_{-a}^0 
\frac{e'(\sigma)(1-e'(\sigma))}{\left(1+2\sigma\right)}
\left(\frac{y\partial_yp_k}{e'(\sigma)} - \frac{x\partial_xp_k}{1-e'(\sigma)}\right)
\left(\frac{y\partial_yp_m}{e'(\sigma)} - \frac{x\partial_xp_m}{1-e'(\sigma)}\right)
d\sigma + O(n^{-1/2}).
\end{align}
The sum of these integrals can be shown to equal $C_{k,l}$ by checking the coefficients in front of the terms $\partial_x p_k \partial_x p_m,$ $\partial_y p_k \partial_y p_m,$ $\partial_x p_k \partial_y p_m$ and $\partial_y p_k \partial_x p_m.$  The coefficient on $\partial_x p_k \partial_x p_m$ in the sum of the integrands \eqref{rawsint_a} and \eqref{rawsint_b} is given by
\[
\left[\frac{e(\sigma)(1-e(\sigma))}{\left(1 + 2\sigma \right)} \frac{x^2}{e(\sigma)^2}
+
\frac{e'(\sigma)(1-e'(\sigma))}{\left(1 + 2\sigma \right)} \frac{x^2}{1-e'(\sigma)^2}\right]
=\frac{[1-e(\sigma)][1-e'(\sigma)] + e(\sigma)e'(\sigma)}{1+2\sigma}
=\frac{1-x^2-y^2}{1+2\sigma}.
\]
Similar manipulations show that the coefficients on each of the other terms agree with the coefficients in the integrand of $C_{k,l},$ completing the proof.
\end{proof}

\section{Diagonalizing the Covariance Matrix}
\label{sec:covariance}
We proceed by showing that the covariances are diagonalized by the
appropriate Chebyshev polynomial basis.  This will be done by verifying that certain generating functions agree.  We would like to show that the infinite covariance matrix can be decomposed as
\[
C = L \Lambda L^t,
\]
for the diagonal matrix $\Lambda=\diag(0,1,2,3,4,\ldots)$, and some lower triangular matrix $L.$  The $L_{n,k}$ entry of this matrix is the coefficient of the $k^{th}$ Chebyshev polynomial $\Gamma_k(x)$ in the expansion of $x^n.$ 
Define the exponential covariance generating function $\mathscr{C}(s,t)$ as
\[
\mathscr{C}(s,t) = \sum_{k,l>0} \frac{s^k}{k!}\frac{t^l}{l!}C_{k,l},
\]
and define the exponential generating function of $L \Lambda L^t$ analogously,
\[
\mathscr{T}(s,t) = \sum_{k,l>0} \frac{s^k}{k!}\frac{t^l}{l!}\left[L\Lambda L^t\right]_{k,l}.
\]
We will show that these generating functions are equal by computing
their bivariate Laplace transforms and showing they are the same, from which it follows that $C = L \Lambda L^t.$

\subsubsection{Computing $\Lp_{s,t}[\mathscr{T}]$}

The coefficients $L_{n,k}$ can be computed by a recursive formula, but they have a useful Fourier-like expansion.  Define $\theta$ in terms of $x$ so that
\[
\cos(\theta) = \frac{2x-\lambda_--\lambda_+}{\lambda_+ - \lambda_- },
\]
from which it follows that
\[
2 \cos(n \theta)  = 2 T_n(\cos \theta) = 
2T_n\left( \frac{2x-\lambda_--\lambda_+}{\lambda_+ - \lambda_- }\right).
\]
Expand $\tfrac12e^{tx}$ as a series in $t,$
\[\tfrac12 e^{tx} = \tfrac12\sum_{n=0}^\infty \frac{t^n}{n!} x^n
=\tfrac12\sum_{n=0}^\infty \frac{t^n}{n!} \left(\tfrac{\lambda_++\lambda_-}{2} + \tfrac{\lambda_+-\lambda_-}{2}\cos\theta\right)^n
=\tfrac12\sum_{n=0}^\infty\sum_{k=0}^k L_{n,k} \frac{t^n}{n!} 2\cos{k \theta}
=\sum_{n=0}^\infty\sum_{k=0}^n L_{n,k} \frac{t^n}{n!} \cos{k \theta},
\]
where we have used the definition of $L_{n,k}$ as the coefficient of the $k^{th}$ Chebyshev polynomial in the expansion of $x^n.$

The Fourier interpretation allows for the matrix multiplication $L\Lambda L^t$ to be carried out by an integral.  Consider the kernel $K_N(\theta,\phi),$ which will formally play the role of $\Lambda,$ given by
\[
K_N(\theta,\phi) = \sum_{k=1}^N k\,\cos k\theta\,\cdot\,\cos k\phi.
\]
This allows for $\mathscr{T}$ to be given by
\[
\mathscr{T}(s,t) = \lim_{N \tendsto \infty} \frac{1}{4\pi^2} \int_0^\pi\int_0^\pi e^{t\left(\tfrac{\lambda_+\lambda_-}{2} + \tfrac{\lambda_+-\lambda_-}{2}\cos\theta\right)} K_N(\theta,\phi) e^{s\left(\tfrac{\lambda_++\lambda_-}{2} + \tfrac{\lambda_+-\lambda_-}{2}\cos\theta\right)}d\phi d\theta,
\]
as the coefficient on $t^ks^l$ would be
\[
\lim_{N \tendsto \infty} \frac{1}{\pi^2}\int_0^\pi\int_0^\pi \left(\sum_{j=0}^k L_{k,j} \frac{t^k}{k!} \cos{j \theta}\right) \left(\sum_{k=1}^N k~\cos k\theta ~\cdot~\cos k\phi\right) \left(\sum_{j=0}^l L_{l,j} \frac{s^l}{l!} \cos{j \phi} \right)
d\phi d\theta,
\]
which by the orthogonality of $\{\cos j\theta\}_{j=0}^\infty$ on $[0,\pi],$ is exactly $ (L \Lambda L^t)_{k,l}$ when $N > \min(k,l).$  Further, these integrals can be evaluated, as the expression $e^{z\cos\theta}$ has an expansion in terms of Bessel functions.  Namely,
\[
e^{z\cos\theta} = I_0(z) + 2\sum_{k=1}^\infty I_k(z)\cos k\theta,
\]
(see~\cite[p. 376]{AS}).  This defines the Fourier coefficients of $e^{z \cos \theta},$ from which it follows that $\mathscr{T}(s,t)$ can be rewritten as
\begin{align*}
\mathscr{T}(s,t) &= e^{c(t+s)}\lim_{N\tendsto\infty}\frac{1}{4\pi^2}\int_0^\pi\int_0^\pi e^{rt \cos\theta} K_N(\theta,\phi) e^{rs \cos\phi}d\theta d\phi \\
 &= e^{c(t+s)} \sum_{k=1}^\infty kI_k(rt)I_k(rs),
\end{align*}
where $c=\left(\tfrac{\lambda_++\lambda_-}{2}\right)$ and $r= \left(\tfrac{\lambda_+-\lambda_-}{2}\right).$

Again, we will require the Laplace transform of this generating function.  Each summand $kI_k(rt)I_k(rs)$ is positive for $s,t>0,$ and so commuting the sum and the Laplace transform is justified.
\begin{align}
\Lp_{s,t}[\mathscr{T}(s,t)](\eta,\omega)
&=  \sum_{k=1}^\infty k\Lp_{s,t}\left[e^{c(t+s)}I_k(rt)I_k(rs)\right](\eta,\omega) \notag \\
&=  \sum_{k=1}^\infty k\frac{r^k}{\left(\tilde\omega + \sqrt{\tilde\omega^2 - r^2}\right)^k}\frac{1}{\sqrt{\tilde\omega^2-r^2}}\frac{r^k}{(\tilde\eta + \sqrt{\tilde\eta^2 - r^2})^k}\frac{1}{\sqrt{\tilde\eta^2-r^2}}~, \notag \\
\intertext{where $\tilde\omega=\omega - c.$  This has the form for the series expansion of $\tfrac{x}{(1-x)^2}$. After simplifying, this expression is}
\label{finalform}
\Lp_{s,t}[\mathscr{T}(s,t)](\eta,\omega)
&=\frac{r^2}{{\sqrt{\tilde\eta^2-r^2}}{\sqrt{\tilde\omega^2-r^2}}\left[\sqrt{(\tilde\omega+r)(\tilde\eta-r)} + \sqrt{(\tilde\omega-r)(\tilde\eta+r)}\right]^2}~. 
\end{align}

\subsubsection{Computing $\Lp_{s,t}[\mathscr{C}]$}
We will now turn to computing the Laplace transform of $\mathscr{C}$.  The integrand of $C_{k,l}$ is not positive, but it can be split into two integrals whose integrands are positive (see Remark~\ref{covariancesigns})
\[
C_{k,l} = L_{k,l} - R_{k,l},
\]
with
\begin{align*}
L_{k,l} &= \frac{\alpha}{4}\int_{-a}^0\frac{1}{1+2\sigma}\left[\left(\partial_x p_k\partial_x p_m + \partial_y p_k \partial_y p_m\right)\right]d\sigma, \\
\intertext{and}
R_{k,l} &= \frac{\alpha}{4}\int_{-a}^0\frac{1}{1+2\sigma}\left[\left(\partial_x p_k\partial_x p_m + \partial_y p_k \partial_y p_m\right)(x^2+y^2) +\left(\partial_x p_k\partial_y p_m + \partial_y p_k \partial_x p_m\right)(2xy)\right]d\sigma.
\end{align*}

As $p_l(x,y)$ has all positive coefficients, and both $x$ and $y$ are positive on the domain of integration, each of these integrands is positive.  Defining generating functions for each array,
\[
\mathscr{L}(s,t) = \sum_{k,l>0} \frac{s^k}{k!}\frac{t^l}{l!}L_{k,l}~~~~\text{and}~~~~
\mathscr{R}(s,t) = \sum_{k,l>0} \frac{s^k}{k!}\frac{t^l}{l!}R_{k,l},
\]
we can write
\begin{align}
\mathscr{L}(s,t) &= \frac{\alpha}{4}\int_{-a}^0\frac{1}{1+2\sigma}\left[\left(\partial_x \mathscr{P}(s)\partial_x \mathscr{P}(t) + \partial_y \mathscr{P}(s) \partial_y \mathscr{P}(t)\right)\right]d\sigma,~~~~~\text{and}\notag\\
\label{LandR}
\mathscr{R}(s,t) &= \frac{\alpha}{4}\int_{-a}^0\frac{1}{1+2\sigma}\bigl[\left(\partial_x \mathscr{P}(s)\partial_x \mathscr{P}(t) + \partial_y \mathscr{P}(s) \partial_y \mathscr{P}(t)\right)(x^2+y^2)  \\
&~~~~~~~+\left(\partial_x \mathscr{P}(s)\partial_y \mathscr{P}(t) + \partial_y \mathscr{P}(s) \partial_x \mathscr{P}(t)\right)(2xy)\bigr]d\sigma, \notag
\end{align}
where we have commuted sum and integral by the positivity of the integrands.
Recall that $\mathscr{P}(t) = \mathscr{P}(t,x,y)$ is the exponential generating function for the polynomials $p_k(x,y),$ and from \eqref{pkexpgf}, it is jointly analytic in all variables.  As $-a > -1/2,$ it follows that the integrands are continuous for all $-a \leq \sigma \leq 0,$ and all $s,t.$  In particular, each of $\mathscr{L}$ and $\mathscr{R}$ is finite for all $s,t$, and it follows that we can write $\mathscr{C}(s,t)$ as the sum of these two functions, so
\[
\mathscr{C}(s,t) = \mathscr{L}(s,t) - \mathscr{R}(s,t).
\]
The joint Laplace transforms in $s$ and $t$ will be computed for both of these expressions.  This makes heavy use of Lemma~\ref{derivative_characteristic}.  Additionally, it requires that the order of integration be switched, which requires an argument.  We prove a simplified statement, by whose method it is easily seen that these integrals can be exchanged.
\begin{lemma}
Suppose that $\omega > \lambda_{+}$ and that $\eta > \lambda_{+},$ then
\[
\int_{s,t>0}\int_{-a}^0\frac{e^{-\omega t - \eta s}}{1+2\sigma}\partial_x \mathscr{P}(s)\partial_x \mathscr{P}(t) d\sigma ds dt
=\int_{-a}^0\int_{s,t>0}\frac{e^{-\omega t - \eta s}}{1+2\sigma}\partial_x \mathscr{P}(s)\partial_x \mathscr{P}(t) d\sigma ds dt,
\]
and each is finite.
\end{lemma}
\begin{proof}
We begin by maximizing $x + y$ over $\sigma \in [-a,0],$ where it is seen that the maximum is attained at $\sigma = 0,$ at which point,
\[
(x(0)+y(0))^2 = \left(\sqrt{b(1-a)} + \sqrt{a(1-b)}\right)^2 = \lambda_{+}.
\]
Thus, it follows that $(x+y)^2 < \lambda_{+} < \omega$ for all $-a \leq \sigma \leq 0.$  Recall that $\mathscr{P}(t)$ is given by $e^{t(x^2+y^2)}I_0(2xyt)$, and thus
\[
\partial_x\mathscr{P}(t)
=2xte^{t(x^2+y^2)}I_0(2xyt)
+2yte^{t(x^2+y^2)}I_1(2xyt).
\]
Using that $0 \leq I_n(2xyt) \leq e^{2xyt}$ for all $n,$ it follows that
\[
0 \leq \partial_x\mathscr{P}(t) \leq 2(x+y)e^{t(x+y)^2},
\]
for $x,y \geq 0.$
It follows that there is a constant $C$ so that for all $ -a \leq \sigma \leq 0,$
\[
0 \leq \frac{e^{-\omega t - \eta s}}{1+2\sigma}\partial_x \mathscr{P}(s)\partial_x \mathscr{P}(t) < C(x+y)^2e^{-(\omega-(x+y)^2)t -(\eta-(x+y)^2)s)}.
\]
Using the bound on $x+y$ derived above,
\[
0 \leq \frac{e^{-\omega t - \eta s}}{1+2\sigma}\partial_x \mathscr{P}(s)\partial_x \mathscr{P}(t) < C\lambda_+^2e^{-(\omega-\lambda_{+})t -(\eta-\lambda_{+})s)}.
\]
Thus, provided that $\omega > \lambda_{+}$ and $\eta > \lambda_{+},$ the order of integration may be reversed by Fubini.
\end{proof}

We can now compute the bivariate Laplace transform of $\mathscr{C}(s,t).$
\begin{lemma}
\label{Lpintegral}
\[
\Lp_{s,t}[\mathscr{C}(s,t)](\eta,\omega) = \frac{\alpha}{8}\int_{(1-2a)^2}^1 \frac{n_1 - n_2\rho}{\left(p_1-p_2\rho\right)^{\tfrac32}\left(r_1-r_2\rho\right)^{\tfrac32}}d\rho,
\]
where these parameters are given by
\begin{align*}
n_1 &= -\omega\eta[(1-\lambda_{-}-\lambda_{+})^2] - (\omega+\eta)\lambda_-\lambda_+(1-\lambda_{-}-\lambda_{+})-(1-\lambda_--\lambda_++2\lambda_-\lambda_+)\lambda_-\lambda_+  \\
n_2 &= -\omega\eta[1-\lambda_- -\lambda_+ +2\lambda_-\lambda_+] + (\omega+\eta-1)\lambda_-\lambda_+\\
p_1 &= \omega(1 - \lambda_- -\lambda_+) + \lambda_-\lambda_+ \\ 
p_2 &= \omega - \omega^2 \\
r_1 &= \eta(1 - \lambda_- -\lambda_+) + \lambda_-\lambda_+ \\ 
r_2 &= \eta - \eta^2.
\end{align*} 
\end{lemma}

\begin{proof}
We start by commuting the integration in $\sigma$ and the Laplace transform in~\eqref{LandR}.  To evaluate these Laplace transforms, we recall Lemma~\ref{derivative_characteristic}, where the Laplace transform $\Lp_t[\partial_x\mathscr{P}(t)]$ was computed to be
\[
\Lp_t[\partial_x\mathscr{P}(t)](\omega) = 
\frac{2x(\omega + y^2 - x^2)}{\left(\left(\omega-x^2-y^2\right)^2-4x^2y^2\right)^{\tfrac32}}, ~~~~ \omega > (x+y)^2. 
\]
The quantity $x^2 - y^2$ simplifies to
\[
x^2 - y^2 = \frac{b-a}{1+2\sigma}.
\]
The Laplace transform of $\partial_x\mathscr{P}(t)$ can be rewritten as 
\[
\Lp_t[\partial_x\mathscr{P}(t)](\omega) = 
\frac{2x(1+2\sigma)^2(\omega(1+2\sigma) - (b - a))}{\left((\omega^2-\omega)(1+2\sigma)^2 - (1-2a)(1-2b)\omega + (b-a)^2\right)^{\tfrac32}}, ~~~~ \omega > \lambda_{+},
\]
for $\sigma \in [-a,0].$  By symmetry, the Laplace transform $\partial_y\mathscr{P}(t)$ is
\[
\Lp_t[\partial_y\mathscr{P}(t)](\omega) = 
\frac{2y(1+2\sigma)^2(\omega(1+2\sigma) + (b - a))}{\left((\omega^2-\omega)(1+2\sigma)^2 - (1-2a)(1-2b)\omega + (b-a)^2\right)^{\tfrac32}}, ~~~~ \omega > \lambda_{+}.
\]
Define $\Delta(\omega)$ to be
\[
\Delta(\omega) = \left((\omega^2-\omega)(1+2\sigma)^2 - (1-2a)(1-2b)\omega + (b-a)^2\right),
\]
and define $\rho = (1+2\sigma)^2.$  We will now split the computation of $\mathscr{C}(s,t)$ into two pieces for simplicity's sake.  The first piece is
\begin{align*}
\Lp_{s,t}\left[ \partial_x \mathscr{P}(s)\partial_x \mathscr{P}(t) + \partial_y \mathscr{P}(s) \partial_y \mathscr{P}(t)\right](\eta,\omega)\hspace{-4cm}&\hspace{4cm} \\
&=4\rho^2\frac{(x^2+y^2)[\omega\eta\rho+(b-a)^2] - (b-a)^2(\omega+\eta)}{\Delta(\omega)^{\tfrac32}\Delta(\eta)^{\tfrac32}}.
\end{align*}
The second piece is
\begin{align*}
\Lp_{s,t}\left[ \partial_x \mathscr{P}(s)\partial_y \mathscr{P}(t) + \partial_y \mathscr{P}(s) \partial_x \mathscr{P}(t)\right](\eta,\omega)\hspace{-4cm}&\hspace{4cm}& \\
&= \frac{8xy\rho^2[\omega\eta\rho-(b-a)^2]}{\Delta(\omega)^{\tfrac32}\Delta(\eta)^{\tfrac32}}.
\end{align*}
Combining these two pieces,
\begin{align*}
\Lp_{s,t}[\mathscr{C}](\eta,\omega)
&=\alpha \int_{-a}^0 \frac{\rho^2\left[((x^2+y^2)(1-x^2-y^2) - 4x^2y^2)\omega\eta\rho+((x^2+y^2)(1-x^2-y^2) + 4x^2y^2)(b-a)^2\right]}{\sqrt{\rho}\Delta(\omega)^{\tfrac32}\Delta(\eta)^{\tfrac32}} \\
&~~~~~~~~~+\frac{\rho^2\left[-(b-a)^2(\omega+\eta)(1-x^2-y^2)\right]}{\sqrt{\rho}\Delta(\omega)^{\tfrac32}\Delta(\eta)^{\tfrac32}}d\sigma. 
\end{align*}
We simplify some of these expressions,
\begin{align*}
\hspace{1cm}&\hspace{-1cm} ((x^2+y^2)(1-x^2-y^2) - 4x^2y^2) \\
&= \tfrac12\rho^{-2}\left[(-4a(1-a)+1)(-4b(1-b) + 1)\right] + \tfrac12\rho^{-1}\left[ 1 - 2b(1-b) -2a(1-a)\right], \\
\hspace{1cm}&\hspace{-1cm} ((x^2+y^2)(1-x^2-y^2) + 4x^2y^2) \\
&= \tfrac12\rho^{-1}\left[ \rho - 1 + 2b(1-b) + 2a(1-a)\right], \\
\hspace{1cm}&\hspace{-1cm} 1-x^2-y^2 \\
&= \tfrac12\rho^{-1}\left[ \rho + (1-2a)(1-2b) \right].
\end{align*}
After changing the integration to be over $\rho,$ we produce the desired formula.
\end{proof}
We will explicitly evaluate the integral in Lemma~\ref{Lpintegral} to conclude that
\begin{lemma}
\[
\Lp_{s,t}[\mathscr{C}(s,t)]
= \alpha \frac{d^2}{\left[\sqrt{(\tilde\omega+d)(\tilde\eta-d)} + \sqrt{(\tilde\omega-d)(\tilde\eta+d)}\right]^2\sqrt{\tilde\omega^2-d^2}\sqrt{\tilde\eta-d^2}},
\]
where $d=\left(\tfrac{\lambda_+ + \lambda_-}{2}\right),$ $\tilde\omega = \omega - \left(\tfrac{\lambda_+ + \lambda_-}{2}\right)$ and $\tilde\eta = \eta - \left(\tfrac{\lambda_+ + \lambda_-}{2}\right).$ 
\end{lemma}
By comparing with the expression for $\Lp_{s,t}[\mathscr{T}(s,t)]$ derived in \eqref{finalform}, this lemma completes the proof of the diagonalization of the covariances.
\begin{proof}
Differentiating both sides, it can be shown that
\[
\int \frac{n_1 - n_2\rho}{\left(p_1-p_2\rho\right)^{\tfrac32}\left(r_1-r_2\rho\right)^{\tfrac32}}d\rho = \frac{-2}{\left(p_1r_2 - p_2r_1\right)^2}\frac{2n_2p_1r_1 - n_1(p_1r_2+p_2r_1)+\rho(2n_1p_2r_2 - n_2(p_1r_2+p_2r_1)}{\left(p_1-p_2\rho\right)^{\tfrac12}\left(r_1-r_2\rho\right)^{\tfrac12}}.
\]
The indefinite integral can be greatly simplified, plugging in some of the $n,p,$ and $r$ terms.
\[
\int \frac{n_1 - n_2\rho}{\left(p_1-p_2\rho\right)^{\tfrac32}\left(r_1-r_2\rho\right)^{\tfrac32}}d\rho = \frac{2((1-\lambda_--\lambda_+)(\eta+\omega)+2\lambda_-\lambda_+)-2\rho(\omega+\eta-2\omega\eta)}{(\eta-\omega)^2\left(p_1-p_2\rho\right)^{\tfrac12}\left(r_1-r_2\rho\right)^{\tfrac12}}.
\]
The antiderivative will now be evaluated at both endpoints.  At $\rho=1,$ it becomes
\[
2\frac{2\omega\eta - (\omega+\eta)(\lambda_-+\lambda_+) + 2\lambda_-\lambda_+}{(\eta-\omega)^2\sqrt{(\omega-\lambda_-)(\omega-\lambda_+)}\sqrt{(\eta-\lambda_-)(\eta-\lambda_+)}}.
\]
To evaluate at $\rho=(1-2a)^2,$ it is helpful to work with $a$ and $b$ instead of $\lambda_{\pm}.$  Using the formulae
\begin{equation*}
\lambda_-\lambda_+ = (b-a)^2 ~~~\text{and}~~~
\lambda_- + \lambda_+ = 2(a + b - 2ab ),
\end{equation*}
the antiderivative evaluated at $\rho=(1-2a)^2$ is simply
\[
\frac{4}{(\eta-\omega)^2}.
\]
At last we can give a single expression for the Laplace transform of the covariance function:
\begin{equation*}
\Lp_{s,t}[\mathscr{C}(s,t)]
= \frac{\alpha}{4} \frac{\left[\sqrt{(\omega-\lambda_-)(\eta-\lambda_+)} - \sqrt{(\omega-\lambda_+)(\eta-\lambda_-)}\right]^2}{(\eta-\omega)^2\sqrt{(\omega-\lambda_-)(\omega-\lambda_+)}\sqrt{(\eta-\lambda_-)(\eta-\lambda_+)}}~.
\end{equation*}
Recall that $r=\left(\tfrac{\lambda_+ + \lambda_-}{2}\right),$ $\tilde\omega = \omega - \left(\tfrac{\lambda_+ + \lambda_-}{2}\right)$ and $\tilde\eta = \eta - \left(\tfrac{\lambda_+ + \lambda_-}{2}\right).$  We rewrite this expression in terms of these modified parameters to get
\begin{equation*}
\Lp_{s,t}[\mathscr{C}(s,t)]
= \alpha\frac{d^2}{\left[\sqrt{(\tilde\omega+d)(\tilde\eta-d)} + \sqrt{(\tilde\omega-d)(\tilde\eta+d)}\right]^2\sqrt{\tilde\omega^2-d^2}\sqrt{\tilde\eta-d^2}} = \alpha \Lp_{s,t}[\mathscr{T}(s,t)].
\end{equation*}

\end{proof}

\section{Extension to Continuously Differentiable Test Functions}
\label{sec:extension}

We learned the idea for the extending the CLT from the appendix of Anderson-Zeitouni~\cite{AndersonZeitouni}.  Roughly speaking, one would like to extend a CLT for polynomial test functions to a CLT for a larger class of functions, the hope being to invoke the density of the polynomials.  However, it needs to be assured that error-in-approximation produces small error in the fluctuations when evaluated on the empirical process.  The property of a matrix ensemble that allows one to execute this is a type of global concentration of eigenvalues.  See also Proposition 11.6 in~\cite{AndersonZeitouni} and Lemma 1 of~\cite{ShcherbinaLemma} for related approaches. 
\begin{proposition}
\label{extension_lemma}
Let $\{A_n\}$ be an ensemble of matrices with compact spectral support $S,$ and let $V : C^1(S) \to \mathbb{R}$ be a postive semidefinite quadratic form for which there is constant $C_1$ so that $V(f) \leq C_1^2\|f\|_{Lip}^2$ for all $f \in C^1(S).$ Suppose that $\{A_N\}$ satisfies a polynomial-type CLT, i.e. for all polynomials $g,$
\[
\tr g(A_n) - \Exp \tr g(A_n) \Rightarrow N(0,V(g))
\]
and additionally $\Var \tr g(A_n) \to V(g).$ If the ensemble satisfies a Poincar\'e type \emph{concentration inequality}, i.e.
\begin{equation}
\label{concentration_inequality}
\Var( \tr f (A_n)) \leq C_2^2\|f\|_{Lip}^2.
\end{equation}
for some constant $C_2$ independent of $n$ and any Lipschitz $f$ on $S$, then the polynomial CLT extends to all $C^1$ functions $f: S \to \mathbb{R},$ as
\[
\tr f(A_n) - \Exp \tr f(A_n) \Rightarrow N(0,V(f)).
\]  
\end{proposition}

\begin{proof}
We recall the quadratic Wasserstein metric
\[
W_2(\mu,\nu)^2 = \inf \expect \left( X - Y \right)^2,
\]
with the infimum over all couplings $(X,Y)$ with marginals $\mu$ and $\nu$ respectively.  For a random variable $X$, we let $\lawof X$ denote its law.  It is well known that $W_2(\lawof X_n, \lawof X) \to 0$ if and only if $X_n \Rightarrow X$ and $\expect X_n^2 \to \expect X^2$ (see Theorem 7.12 of~\cite{Villani}).  For any $f \in C^1(S),$ let $Z_f$ denote a centered normal random variable with variance $V(f).$  Thus for any polynomial $g,$ $W_2(\lawof (\tr g(A_n) - \Exp \tr g(A_n)), \lawof Z_g) \to 0.$  

Let $f$ be any $C^1(S)$ function.  By Weierstrass approximation of the derivative of $f$, there is a sequence of polnomials $p_k$ so that $\|f - p_k\|_{Lip} \to 0$ as $k \to \infty.$  It follows that $V(p_k) \to V(f)$ from its continuity with respect to the Lipschitz seminorm, and hence that $W_2(\lawof Z_{p_k}, \lawof Z_f) \to 0$ as $k \to \infty.$  For any $k$ we can bound,
\begin{align*}
W_2\left( \lawof(\tr f(A_n) - \Exp \tr f(A_n)), \lawof Z_f \right)
\hspace{-1in}&\hspace{1in}
\leq \\
&W_2\left( \lawof(\tr f(A_n) - \Exp \tr f(A_n)), \lawof(\tr p_k(A_n) - \Exp \tr p_k(A_n)) \right) \\
+&W_2\left(\lawof(\tr p_k(A_n) - \Exp \tr p_k(A_n)), \lawof Z_{p_k}\right) \\
+&W_2\left(\lawof Z_{p_k}, \lawof Z_f\right).
\end{align*}
By the concentration inequality, it is possible to bound 
\[
\expect \left[\tr f(A_n) - \Exp \tr f(A_n) - \tr p_k(A_n) - \Exp \tr p_k(A_n) \right]^2 \leq C_1^2\|f\|_{Lip}^2,
\]
from which it follows that 
$W_2\left( \lawof(\tr f(A_n) - \Exp \tr f(A_n)), \lawof(\tr p_k(A_n) - \Exp \tr p_k(A_n)) \right) \leq C_1\|f\|_{Lip}$ by the definition of the Wasserstein metric as the infimum over couplings.  Likewise 
\[
W_2\left(\lawof Z_{p_k}, \lawof Z_f\right)=\left|\sqrt{V(p_k)} - \sqrt{V(f)}\right| \leq \sqrt{V(p_k -f)} \leq C_2\|f\|_{Lip}.
\]
Therefore, from the polynomial CLT,
\[
\limsup_{n \to \infty}
W_2\left( \lawof(\tr f(A_n) - \Exp \tr f(A_n)), \lawof Z_f \right)
\leq
(C_1+C_2)\|f-p_k\|_{Lip}.
\]
Taking $k \to \infty$ completes the proof.
\end{proof}

Note that the moment-method proof used for the polynomial CLT implies $\Var \tr(g(A_n)) \to V(g),$ and that the bound of $V(f) \leq C\|f\|_{Lip}$ follows from Remark~\ref{variancecomparison}.  To show that linear statistics of the Jacobi ensemble satisfy a Poincar\'e inequality, we will work directly with the joint eigenvalue density function.  Recall \eqref{jacobidensity}, which stated
\begin{equation*}
d\mu_J(\lambda_1, \ldots, \lambda_n) = \frac1Z \prod_ i \lambda_i^{\tfrac{n}{\alpha}\left[\tfrac ba - 1\right] +\tfrac1\alpha-1}(1-\lambda_i)^{\tfrac{n}{\alpha}\left[\tfrac{1-b}{a} - 1\right] +\tfrac1\alpha-1} \prod_{i <j } |\lambda_i - \lambda_j |^{\tfrac 2\alpha}. 
\end{equation*}
We first show that the Jacobi ensemble satisfies a log-Sobolev inequality, which is strictly stronger than the Poincar\'e inequality.  Define the entropy of a non-negative measurable function $f$ with respect to a probability measure $\mu$ by
\[
\Ent{\mu}(f) \Def \int f\log f d\mu - \left(\int f d\mu \right)
\left(\log \int f d\mu\right),
\]
if $\int f\log(1+f) d\mu < \infty$ and $+\infty$ otherwise.  Our tool in this direction is a consequence of the well-known Bakry-Emery condition, the content of which is contained in the following proposition (see Proposition 3.1 of~\cite{BobkovLedoux}).
\begin{proposition}
\label{bakryemery}
Suppose that $d\mu = e^{-U}dx$ is supported on a convex set $\Omega.$  If there is a $c>0$ so that for all $x \in \operatorname{int}(\Omega),$ $\Hess U(x) \geq c \Id,$ where $Id$ is the identity matrix and $\geq$ is the partial ordering on positive semidefinite matrices, then for all smooth functions $f$ on $\R^n,$
\[
\Ent{\mu}(f^2) \leq \frac{2}{c} \int \left| \nabla f\right|^2 d\mu.
\]
\end{proposition}

To prove the log-Sobolev inequality with the appropriate constant, we need only check that the condition of Proposition~\ref{bakryemery} is satisfied.  This we do in showing the following lemma.
\begin{lemma}
\label{Jacobi_log_Sobolev}
The Jacobi ensemble satisfies a log-Sobolev inequality
\[
\Ent{\mu_J}(f^2) \leq \frac{2}{c} \int \left| \nabla f\right|^2 d\mu_J,
\]
with $c=4\tfrac{n}{\alpha}\min\left\{\tfrac ba - 1, \tfrac{1-b}{a} -1 \right\}.$
\end{lemma}
\begin{proof}
We will employ Proposition~\ref{bakryemery}, and thus we begin by computing the Hessian of the logarithm of the density.  Let $p \Def \tfrac{n}{\alpha}\left[\tfrac ba - 1\right] +\tfrac1\alpha-1,$ and let $q \Def \tfrac{n}{\alpha}\left[\tfrac{1-b}{a} - 1\right] +\tfrac1\alpha-1.$ The first derivative is given by
\[
\frac{d}{d\lambda_i}(\log (d\mu_J)) = \frac{p}{\lambda_i} - \frac{q}{1-\lambda_i} + \frac{1}{\alpha} \sum_{j \neq i} \frac{1}{\left|\lambda_i - \lambda_j\right|}.
\]
The second derivative is thus
\[
\frac{d^2}{d\lambda_i^2}(\log (d\mu_J)) = -\frac{p}{\lambda_i^2} - \frac{q}{(1-\lambda_i)^2} - \frac{1}{\alpha} \sum_{j \neq i} \frac{1}{\left(\lambda_i - \lambda_j\right)^2}.
\]
The mixed partials are just
\[
\frac{d}{d\lambda_j}\frac{d}{d\lambda_i}(\log (d\mu_J)) = -\frac{1}{\alpha}\frac{1}{\left(\lambda_i - \lambda_j\right)^2}.
\]
By the method of Gershgorin discs we conclude that the smallest eigenvalue of $\Hess(-\log d\mu_J)$ is at least
\[
\min_{\substack{1\leq i \leq n \\ 0 \leq \lambda_i \leq 1}}\left[\frac{p}{\lambda_i^2} + \frac{q}{(1-\lambda_i)^2}\right] \geq 4\min\{p,q\} \geq \frac{4n}{\alpha}\min\left\{\tfrac ba - 1, \tfrac{1-b}{a} -1 \right\}.
\]
\end{proof}
It is now a simple manner to show the needed concentration inequality and prove Theorem~\ref{clt}.
\begin{proof}[Proof of Theorem~\ref{clt}]
From Proposition~\ref{polynomial_clt} and Proposition~\ref{extension_lemma}, it suffices to demonstrate a constant $C$ so that $\Var \tr f \leq C \|f\|_{Lip}^2,$ with the Lipschitz norm on $[0,1],$ for all Lipschitz $f.$  This is turn follows from the somewhat sharper inequality that 
\[
\Var \tr f \leq C \int \left| \partial_{\lambda_i}(f(\lambda_i)) \right|^2 d\mu_J(\lambda_1, \ldots, \lambda_n) = \frac{C}{n} \int \left| \nabla \tr f \right|^2 d\mu_J(\lambda_1,\ldots, \lambda_n),
\]
where in the last step we have used the symmetry of the linear statistic.  It is a standard fact that the log-Sobolev inequality implies the Poincar\'e inequality with half the constant (see~\cite[Chapter 5]{Ledoux}).  Thus by Lemma~\ref{Jacobi_log_Sobolev} we have that for all smooth functions $f,$
\[
\Var \tr f \leq \frac{\alpha}{4n\min\left\{\tfrac ba - 1, \tfrac{1-b}{a} -1 \right\}} \int  \left| \nabla \tr f \right|^2 d\mu_J(\lambda_1,\ldots, \lambda_n).
\]
Extension to Lipschitz functions follows from the density of smooth functions in $L^2,$ and the proof is complete.
\end{proof}

\section{Computing the Expectation}
\label{sec:expectation}
In this section, we will prove Theorem~\ref{jacobi_expectation}.  To establish the theorem for polynomial linear statistics $\phi$, a proof will be given that follows a similar tract to the analogous statement proven for the Laguerre and Hermite ensembles in~\cite{DumitriuEdelman}.  The key to this method of proof is establishing a certain palindromy.  Recall that a polynomial $p(z) = a_n z^n + a_{n-1} z^{n-1} + \cdots + a_1 z + a_0$ is palindromic in $z$ if $a_n z^n + a_{n-1} z^{n-1} + \cdots + a_1 z + a_0 = a_0 z^n + a_{1} z^{n-1} + \cdots + a_{n-1} z + a_n,$ or equivalently that $p(z) = z^n p( z^{-1} ).$

\begin{theorem}
\label{moment_palindromy}
The scaled moment $\tfrac1n\Exp \tr(A^k)$ has a series expansion
\[
\tfrac1n\Exp \tr(A^k) = \sum_{j=0}^{\infty}\eta_k(j,\alpha)n^{-j}
\]
whose coefficients $\eta_k(j,\alpha)$ are palindromic polynomials in $(-\alpha)$ of degree $j.$
\end{theorem}
While the proof of this palindromy works for all of these coefficients $\eta$ simultaneously, only the palindromy of $\eta_k(0,\alpha)$ and $\eta_k(1,\alpha)$ are required for Theorem~\ref{jacobi_expectation}.  Especially, palindromy forces $\eta_k(0,\alpha)$ to have no $\alpha$ dependence, and it forces $\eta_k(1,\alpha)$ to be a multiple of $1-\alpha.$ As will be seen, this allows the $\alpha=0$ case to be used to study the arbitrary $\alpha$ case.  As the proof of Theorem~\ref{moment_palindromy} requires symmetric function theory, we delay the proof to Appendix~\ref{sec:symmetric} to allow a brief introduction to the relevant symmetric function theory.

\begin{proof}[Proof of Theorem~\ref{jacobi_expectation} for polynomial $\phi$]

Formally, let $\tilde{m}(x)$ be the moment generating function for the ensemble, and expand each moment asymptotically around $n = \infty,$ i.e.
\[
\tilde{m}(x) = \frac1n\sum_{k=0}^\infty \frac{\Exp \tr(A^k)}{x^k} = \sum_{k=0}^\infty x^{-k}\sum_{j=0}^\infty \eta_k(j,\alpha)n^{-j},
\]
then one has, to order $\tfrac1n,$
\[
\tilde{m}(x) = \sum_{k=0}^{\infty}x^{-k}\left(\eta_k(0,\alpha) + \frac{\eta_k(1,\alpha)}{n}\right) + O(n^{-2}).
\]
The $\alpha$-dependence of either of these terms is completely determined by Theorem~\ref{moment_palindromy}, as $\eta_k(0,\alpha)$ can have no $\alpha$ dependence, and $\eta_k(1,\alpha)$ is a multiple of $(1 - \alpha).$ Define $m_0(x)$ and $m_1(x)$ so that 
\[
\tilde{m}(x) \big\vert_{\alpha = 0} = m_0(x) + \tfrac1n m_1(x) + O(n^{-2}).
\]
In this notation, the palindromy shows that 
\[
\tilde{m}(x) = m_0(x) + (1-\alpha)\tfrac1n m_1(x) + O(n^{-2}).
\]

Further, the $\alpha=0$ case, for fixed $n$, is relatively simple.  As observed by Sutton~\cite{SuttonThesis}, the Jacobi matrix model tends to a deterministic one as $\alpha \tendsto 0;$ precisely, it has eigenvalues that are the roots of $J_n^{r,s}$, the Jacobi polynomial of degree $n$ and parameters \[
r = n\left(\tfrac{b}{a} -1 \right),~~~
s = n\left(\tfrac{1-b}{a} -1 \right).
\]   
Suppose that the roots of $J_n^{r,s}$ are given by $\{\lambda_i\}_{i=1}^n$.  Then for $\alpha=0,$ the moment generating function takes on the form
\[
\tilde{m}(x) = \frac{1}{n} \sum_{k=0}^\infty \sum_{i=1}^n \frac{\lambda_i^k}{x^k} = \frac{1}{n} \sum_{i=1}^n \frac{1}{x-\lambda_i} = \frac{1}{n}\left(\ln J_n^{r,s}(x)\right)'=
\frac{1}{n} \frac{{J_n^{r,s}}'(x)}{J_n^{r,s}(x)}.
\]

Using the differential recurrence for Jacobi polynomials, it follows that $\tilde{m}(x)$ satisfies a formal power series equation
\begin{equation}
\label{differential_recurrence}
\tilde{m}^2 + \frac{\tfrac{r+1}{n} - x\tfrac{r+s+2}{n}}{x(1-x)}\tilde{m} + \frac{1+\tfrac{r+s+1}{n}}{x(1-x)} + \frac{\tilde{m}'}{n} = 0. 
\end{equation}
It follows that the constant-order term $m_0$ satisfies
\[
am_0^2 + \frac{b-a -(1-2a)x}{x(1-x)}m_0 + \frac{1-a}{x(1-x)} =0.
\]
This leads to an explicit form for $m_0,$
\begin{align*}
m_0 &= \frac{(a-b)+(1-2a)x - \sqrt{(b-a -(1-2a)x)^2-4a(1-a)x(1-x)}}{2ax(1-x)} \\
&= \frac{(a-b)+(1-2a)x - \sqrt{(x-\lambda_{-})(x-\lambda_{+})}}{2ax(1-x)},
\end{align*}
where
\[
\lambda_{\pm} = \left[\sqrt{b(1-a)} \pm \sqrt{a(1-b)} \right]^2.
\]

Note that $\lambda_{\pm}$ are always real, and that they are always on $[0,1]$.  They are $0$ and $1$ exactly when $a=b$ and when $a=1-b$, respectively. Taking an inverse Stieltjes transform gives absolutely continuous part
\[
d\mu(x) = \frac{\sqrt{-(x-\lambda_{-})(x-\lambda_{+})}}{2\pi a x(1-x)}\mathbf{1}_{[\lambda_{-},\lambda_{+}]}.
\]
This integrates to $1,$ as it can be shown that
\[
\int_{\lambda_{-}}^{\lambda_{+}} 
\frac{\sqrt{-(x-\lambda_{-})(x-\lambda_{+})}}{x(1-x)}= \pi\left[1-\sqrt{\lambda_{-}\lambda_{+}} - \sqrt{(1-\lambda_{-})(1-\lambda_{+})}\right] = 2\pi a.
\]
Note that this implies that the distribution has no discrete part. 

In the same fashion, one can also derive an explicit form for $m_1$.  Pulling out the $\tfrac1n$ terms from~\eqref{differential_recurrence}, one is left with
\[
2am_0m_1 + \frac{(b-a) - (1-2a)x}{x(1-x)}m_1 + \frac{1-2x}{x(1-x)}am_0 + \frac{a}{x(1-x)} + am_0' = 0.
\]
Solving for $m_1,$
\[
m_1 = \frac{-x+\tfrac12(\lambda_{-}+\lambda_{+}) + \sqrt{(x-\lambda_{+})(x-\lambda_{-})}}{2(x-\lambda_{+})(x-\lambda_{-})}.
\]
To recover the density, one again applies the inverse Stieltjes transform.  When $x$ is neither $\lambda_{+}$ nor $\lambda_{-},$ the limit $\lim_{\epsilon \tendsto 0} m_1(x+i\epsilon)$ exists, and 
\[
\lim_{\epsilon \tendsto 0} m_1(x+i\epsilon) = -\frac{1}{2\pi\sqrt{-(x-\lambda_{+})(x-\lambda_{-})}}\mathbf{1}_{(\lambda_{-},\lambda_{+})}(x).
\]
Computing the inverse Stieltjes transform at either of the poles, it is seen that there are point masses, so that the entire signed measure is
\[
\nu(x) = 
\tfrac14\delta_{\lambda_{-}}(x)
+\tfrac14\delta_{\lambda_{+}}(x)
-\frac{1}{2\pi\sqrt{-(x-\lambda_{+})(x-\lambda_{-})}}\mathbf{1}_{(\lambda_{-},\lambda_{+})}(x).
\]

\end{proof}

\section{Numerics for the Extremal Case}
\label{sec:numerics}
In this section, we investigate the choice $p=q=1,$ which was not covered by Theorem~\ref{clt}.  The method of proof breaks down in this extreme case, and so we have run a numerical simulation to help conjecture if the theorem extends.

\begin{figure}[h]
\label{fig:plots}
\begin{centering}
\subfigure[Quantile plots for $tr(A)$ experiments]{
\label{fig:qqplot}
\includegraphics[width=0.45\textwidth]{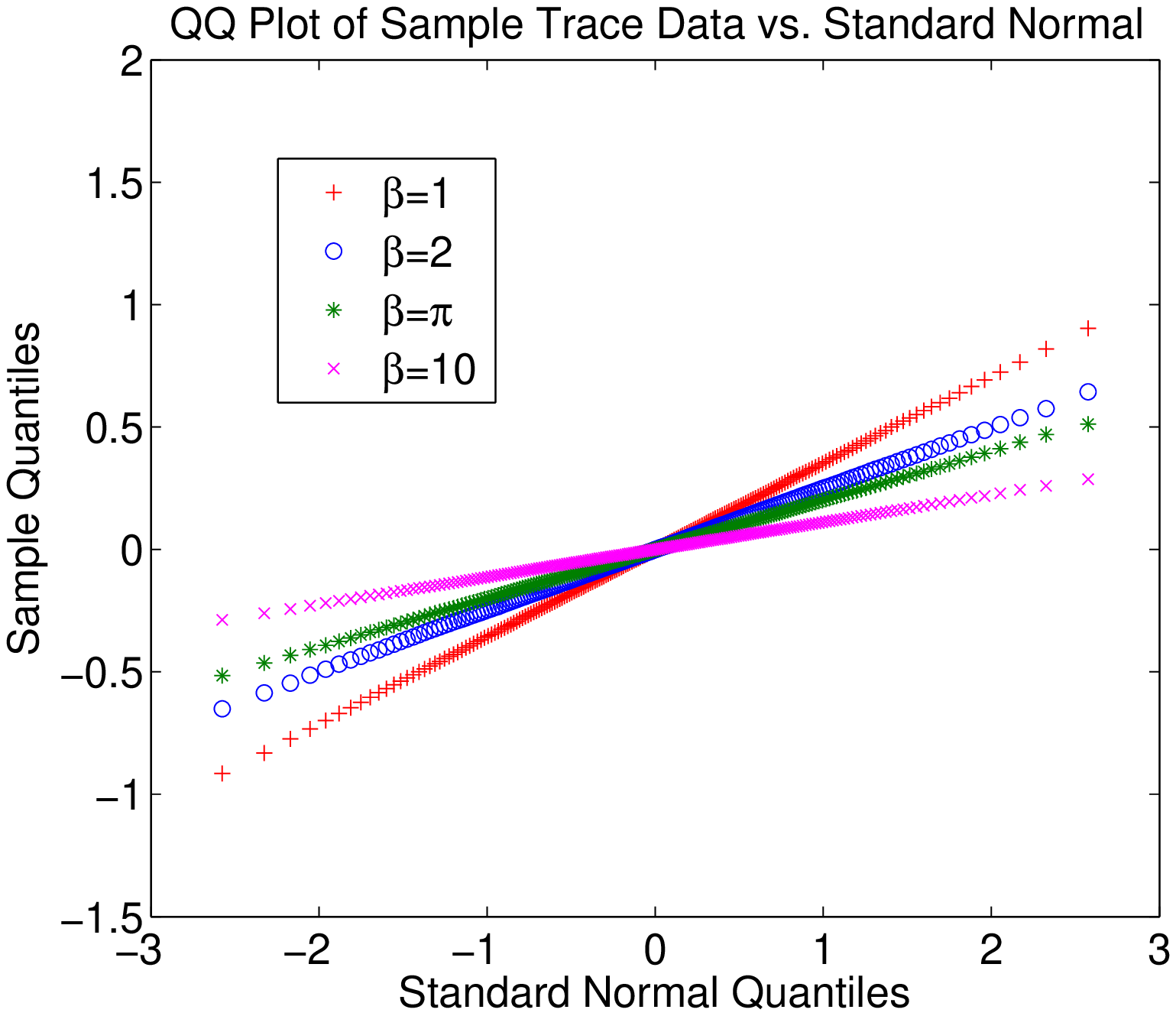}
}
\subfigure[PDF plots for $tr(A)$ experiments]{
\label{fig:pdfplot}
\includegraphics[width=0.45\textwidth]{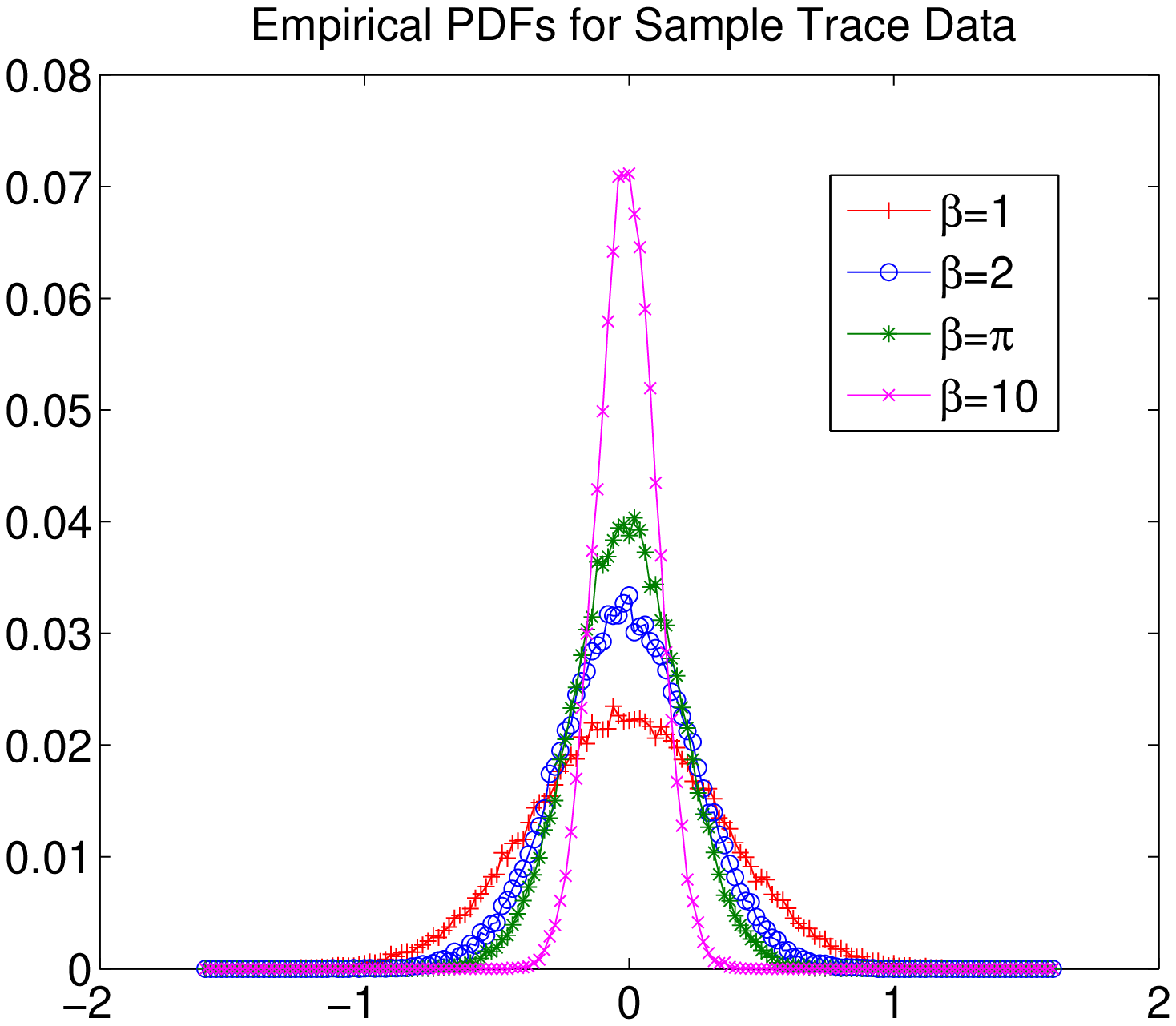}
}
\caption{Experimental data for different values of $\beta,$ with $n=5000$ and $50000$ samples of each.
All experiments were run in Matlab R2010B, using the Edelman-Sutton matrix model. 
}
\end{centering}
\end{figure}

In the alternate parameterization we have that $a=\tfrac12$ and $b=\tfrac12.$  The density of the Jacobi ensemble becomes
\begin{equation}
\label{jacobidensity_critical_a}
d\mu_J(\lambda_1, \ldots, \lambda_n) = \tfrac1Z \prod_ i \lambda_i^{\tfrac{\beta}{2}-1}(1-\lambda_i)^{\tfrac{\beta}{2}-1} \prod_{i <j } |\lambda_i - \lambda_j |^{\beta}. 
\end{equation}
Note that the constraining potential no longer carries any dependence on $n.$  However, because the particles are forced to lie on $[0,1]$ (physically speaking, they are trapped in an infinite potential well), it is likely that we have some limiting behavior. For polynomial test functions and $\beta = 2,$ this case is covered by a theorem of Johansson (see Theorem 3.1 of~\cite{JohanssonCompactGroups}).

However, the method of proof used here breaks down in the case $a < \tfrac 12$, as it requires the entries of the sparse matrix model to have uniform variance estimates on the order of $n^{-1}.$  When $a = \tfrac 12,$ the matrix model entries are
\begin{align*}
c_i \sim \sqrt{\operatorname{Beta}(\tfrac{\beta}{2}i, \tfrac{\beta}{2}i)} 
~~&\text{ and }~~
c_i' \sim \sqrt{\operatorname{Beta}(\tfrac{\beta}{2}i, \tfrac{\beta}{2}(i+1))}. 
\end{align*}
The variances of entries $c_i$ and $s_i$ are on the order of $i^{-1},$ for which reason many of the arguments in later sections are no longer valid. To see how different the $a=b=\tfrac12$ case is from the $a < \tfrac12$ case, consider taking $f(x) = x.$  It is easily seen that
\[
\cmt{1} \tendsto \sum_{i=1}^\infty (c_i^2 - \tfrac 12)(1- (c_i')^2 - (c_{i-1}')^2),
\]
with the convergence in $L^2.$ Note that while a normal limit is expected if the summands are becoming infinitesimal (and this is what happens when $a < \tfrac 12$), the normal limit here must follow from something else; in particular, the staircase dependency structure of the variables can not be ignored.  We invite the reader to check that the variable is symmetric and to note how much cancellation occurs in computing the second and fourth moments (they are $1/(8\beta)$ and $3/(64\beta^2)$ respectively).  Again, the fact that this variable is normally distributed follows from the mentioned theorem of Johansson.
 

%
%
%
%
%
%
%

\begin{appendices}

\section{Symmetric Functions}
\label{sec:symmetric}
To find the asymptotic distribution of the traces, we will appeal to Kadell's integral formula~\cite{Kadell}.  This formula makes use of Jack functions, and so we will provide a skeletal introduction to the relevant portions of symmetric function theory.  A more expansive treatment is available in Macdonald's book~\cite{MacdonaldBook}, whose notation we will follow.

By a partition $\lambda$, we mean a non-increasing sequence of positive integers.  The notation $\lambda \vdash n$, read `$\lambda$ partitions $n$,' means that the sum of the parts of $\lambda$ equal $n.$  There is an important pictorial representation of a partition called a Young diagram.  The diagram representation of a partition $(\lambda_1,\ldots, \lambda_n)$ is drawn by placing $\lambda_1$ boxes horizontally in a row, placing $\lambda_2$ boxes horizontally below that, continuing through $n$ and left justifying each row.  Having drawn a diagram representation, we can easily define the \emph{conjugate}\footnote{This is also called the \emph{transpose}.} partition $\lambda'$ to be that partition represented by reflecting the diagram across the vertical axis and rotating counterclockwise by a quarter turn. 
\begin{example} The partition $\lambda=(5,4,1)$ is to the left, and its conjugate $\lambda'=(3,3,2,2,1)$ is to the right.

\begin{center}
\begin{tikzpicture}
\foreach \x in {1,2,...,5}
	\draw (\x, 5) +(-.5,-.5) rectangle ++(.5,.5);
\foreach \x in {1,2,...,4}
	\draw (\x, 4) +(-.5,-.5) rectangle ++(.5,.5);
\foreach \x in {1,2,...,1}
	\draw (\x, 3) +(-.5,-.5) rectangle ++(.5,.5);
\foreach \y in {1,2,...,5}
	\draw (7, \y) +(-.5,-.5) rectangle ++(.5,.5);
\foreach \y in {2,3,...,5}
	\draw (8, \y) +(-.5,-.5) rectangle ++(.5,.5);
\draw (9, 5) +(-.5,-.5) rectangle ++(.5,.5);
\draw (9, 4) +(-.5,-.5) rectangle ++(.5,.5);
\end{tikzpicture}
\end{center}
\end{example}
Many formulas in symmetric function theory have sums or products computed from statistics of the diagram representation.  For our purposes, we will need the arm length $a$, arm co-length $a'$, leg length $l$, and leg co-length $l'$ of a box $s$.  The statistics $a(s)$ and $a'(s)$ are the number of boxes to the right and to the left of box $s,$ respectively.  Likewise, the statistics $l(s)$ and $l'(s)$ are the number of boxes below and above box $s$.

\begin{example}  

\begin{tikzpicture}
\draw (-2, 3) node {This is $\lambda=(6,5,5).$};
\foreach \x in {1,2,...,6}
	\draw (\x, 3) +(-.5,-.5) rectangle ++(.5,.5);
\foreach \x in {1,2,...,5}
	\draw (\x, 2) +(-.5,-.5) rectangle ++(.5,.5);
\foreach \x in {1,2,...,5}
	\draw (\x, 1) +(-.5,-.5) rectangle ++(.5,.5);
\fill[green!30] (4,1) +(-.5,-.5) rectangle ++(.5,.5);
\draw (4, 1) +(-.5,-.5) rectangle ++(.5,.5);
\draw (4,1) node {$s$};
\draw (8, 3) node {$a(s) = 1$};
\draw (8, 2.5) node {$a'(s) = 3$};
\draw (8, 2) node {$l(s) = 0$};
\draw (8, 1.5) node {$l'(s) = 2$};
\end{tikzpicture}
\end{example}

The ring of symmetric functions $\Lambda,$ are all those formal power series with complex coefficients\footnote{More often in the literature on Jack functions, these coefficients are defined to be from $\mathbf{Q}(\alpha),$ but the distinction here is immaterial.} in the indeterminates $\{x_1,x_2,\ldots \},$ that are symmetric under permutation of the indices.  In this application, the symmetric functions will be evaluated at some point $y = (y_1,y_2,\ldots,y_n) \in \mathbf{C}^n,$ where it is understood that $f(y) = f(y_1,y_2,\ldots,y_n,0,0,\ldots).$  In this way, symmetric functions specialize to symmetric polynomials.

The symmetric functions of interest here are the power sums, as they describe traces.  For an integer $k,$ define $p_k$ by
\[
p_k = x_1^k + x_2^k + x_3^k + \cdots,
\]
and for a partition $\lambda = (\lambda_1,\lambda_2,\ldots, \lambda_n)$, define $p_\lambda$ by
\[
p_\lambda = p_{\lambda_1} p_{\lambda_2}\cdots p_{\lambda_n}.
\]
These are called the \emph{power sum symmetric functions,} and $\{p_\lambda\}_\lambda$ are a basis for $\Lambda.$ Note that the trace of a power of a matrix $\tr A^k$ can alternately be expressed as $p_k$ evaluated at the eigenvalues of $A.$ 

The second basis we require are the \emph{Jack symmetric functions} $P^\alpha_\lambda.$  For those interested, there is a concise introduction available in Stanley's paper~\cite{Stanley}.  By virtue of being a basis, it is possible to write $p_k$ as a finite linear combination of $\{P^\alpha_\lambda\}_{\lambda \vdash k}$.  

There are multiple normalizations for the Jack functions in the literature.  In citing some theorems, we will require a second normalization, $J^\alpha_\lambda.$  The two are related, as $J_\lambda^\alpha = c(\lambda,\alpha) P_\lambda^\alpha,$ where 
\begin{equation}
\label{c_lambda_alpha}
c(\lambda,\alpha) = \prod_{s\in\lambda}\left(\alpha a(s) + l(s) + 1\right),
\end{equation}
using the arm length $a(s)$ and leg length $l(s).$

One final tool we will use is the Macdonald automorphism $\omega_\alpha.$  It is defined in terms of the symmetric power functions by $\omega_\alpha p_k = \alpha p_k;$ it is extended to each $p_\lambda$ as a multiplicative homomorphism; and at last it is extended to all $\Lambda$ as a $\mathbf C$-linear transformation.  This automorphism acts on the Jack functions in a nice way as well, as by a formula of Stanley~\cite{Stanley}, 
\begin{equation}
\label{Stanley_automorphism}
\omega_{-1/\alpha}J_{\lambda'}^{\alpha^{-1}} = (-\alpha)^{|\lambda|}J_\lambda^\alpha.
\end{equation}

\subsection{Kadell's Integral}
Kadell's integral (see~\cite{Kadell}) is a generalization of Selberg's integral~\cite{Selberg}, which states the following
\[
\int_{[0,1]^n} \prod_{i<j} |x_i - x_j|^{2/\alpha} \prod_{i=1}^n x_i^{r-1}(1-x_i)^{s-1} dx = \prod_{i=1}^n \frac{\Gamma(1+\tfrac{i}{\alpha})\Gamma(r+\tfrac{i-1}{\alpha})\Gamma(s+\tfrac{i-1}{\alpha})}{\Gamma(1+\tfrac1\alpha)\Gamma(r+s+\tfrac{n+i-2}{\alpha})}.
\]

It was generalized to include the Jack function $P_\lambda^{1/\alpha}(x)$ in the integrand.  Letting $W(n,\alpha,r,s)$ be the integrand of Selberg's integral, Kadell's integral is
\begin{equation}
\label{Kadell_thm}
\int_{[0,1]^n} P_\lambda^{1/\alpha}(x) W(n,\alpha,r,s)dx = n!v_\lambda^\alpha \prod_{i=1}^n \frac{\Gamma(\lambda_i+r+\tfrac{n-i}{\alpha})\Gamma(s+\tfrac{n-i}{\alpha})}{\Gamma(\lambda_i+r+s+\tfrac{2n-i-1}{\alpha})}, 
\end{equation}
where the term $v_\lambda^\alpha$ is defined as 
\begin{equation}
\label{v_lambda_def}
v_\lambda^\alpha = \prod_{i<j} \frac{\Gamma(\lambda_i - \lambda_j + \tfrac{j-i+1}{\alpha})}{\Gamma(\lambda_i - \lambda_j + \tfrac{j-i}{\alpha})}.
\end{equation}

Our goal is to show that
\[
\int_{[0,1]^n} \tfrac{P_\lambda^{1/\alpha}}{P_\lambda^{1/\alpha}(I_n)} W(n,\alpha,r,s)dx,
\]
where $I_n=(1,1,\ldots 1)$ has $n$ $1's$, has a quasi-palindromic property.  The constant $P_\lambda^{1/\alpha}(I_n)$ is computable in terms of diagram statistics.  From formula VI.10.20 of~\cite{MacdonaldBook},
\begin{equation}
\label{normalization_def}
P_\lambda^{1/\alpha}(I_n) = \prod_{s\in\lambda} \left(\tfrac{n+\alpha a'(s)-l'(s)}{\alpha a(s) + l(s) + 1} \right)= \tfrac{1}{c(\lambda,\alpha)} \prod_{s\in\lambda} \left({n+\alpha a'(s)-l'(s)}\right)
\end{equation}
where $c(\lambda,\alpha)$ is the constant that relates $J_\lambda^\alpha$ and $P_\lambda^\alpha$ (see \eqref{c_lambda_alpha}).  To compare the two, we will convert Kadell's expression using $\Gamma$ functions into a Young diagram formula.

Recall that a quotient of $\Gamma$ functions, also known as the Pochhammer symbol $(x)_k,$ may be expressed alternately as
\[
\frac{\Gamma(x+k)}{\Gamma(x)} = (x)_k = (x)(x+1)\cdots(x+k-1),
\]
when $k$ is a natural number.  Define the generalized Pochhammer symbol $(t)_\mu$ (also known as the shifted factorial) to be
\begin{equation}
\label{shifted_factorial_def}
(t)_\mu = \prod_{s\in \mu} \left(t + a'(s) - \tfrac{1}{\alpha}l'(s)\right).
\end{equation}
In terms of these expressions, \eqref{normalization_def} can be rewritten as
\begin{equation}
\label{normalization_form2}
P_\lambda^{1/\alpha}(I_n) = \frac{(\tfrac{n}{\alpha})_\lambda (\alpha)^{|\lambda|}}{c(\lambda,\alpha)}.
\end{equation}

We will need a closely related quantity to $c(\lambda,\alpha),$ so define $c'(\lambda, \alpha)$ to be
\[
\prod_{s\in\lambda}\left(\alpha a(s) + l(s) + \alpha\right).
\] 
Both $c(\lambda, \alpha)$ and $c'(\lambda,\alpha)$ can be expressed as products of $\Gamma$ terms, which we will need to rewrite Kadell's integral.
Write out the terms in $\alpha^{-|\lambda|}c'(\lambda,\alpha)$ by going from right to left along the first row of the diagram of $\lambda.$ 
There are $\lambda_1 - \lambda_2$ terms that have $l(s) = 0:$
\[
(\tfrac0\alpha + 1 + 0)
(\tfrac0\alpha + 1 + 1)
\cdots
(\tfrac0\alpha + 1 + \lambda_1-\lambda_2 - 1) = \frac{\Gamma(\lambda_1 - \lambda_2 + 1)}{\Gamma(1)}.
\]
There are then $\lambda_2 - \lambda_3$ terms that have $l(s) = 1:$
\[
(\tfrac1\alpha + 1 + \lambda_1 - \lambda_2)
(\tfrac1\alpha + 1 + \lambda_1 - \lambda_2 + 1)
\cdots
(\tfrac1\alpha + 1 + \lambda_1-\lambda_3 - 1) = 
\frac{\Gamma(\lambda_1 - \lambda_3 + 1 + \tfrac1\alpha)}{\Gamma(\lambda_1 - \lambda_2 + 1 + \tfrac1\alpha)}.
\]
This pattern continues until at last there are $\lambda_n$ terms that have $l(s) = n-1:$
\[
(\tfrac{n-1}{\alpha} + 1 + \lambda_1 - \lambda_n)
(\tfrac{n-1}{\alpha} + 1 + \lambda_1 - \lambda_n +1)
\cdots
(\tfrac{n-1}{\alpha} + \lambda_1)
=
\frac{\Gamma(\lambda_1 +1 + \tfrac{n-1}{\alpha})}{\Gamma(\lambda_1 - \lambda_n + 1 + \tfrac{n-1}{\alpha})}.
\]
Writing out all the terms in the first row gives
\[
\frac{\Gamma(\lambda_1 - \lambda_2 + 1)}{\Gamma(1)}
\frac{\Gamma(\lambda_1 - \lambda_3 + 1 + \tfrac1\alpha)}{\Gamma(\lambda_1 - \lambda_2 + 1 + \tfrac1\alpha)}
\frac{\Gamma(\lambda_1 - \lambda_4 + 1 + \tfrac2\alpha)}{\Gamma(\lambda_1 - \lambda_3 + 1 + \tfrac2\alpha)}
\cdots
\frac{\Gamma(\lambda_1 +1+ \tfrac{n-1}{\alpha})}{\Gamma(\lambda_1 - \lambda_n + 1 + \tfrac{n-1}{\alpha})}.
\]
Inducting over the rows, it follows that $c'(\lambda,\alpha)$ can be written as
\begin{equation}
\label{cprimegamma}
c'(\lambda,\alpha) = (\alpha)^{|\lambda|}\prod_{i<j} \frac{\Gamma(\lambda_i - \lambda_j + 1 + \tfrac{j-i-i}{\alpha})}{\Gamma(\lambda_i - \lambda_j + 1 + \tfrac{j-i}{\alpha})}\prod_{i=1}^n \Gamma(\lambda_i + 1 + \tfrac{n-i}{\alpha})
\end{equation}
If one does the same expansion along the first row for $c(\lambda,\alpha),$ one gets
\[
\frac{\Gamma(\lambda_1 - \lambda_2 +\tfrac1\alpha)}{\Gamma(\tfrac1\alpha)}
\frac{\Gamma(\lambda_1 - \lambda_3 + \tfrac2\alpha)}{\Gamma(\lambda_1 - \lambda_2 + \tfrac2\alpha)}
\frac{\Gamma(\lambda_1 - \lambda_4 + \tfrac3\alpha)}{\Gamma(\lambda_1 - \lambda_3 + \tfrac3\alpha)}
\cdots
\frac{\Gamma(\lambda_1 + \tfrac{n}{\alpha})}{\Gamma(\lambda_1 - \lambda_n + \tfrac{n}{\alpha})}.
\]
Repeating the analogous procedure for the rest of the rows, we eventually conclude
\begin{equation}
\label{cgamma}
c(\lambda,\alpha) = (\alpha)^{|\lambda|}\prod_{i<j} \frac{\Gamma(\lambda_i - \lambda_j + \tfrac{j-i}{\alpha})}{\Gamma(\lambda_i - \lambda_j + \tfrac{j-i+1}{\alpha})}\prod_{i=1}^n \frac{\Gamma(\lambda_i + \tfrac{n-i+1}{\alpha})}{\Gamma(\tfrac1\alpha)}.
\end{equation}
Equations~\eqref{cprimegamma} and \eqref{cgamma} allow \eqref{v_lambda_def} to be rewritten as
\begin{equation}
\label{v_lambda_form2}
v_\lambda^\alpha = \tfrac{(\alpha)^{|\lambda|}}{c(\lambda,\alpha)}\prod_{i=1}^n \frac{\Gamma(\lambda_i + \tfrac{n-i+1}{\alpha})}{\Gamma(\tfrac1\alpha)}.
\end{equation}
We can repeat the same procedure as used for $c$ and $c'$ to show that $(t)_\lambda$ can be computed by
\begin{equation}
\label{shifted_factorial_form2}
(t)_\lambda =  \prod_{i=1}^n \frac{\Gamma(t - \tfrac{i-1}{\alpha}+\lambda_i)}{\Gamma(t - \tfrac{i-1}{\alpha})}.
\end{equation}
This allows the expression in \eqref{v_lambda_form2} for $v_\lambda^\alpha$ to be replaced by
\begin{equation}
\label{v_lambda_form3}
v_\lambda^\alpha = \tfrac{(\alpha)^{|\lambda|}}{c(\lambda,\alpha)}\prod_{i=1}^n \frac{\Gamma(\lambda_i +\frac{n}{\alpha} - \tfrac{i-1}{\alpha})}{\Gamma(\tfrac1\alpha)} \frac{\Gamma(\frac{n}{\alpha} - \tfrac{i-1}{\alpha})}{\Gamma(\frac{n}{\alpha} - \tfrac{i-1}{\alpha})}
=\tfrac{(\alpha)^{|\lambda|}}{c(\lambda,\alpha)}\left(\tfrac{n}{\alpha}\right)_\lambda\prod_{i=1}^n\frac{\Gamma(\tfrac{i}{\alpha})}{\Gamma(\tfrac1\alpha)}.
\end{equation}
Combine this expression for $v_\lambda^\alpha$ with Kadell's integral formula~\eqref{Kadell_thm} and the simplified expression~\eqref{normalization_form2} for $P_\lambda^{1/\alpha}(I_n)$ to get
\begin{align} 
\int_{[0,1]^n}\frac{P_\lambda^{1/\alpha}(x)}{P_\lambda^{1/\alpha}(I_n)}W(n,\alpha,r,s)dx
& \notag\\
&\hspace{-1.5in}=\frac{c(\lambda,\alpha)}{(\tfrac{n}{\alpha})_\lambda (\alpha)^{|\lambda|}}
n!\frac{(\alpha)^{|\lambda|}}{c(\lambda,\alpha)}\left(\tfrac{n}{\alpha}\right)_\lambda\prod_{i=1}^n\frac{\Gamma(\tfrac{i}{\alpha})}{\Gamma(\tfrac1\alpha)}
 \prod_{i=1}^n \frac{\Gamma(\lambda_i+r+\tfrac{n-i}{\alpha})\Gamma(s+\tfrac{n-i}{\alpha})}{\Gamma(\lambda_i+r+s+\tfrac{2n-i-1}{\alpha})} \notag\\
&\hspace{-1.5in}=  \prod_{i=1}^n \frac{\Gamma(1+\tfrac{i}{\alpha})}{\Gamma(1+\tfrac1\alpha)}\frac{\Gamma(\lambda_i+r+\tfrac{n-i}{\alpha})\Gamma(s+\tfrac{n-i}{\alpha})}{\Gamma(\lambda_i+r+s+\tfrac{2n-i-1}{\alpha})}. \notag
\end{align} 

Let $\mu_J$ be the $(\tfrac1\alpha,r,s)$-Jacobi ensemble measure on $[0,1]^n$.  This has density function proportional to $W(n,\alpha,r,s)$, but it is appropriately renormalized to be a probability measure.  This normalization is given by Selberg's integral.

The integral expression above can be rewritten as
\begin{align} 
\int_{[0,1]^n}\frac{P_\lambda^{1/\alpha}(x)}{P_\lambda^{1/\alpha}(I_n)}d\mu_J(x)  &\notag\\
&\hspace{-1in}=  \prod_{i=1}^n \frac{\Gamma(1+\tfrac{i}{\alpha})}{\Gamma(1+\tfrac1\alpha)}\frac{\Gamma(\lambda_i+r+\tfrac{n-i}{\alpha})\Gamma(s+\tfrac{n-i}{\alpha})}{\Gamma(\lambda_i+r+s+\tfrac{2n-i-1}{\alpha})}\frac{1}{\int W(n,\alpha,r,s)dx} \notag\\
&\hspace{-1in} =\prod_{i=1}^n \frac{\Gamma(\lambda_i+r+\tfrac{n-i}{\alpha})}{\Gamma(r+\tfrac{n-i}{\alpha})}
\frac{\Gamma(r+s+\tfrac{2n-i-1}{\alpha})}{\Gamma(\lambda_i+r+s+\tfrac{2n-i-1}{\alpha})} \notag\\
\label{Jacobi_expectation}
&\hspace{-1in} =  \frac{\left(r+\tfrac{n-1}{\alpha}\right)_{\lambda}}{\left(r+s+\tfrac{2n-2}{\alpha}\right)_{\lambda}}.
\end{align} 
\subsection{Palindromy }
\begin{lemma}
\label{palindromic_jack}
Let 
\[
\int_{[0,1]^n}\tfrac{P_\lambda^{1/\alpha}(x)}{P_\lambda^{1/\alpha}(I_n)}d\mu_J(x)
=\sum_{k=0}^\infty \rho(k,\lambda, \alpha)n^{-k}
\]
be the series expansion about $n=\infty.$  The coefficients $\rho(k,\lambda,\alpha)$ are skew-palindromic in that
\[
\rho(k,\lambda,\alpha) = (-\alpha)^{k}\rho(k,\lambda',\tfrac{1}{\alpha})
\]
\end{lemma}
\begin{proof}

In the calculation that follows, let $f(\lambda,\alpha,t)= f(t) = \alpha a'(t) - l'(t),$ for tableau block $t \in \lambda.$  Starting from the formula computed in~\eqref{Jacobi_expectation}, and applying formula~\eqref{shifted_factorial_def} gives 
\begin{align*}
\frac{\left(\tfrac{nb}{a\alpha}\right)_{\lambda}}{\left(\tfrac{n}{a\alpha}\right)_{\lambda}}
&=\prod_{t \in \lambda} \frac{nb + af(t)}{n + af(t)} \\
&=\prod_{t\in\lambda} \frac{b+ \tfrac{a}{n}f(t)}{1+\tfrac{a}{n}f(t)} \\
&=\prod_{t\in\lambda}\left(b\left(1+ \tfrac{a}{bn}f(t)\right)\sum_{k=0}^\infty \left(-\tfrac{a}{n}f(t)\right)^k\right)\\
&=b^{|\lambda|}\prod_{t\in\lambda}\left(1+ (1-\tfrac1b)\sum_{k=1}^\infty \left(-\tfrac{a}{n}f(t)\right)^k\right). \\
\intertext{Let $M(\lambda,k)$ be the collection of all $k$-element multisets sampled from $\lambda.$  If $\tau \in M(\lambda,k)$ is such a multiset, let $m_\tau(t)$ denote the multiplicity of $t \in \tau$ and let $\epsilon_\tau(t)$ be the characteristic function for $t \in \tau.$  The sum can be written as:}
&=b^{|\lambda|}\sum_{k=0}^\infty {n}^{-k}\left[\sum_{\tau \in M(\lambda,k)}\prod_{t \in \lambda} \left(-a f(t)\right)^{m_\tau(t)}\left(1-\tfrac1b\right)^{\epsilon_\tau(t)} \right].
\end{align*}
This gives an explicit form for the coefficients $f(k,\lambda,\alpha).$  Mapping $\lambda$ to $\lambda'$ induces a bijection mapping the collection $M(\lambda,k)$ to $M(\lambda',k).$ In the conjugate, the arm co-length $a'$ and leg co-length $l'$ are reversed, so that $f(t)$ becomes $\alpha l'(t) - a'(t).$ Thus $f(\lambda,\alpha,t) = -\alpha f(\lambda',\alpha^{-1},t),$ so that
\[
\sum_{\tau \in M(\lambda,k)}\prod_{t \in \lambda} \left(-af(\lambda,\alpha,t)\right)^{m_\tau(t)}\left(1-\tfrac1b\right)^{\epsilon_\tau(t)}
=
(-\alpha)^{k}\sum_{\tau \in M(\lambda',k)}\prod_{t \in \lambda'} \left(-af(\lambda',\alpha^{-1},t)\right)^{m_\tau(t)}\left(1-\tfrac1b\right)^{\epsilon_\tau(t)}
\]
\end{proof}

Let $J_\lambda^{1/\alpha}$ be the Jack functions renormalized by
\begin{equation}
\label{j_definition}
J_\lambda^{1/\alpha} = c(\lambda,\alpha)P_\lambda^{1/\alpha}.
\end{equation}

Expand the symmetric power function $p_k$ as 
\[
p_k = \sum_{\lambda \vdash k}\xi(\lambda,\alpha)J_\lambda^{1/\alpha}.
\]
By applying Stanley's formula (see~\eqref{Stanley_automorphism}), it follows (see~\cite{DumitriuEdelman}) that 
\begin{equation}
\label{xi_palindromy}
\xi(\lambda,\alpha) = (-\alpha)^{1-|\lambda|}\xi(\lambda',\alpha^{-1}).
\end{equation}
One last piece is needed.  The normalization factor $J_\lambda^{1/\alpha}(I_n)$ can be computed by relating~\eqref{normalization_form2} and the definition of $J_\lambda^\alpha$ in~\eqref{j_definition}.   These two combined give that 
\[
J_\lambda^{1/\alpha(I_n)}=\left(\tfrac{n}{\alpha}\right)_\lambda(\alpha)^{|\lambda|}=\prod_{t \in \lambda} \left(n + \alpha a'(t) - l'(t)\right);
\]
expand this as a polynomial in $n$, i.e. put 
\[
\prod_{t \in \lambda} \left(n + \alpha a'(t) - l'(t)\right) = \sum_{j=0}^{|\lambda|}\zeta(j,\lambda,\alpha)n^j.
\]
Because the product can be expressed as
\[
\prod_{t \in \lambda} \left(n + \alpha a'(t) - l'(t) \right)
=\prod_{t \in \lambda'}  \left(n + (-\alpha)\left(\alpha^{-1}a'(t) - l'(t)\right)\right),
\]
it follows that
\begin{equation}
\label{zeta_palindromy}
\zeta(j,\lambda,\alpha) = (-\alpha)^{|\lambda|-j}\zeta(j,\lambda',\alpha^{-1}).
\end{equation}

\begin{proof}[Proof of Theorem~\ref{moment_palindromy}]
Expand $p_{[k]}$ in the Jack function basis:
\begin{align*}
\tfrac1n\Exp_\alpha p_{[k]} &= \tfrac1n\sum_{\lambda \vdash k}\xi(\lambda,\alpha)\Exp_\alpha J_\lambda^{1/\alpha} \\
&= \tfrac1n\sum_{\lambda \vdash k}\xi(\lambda,\alpha)J_\lambda^{1/\alpha}(I_n)\Exp_\alpha \frac{J_\lambda^{1/\alpha}}{J_\lambda^{1/\alpha}(I_n)}. 
\intertext{Apply Lemma~\ref{palindromic_jack}, and expand $J_\lambda^{1/\alpha}(I_n).$  Note that the alternative normalization used in the Lemma cancels out.}
\tfrac1n\Exp_\alpha p_{[k]} 
&= \tfrac1n\sum_{\lambda \vdash k}\xi(\lambda,\alpha)\left(\sum_{j=0}^{k}\zeta(j,\lambda,\alpha)n^j\right)\left(\sum_{j=0}^\infty \rho(j,\lambda,\alpha)n^{-j}\right) \\
&= \sum_{j=-\infty}^{k}n^{j-1}\left(\sum_{\lambda \vdash k}\xi(\lambda,\alpha)\sum_{l=0}^k\zeta(l,\lambda,\alpha)\rho(l-j,\lambda,\alpha)\right),
\intertext{with $\rho(l-j,\lambda,\alpha) = 0$ for negative $l-j$.}
\end{align*}
This gives a formula for $\eta_k(j,\alpha),$ namely that 
\[
\eta_k(j,\alpha) = \sum_{\lambda \vdash k}\xi(\lambda,\alpha)\sum_{l=0}^k\zeta(l,\lambda,\alpha)\rho(l+j-1,\lambda,\alpha).
\]
The $j < 0$ terms vanish, which can be seen because the trace can naturally be bounded as
\[
\tfrac1n|\Exp_\alpha p_{k}| \leq \tfrac1n\Exp_\alpha \sum_{i=0}^n |x_i|^k \leq \tfrac1n n=1,
\]
as the Jacobi distribution is supported on $[0,1]^n.$

We will show that each $\eta_k(j,\alpha)$ is palindromic.  Applying Lemma~\ref{palindromic_jack},~\eqref{xi_palindromy}, and~\eqref{zeta_palindromy}, these can be written as 
\begin{align*}
\eta_k(j,\alpha) &= \sum_{\lambda \vdash k}\xi(\lambda,\alpha)\sum_{l=0}^k\zeta(l,\lambda,\alpha)\rho(l+j-1,\lambda,\alpha) \\
&= \sum_{\lambda \vdash k}(-\alpha)^{1-k}\xi(\lambda',\alpha^{-1})\sum_{l=0}^k(-\alpha)^{k-l}\zeta(l,\lambda',\alpha^{-1})(-\alpha)^{l+j-1}\rho(l+j-1,\lambda',\alpha^{-1}) \\
&= (-\alpha)^{j}\sum_{\lambda \vdash k}\xi(\lambda',\alpha^{-1})\sum_{l=0}^k\zeta(l,\lambda',\alpha^{-1})\rho(l+j-1,\lambda',\alpha^{-1}) \\
\intertext{The sum is over all partitions of $k$, so taking conjugates makes no difference.  Thus,}
\eta_k(j,\alpha)&= (-\alpha)^{j}\eta_k(j,\alpha^{-1}).
\end{align*}

The last claim we make is that $\eta_k(j,\alpha)$ is a polynomial in $\alpha$ of degree $j$.  This is more involved, and requires that we appeal to Edelman and Sutton's tridiagonal matrix model (see the start of Section 3).  The moment $\tfrac1n \Exp p_{k} = \tfrac1n \Exp \tr(A^k)$ can be written in terms of a sum over alternating bridges (see Section 3.1),
\[
\tfrac1n \Exp \tr A^k = \tfrac 1n \sum_{\bar w \in \mathcal{A}_{2k}} \Exp \left(B_\beta\right)_{\bar w+i}.
\]

A priori, these expectations are moments of random variables distributed as the square root of a Beta random variable.  However, by Lemma~\ref{even_many_steps}, the alternating bridge visits each matrix entry an even number of times.  Thus, any term in the sum takes the form
\[
\Exp \prod_{i=1}^k c_{\omega_{2i-1}}^{2m_{2i-1}}s_{\omega_{2i-1}}^{2n_{2i-1}}{c'}_{\omega_{2i}}^{2m_{2i}}{s'}_{\omega_{2i}}^{2n_{2i}}, 
\]
where $\omega_i$ ranges over the matrix entries referenced by the bridge $\bar w$ and $\sum_0^{2k} m_i = k.$  By independence, this expectation is a product of terms of the form
\[
\Exp c_\omega^{2m} s_\omega^{2n}~~\text{ and }~~\Exp {c'}_\omega^{2m} {s'}_\omega^{2n}
\]
By Lemma~\ref{beta_moment_expansion}, each such Beta moment admits a series expansion around $n=\infty$ and a $K$ so that
\[
\Exp \left(B_\beta\right)_{\bar w+i}=\sum_{m=0}^\infty n^{-m}\alpha^m\Omega_{\bar w+i,m}(n),
\]
where $0 < \Omega_{\bar w+i,m}(n) < K^m$ for all $n.$  Moreover, this constant $K$ can be chosen independently of $\bar w + i$.  Thus the entire trace admits such a series expansion, 
\begin{align*}
\tfrac1n \Exp \tr(A^k) 
&= \frac 1n \sum_{i=1}^{n} \sum_{\bar w \in \mathcal{A}_{2k}} \sum_{m=0}^\infty n^{-m}\alpha^m\Omega_{\bar w + i,m}(n) \\
&= \sum_{m=0}^\infty n^{-m}\alpha^m\left(\frac 1n \sum_{i=1}^n \sum_{\bar w \in \mathcal{A}_{2k}} \Omega_{\bar w+i,m}(n)\right).
\end{align*}
Because the cardinality of $\mathcal{A}_{2k}$ is at most ${2k \choose k}$, the sum $\Omega_m(n) = \frac 1n \sum_{i=k+1}^{n-k} \sum_{\bar w \in \mathcal{A}_{2k}} \Omega_{\bar w+i,m}(n)$ satisfies an estimate $0 < \Omega_m(n) < {2k \choose k}K^m=CK^m.$  
Thus there are two expansions for the trace, valid for all $n$ sufficiently large, i.e.
\begin{equation}
\label{two_expressions}
\sum_{j=0}^\infty \eta(j,\alpha)n^{-j} = \frac1n\Exp \tr(A^k) =
\sum_{j=0}^\infty \Omega_j(n)\alpha^j n^{-j}
\end{equation}

The left hand side expansion shows that the $n \tendsto \infty$ limit must exist.  Thus
\[
\eta(0,\alpha)= \lim_{n \tendsto \infty}
\sum_{j=0}^\infty \eta(j,\alpha)n^{-j}
 = \lim_{n \tendsto \infty} \sum_{j=0}^\infty \Omega_j(n)\alpha^j n^{-j} = \lim_{n \tendsto \infty}\Omega_0(n).
\]
In particular, $\eta(0,\alpha)$ has no $\alpha$ dependence. The proof now proceeds by induction.  Suppose that for all $j < l$, the term $\eta(j,\alpha)$ is a polynomial in $\alpha$ of degree $j$. It should be shown that $\eta(l,\alpha)$ is a polynomial in $\alpha$ of degree $l$.
The limit
\[
\lim_{n \tendsto \infty} n^l\left[\tfrac1n \Exp \tr(A^k) - \sum_{j=0}^{l-1}\eta(j,\alpha)n^{-j}\right] = 
\lim_{n \tendsto \infty} \sum_{j=0}^\infty \eta(j+l)n^{-j} = \eta(l,\alpha) 
\]
exists by virtue of the $\eta$ expansion, and by substituting the right hand side of \eqref{two_expressions}, it follows that
\[
\eta(l, \alpha) = \lim_{n \tendsto \infty} n^l\left[\sum_{j=0}^\infty \Omega_j(n)\alpha^j n^{-j} - \sum_{j=0}^{l-1}\eta(j,\alpha)n^{-j}\right] = \lim_{n \tendsto \infty} \sum_{j=0}^{l-1}\left[\Omega_j(n)\alpha^j - \eta(j,\alpha)\right]n^{l-j} + \Omega_l(n)\alpha^l.
\]
By the inductive hypothesis, this limit can be written in the form
\[
\eta(l,\alpha) = \lim_{n\tendsto \infty}f_0(n)+f_1(n)\alpha + f_2(n)\alpha^2 + \cdots + f_l(n)\alpha^l,
\]
and the limit exists for each fixed $\alpha.$  Take $l+1$ distinct values of $\alpha.$  The convergence is uniform on this finite set $\alpha_0, \ldots,\alpha_l$, and so each $f_i(n)$ converges, where $0 \leq i \leq l.$  Thus $\eta(j,\alpha)$ is a polynomial of degree $l$ in $\alpha,$ concluding the proof. 
\end{proof}

\begin{lemma}
\label{beta_moment_expansion}
Let $f_r(n)$ and $f_s(n)$ be positive real-valued functions defined on $\mathbb{N}$ so that
\[
0 < f_r(n) \leq C_2, ~~~~~~ 0 < f_s(n) \leq C_2i,  ~~~~~~ C_1 < f_r(n) + f_s(n),
\]
where $C_i$ are some positive constants.  Let $r = \alpha^{-1} f_r(n) n$, $s = \alpha^{-1} f_s(n) n$, and let ${\bf X}\sim \operatorname{Beta}(r,s).$  There is an asymptotic expansion
\[
\Exp \left[{\bf X}^k(1-{\bf X})^l\right] = \sum_{m=0}^\infty n^{-m}\alpha^mp_m(n),
\]
and a constant $K$ depending only on $k,l,C_1,\text{ and }C_2$ so that $0 < p_m(n) < K^m.$
\end{lemma}
\begin{proof}
The expectation, which can be computed using Euler's Beta integral formula, gives that
\begin{align*}
\Exp \left[{\bf X}^k(1-{\bf X})^l\right] &= \frac{(r)_k(s)_l}{(r+s)_{k+l}}. \\
\intertext{Substituting in the definitions for $r$ and $s$ and writing out the Pochhammer symbols gives}
&= \prod_{i=0}^{k-1}\frac{\alpha^{-1}f_rn+i}{\alpha^{-1}(f_r+f_s)n+i}
\prod_{i=0}^{l-1}\frac{\alpha^{-1}f_sn+i}{\alpha^{-1}(f_r+f_s)n+k+i} \\
\end{align*}
All rational terms in this product produce similar asymptotic series expansions, and so we will only examine one.  Working with a term from the left hand product,
\begin{align*}
\frac{\alpha^{-1}f_rn+i}{\alpha^{-1}(f_r+f_s)n+i}
&=\left(\frac{\alpha^{-1}f_rn+i}{\alpha^{-1}(f_r+f_s)n}\right)\left(\frac{1}{1+\tfrac{i}{\alpha^{-1}(f_r+f_s)n}}\right) \\
\intertext{Provided that $n$ is sufficiently large (depending on $C_1$ and $\alpha$), this can be expanded as a series.}
\frac{\alpha^{-1}f_rn+i}{\alpha^{-1}(f_r+f_s)n+i}
&=\frac{\alpha^{-1}f_rn+i}{\alpha^{-1}(f_r+f_s)n}\sum_{m=0}^{\infty}\left(f_r+f_s\right)^{-m}n^{-m}\alpha^m \\
&=\frac{f_r}{f_r+f_s}+\sum_{m=1}^{\infty}\left(\tfrac{f_r}{f_r+f_s}+i\right)\left(f_r+f_s\right)^{-m}n^{-m}\alpha^m \\
&=\sum_{m=1}^\infty \tilde{p}_m(n) n^{-m} \alpha^m. \\
\end{align*}
The coefficients $\tilde{p}_m(n)$ satisfy an estimate
\[
0 < \tilde{p}_m(n) < (C_1C_2 + k)(C_1)^{-m}.
\]
\end{proof}

\section{Poincar\'e Inequality for $\operatorname{Beta}$ }
\label{sec:beta}
\begin{lemma}
\label{beta_poincare}
Let $Y \sim \operatorname{Beta}(p,q).$  For any Lipschitz function $f$ on $[0,1],$
\[
\Var f(Y) \leq \frac{1}{4(p+q)} \expect \left|f'(Y)\right|^2.
\]
\end{lemma}
We note that in the case that both $p$ and $q$ are greater than $1,$ the density is log-concave, and it is possible to use the general theory outlined by Bobkov in \cite{Bobkov} to produce an equivalent bound, but we require the inequality to hold for all $p$ and $q$ positive, and thus we use an alternative technique. 
\begin{proof}
We begin by showing the analogous bound for the translated random variable $X = 2Y - 1,$ and write $Y = T(X) \Def \frac{1}{2}(X+1).$  The density of $Y$ is given by
\[
\frac{d\mu_{\beta}}{dx} = \frac{1}{Z_{p,q}} (1-x)^{p-1} (1+x)^{q-1}. 
\]
We will show that for any Lipschitz function $f$ on $[-1,1],$ that
\begin{equation}
\label{pistatement}
\Var f(X) \leq \frac{1}{p+q} \expect\left[ (1-X^2)\left|f'(X)\right|^2\right]. 
\end{equation}
As will be seen in the proof, this inequality is attained taking $f$ to be a multiple of the linear Jacobi polynomial (for definitions, see~\cite{SzegoBook}).
The proof follows from~\eqref{pistatement}, as
\begin{align*}
\Var f(Y) &= \Var (f \circ T)(X) \\  
&\leq \frac{1}{p+q} \expect\left[ (1-X^2) \left|(f \circ T)'(X)\right|^2\right] \\
&\leq \frac{1}{p+q} \expect{ \left|(f \circ T)'(X)\right|^2} \\
&= \frac{1}{p+q} \expect{ \left|(f' \circ T)(X)\right|^2 \frac{1}{2}^2} \\
&= \frac{1}{4(p+q)} \expect{ \left|f'(Y)\right|^2}.
\end{align*}

The method of proof follows the general outline in the notes of Bakry~\cite{Bakry}.  Define the Jacobi differential operator $L$ to be
\[
Lf = (1-x^2)f''(x) + (q-p - (p+q)x)f'(x),
\]
and define the \emph{carr\'e du champ} operator $\Gamma$ by
\[
\Gamma(f,g) = (1-x^2)f'(x)g'(x).
\]
It can be checked by integration by parts that for all $C^2$ functions on $[-1,1]$ that the Dirichlet form $\mathcal{E}(f,g)$ associated to $L$ satisfies 
\[
\mathcal{E}(f,g)
\Def
-\int_{-1}^1 f(x) (Lg)(x) d\mu_{\beta}(x)
=
\int_{-1}^1 \Gamma(f(x),g(x)) d\mu_{\beta}(x).
\]
The spectrum of $L$ restricted to $L^2(\mu_\beta)$ is non-positive, with eigenvalues $y_n = -n(n + p+q-1)$ for non-negative integers $n.$  Further, its eigenfunctions are given by the Jacobi polynomials $P_n^{p-1,q-1}(x),$ which when normalized form a complete orthonormal system for $L^2(\mu_\beta).$  From the density of the polynomials in $L^2(\mu_\beta)$, it is an immediate consequence that
\[
p+q = -y_1 
=\inf_{\substack{f \in L^2(\mu) \\ \expect f(X) = 0}} \frac{\mathcal{E}(f,f)}{\Var f(X)}
=\inf_{\substack{f \in L^2(\mu) \\ \expect f(X) = 0}} \frac{\int_{-1}^1 \Gamma(f(x),g(x)) d\mu_{\beta}(x)}{\Var f(X)},
\]  
which upon rewriting, gives \eqref{pistatement}.
\end{proof}

\section{Coupling Bound for $\sqrt{\operatorname{Beta}}$ }
\label{sec:rootbeta}
We provide an auxiliary lemma regarding the square root of Beta variables that appear in the matrix entries.  Note that because one of the parameters of the $c_i'$ family is not $\Omega(n)$ for all $i$, this approximation can not be applied to every matrix entry with uniform error.
\begin{lemma}
\label{root_beta_clt}
If $Y$ is distributed as $\sqrt{\text{Beta}(np,nq)},$ then
\[
\frac{2(p+q)}{\sqrt{q}}\sqrt{n}\left(Y - \sqrt{\frac{p}{p+q}}\right) \Rightarrow N(0,1),
\]
as $n \tendsto \infty,$ where $p,q$ are fixed positive constants.
Moreover, it is possible to couple $Y$ to a standard normal $X$ so that
\[
\Var \left(Y - \frac{\sqrt{q}}
{2(p+q)\sqrt n}X\right) \leq \frac{K_{p,q}}{n^2},
\] 
for some $K_{p,q} > 0$, independent of $n$, and continuous in $p,q$ positive, provided that $n > \max\{\tfrac 1p, \tfrac 1q\}.$ 
\end{lemma}

\begin{proof}[Proof of ~\ref{root_beta_clt}]
Let $Y$ be distributed as $\sqrt{\text{Beta}(np,nq)}.$  Put
\[
\mu = \sqrt\frac{p}{p+q}, ~~~~~ \sigma = \sqrt\frac{q}{2n(p+q)}.
\]
Note that these are not exactly the mean or standard deviation of $Y$, however,
\[
\tilde{Y} = \frac{Y-\mu}{\sigma} \Rightarrow N(0,1).
\]
Moreover, it will be shown that there is an $X$ distributed as $N(0,1)$ so that
\[
\Exp (\tilde Y - X)^2 \leq \frac{K}{n}.
\]
for some $K=K(p,q)$ depending continuously on $p,q$ positive.  Note that this implies Lemma~\ref{root_beta_clt} after dividing through by $\sigma.$ 

The primary machinery here is Talagrand's transport inequality, which bounds the square $L^2$-Wasserstein distance of $\tilde Y$ and $X$, with $X$ distributed as $N(0,1).$  We use a special case of Theorem 1.1 of~\cite{Talagrand}, which states
\begin{proposition}[Talagrand]
Let $\tilde Y$ be a random variable given by probability measure $\tilde\nu,$ which is absolutely continuous with Lebesgue measure, and let $\gamma$ be a standard Gaussian measure.  There is a standard normal random variable $X$ so that 
\[
\Exp (\tilde Y-X)^2 \leq 2\int \log \tfrac{d\tilde\nu}{d\gamma} d\tilde\nu.
\]
\end{proposition}
The density $\tfrac{d\nu}{dy}$ of $Y$ can be computed to be
\[
\frac{d\nu}{dy} = 2y(y^2)^{np-1}(1-y^2)^{nq-1} \frac{\Gamma(np+nq)}{\Gamma(np)\Gamma(nq)}
\]
for $y \in [0,1].$  It follows that density of $\tilde Y$ is given by
\[
\frac{d\tilde{\nu}}{dy} = 2\sigma(\mu + y\sigma)^{2np-1}(1-(\mu + y\sigma)^2)^{nq-1} \frac{\Gamma(np+nq)}{\Gamma(np)\Gamma(nq)},
\]
and thus the Radon-Nikodym derivative $\tfrac {d\tilde\nu}{d\gamma}(y)$ is a product of four terms
\[
\frac {d\tilde\nu}{d\gamma}(y)
=
 \underbrace{\mystrut{2.5ex}(\mu + y\sigma)^{2np-1}}_\text{(i)} 
 \underbrace{\mystrut{2.5ex}(1-(\mu + y\sigma)^2)^{nq-1}}_\text{(ii)}
 \underbrace{\mystrut{2.5ex} e^{y^2/2} }_\text{(iii)}
 \underbrace{\mystrut{2.5ex}2\sigma\frac{\Gamma(np+nq)}{\Gamma(np)\Gamma(nq)} \sqrt{2\pi}}_{(iv)}.
\]
The logs of terms $(i)$ and $(ii)$ can be controlled by Taylor expansion.  Explicitly, 
\[
\ln [1 + y]^g  = g \ln( 1 + y) \leq g \left[ y - \frac{y^2}{2} + \frac{y^3}{3} \right],
\]
for all $y > -1,$ and all $g > 0.$   Note that both produce a nonzero constant term, by virtue of the relationship $\ln(a + y) = \ln(a) + \ln(1+y/a).$  This bound is applied to the logs of both $(i)$ and $(ii)$ after suitable rearrangement.  This bounds the sum of the logs by a polynomial in $y$ of degree $6.$  
We can bound the log of term $(i)$ as
\begin{align*}
\ln\left[(\mu + y\sigma)^{2np-1}\right] &= (2np-1)\ln \mu + (2np-1)\ln \left[1+ \tfrac {y\sqrt{q}}{\gamma\sqrt{2pn}}  \right] \\
&\leq (2np-1)\left[ \ln \mu + \tfrac {y\sqrt{q}}{\gamma\sqrt{2pn}} - \tfrac12\left(\tfrac {y\sqrt{q}}{\gamma\sqrt{2pn}}\right)^2 + \tfrac13\left(\tfrac {y\sqrt{q}}{\gamma\sqrt{2pn}}\right)^3\right]. 
\end{align*}
Applying the same to term $(ii),$ 
\begin{align*}
\hspace{1cm}&\hspace{-1cm}\ln \left[(1-(\mu + y\sigma)^2)^{nq-1}\right] \leq \\
&(nq-1)\left[ \ln (1-\mu^2) - \left(2\tfrac{y\sqrt{p}}{\sqrt{2qn}} + \left[\tfrac{y}{\sqrt{2n}}\right]^2\right) - \tfrac12 \left(2\tfrac{y\sqrt{p}}{\sqrt{2qn}} +\left[\tfrac{y}{\sqrt{2n}}\right]^2\right)^2 - \tfrac 13 \left(2\tfrac{y\sqrt{p}}{\sqrt{2qn}} + \left[\tfrac{y}{\sqrt{2n}}\right]^2\right)^3\right].
\end{align*}

From this form, it is easy to see that the coefficients of this polynomial depend continuously on $p$ and $q.$  Further, the coefficients of $y^4, y^5, \text{and } y^6$ already decay at least as fast as $1/n.$  The coefficient of $y^3$ decays like $n^{-1/2},$ so some amount of control over $\Exp {\tilde Y}^3$ will need to be gained.  The coefficients of the lower order terms to do not \emph{a priori} decay at all, but there is strong cancellation.
The constant term is
\[ 
C_0(p,q) \Def \left( 2\,n\,p-1 \right) \ln  \left( \mu \right) + \left( n\,q-1
 \right) \ln  \left( 1-{\mu}^{2} \right),
\]
the linear term has coefficient
\[
\frac{C_1(p,q)}{\sqrt{n}} \Def {\frac { \left( 2\,n\,p-1 \right) \sigma}{\mu}}-2\,{\frac { \left( n\,
q-1 \right) \mu\,\sigma}{1-{\mu}^{2}}},
\]
and the quadratic term has coefficient
\[
 -\frac12 + \frac{C_2(p,q)}{n} \Def
-1/2\,{\frac { \left( 2\,n\,p-1 \right) {\sigma}^{2}}{{\mu}^{2}}}+
 \left( n\,q-1 \right)  \left( -{\frac {{\sigma}^{2}}{1-{\mu}^{2}}}-2
\,{\frac {{\mu}^{2}{\sigma}^{2}}{ \left( 1-{\mu}^{2} \right) ^{2}}}
 \right).
\]
The $-\frac12$ in the quadratic term represents the asymptotically Gaussian portion, and it annihilates term $(iii).$  This leaves four sources of error that need to be controlled to show the desired $O(n^{-1})$ bound:
\begin{enumerate}
\item $|\Exp \tilde Y| \leq C(p,q)n^{-\tfrac12}$ to control the linear term.
\item $|\Exp {\tilde Y}^3| \leq C(p,q)n^{-\tfrac12}$ to control the cubic term.
\item $|\Exp (\tilde Y)^k| \leq C(p,q)$ to control the second, fourth, fifth, and sixth terms.
\item The constants from the Taylor approximation and the constants from part $(iv)$ of the Radon-Nikodym derivative need to cancel to order $O(n^{-1}).$
\end{enumerate}

The raw moments of $Y$ are easily computable, and their formula follows immediately from Euler's Beta integral,
\[
\Exp \left(Y\right)^k = \frac{\Gamma(n(p+q)) \Gamma(np + \tfrac k2)}{\Gamma(n(p+q) + \tfrac k2)\Gamma(np)}.
\] 
Appropriate control over the first $6$ raw moments could be achieved by taking sufficiently many terms from the Stirling approximation and canceling terms.  To some extent, doing such a procedure is necessary, as this is necessary to get the precise control over the first and third raw moments.  However, we will not need to do this for all $6$ moments, because we can appeal to a Poincar\'e inequality.  Provided that $n > \max\{\tfrac 1p,\tfrac 1q\},$ the density $\tfrac{d\tilde\nu}{dy}$ is log-concave.  Thus if it can be shown that $\tilde Y$ has constant order variance, we can use the Poincar\'e inequality to bound higher moments by lower moments, i.e.
\[
\Var f(\tilde Y) \leq C \Exp |f'(\tilde Y)|^2,
\]
applied to $f(\tilde Y) = (\tilde Y)^k,$ gives
\[
\Exp {\tilde Y}^{2k} \leq \left(\Exp{ \tilde Y}^k\right)^2 + Ck^2\Exp {\tilde{Y}}^{2k-2}.
\]
Because of the log-concavity, $C$ can be taken to be $12 \Exp |\tilde Y|^2$ (see Corr 4.3 of~\cite{Bobkov}), which is continuous in $p$ and $q$.  Thus provided that $\Exp |\tilde Y|$ can be bounded by some continuous function in $p$ and $q,$ iterating the Poincar\'e inequality gives constant order bounds that are continuous in $p$ and $q$ for all absolute moments.  Further,
\[
\Exp |\tilde Y| \leq \sqrt{ \Exp \left(\tilde Y\right)^2 },
\]
so the problem has been reduced to finding good bounds for the first three raw moments of $\tilde Y.$
 
By appealing to Stirling's formula, and using that the error-in-approximation is bounded by the first omitted term in the asymptotic expansion, the first three moments of $\tilde Y$ can be bounded by
\begin{align*}
\left|\Exp \left(\tilde Y\right)\right| &\leq \frac{\sqrt q}{4\sqrt{p(p+q)}}n^{-\tfrac12}, \\ 
\left|\Exp \left(\tilde Y\right)^2\right| &\leq 1, \\
\left|\Exp \left(\tilde Y\right)^3\right| &\leq \frac{8p+q}{4\sqrt{qp(p+q)}}n^{-\tfrac12}. \\
\end{align*}
It only remains to control the constant terms.  The log of $(iv)$ can be approximated by Stirling's formula:
\[
\left|\ln\left[2\sigma\frac{\Gamma(n(p+q)}{\Gamma(np)\Gamma(nq)}\sqrt{2\pi}\right] - \left[ -np\ln \mu^2 - nq\ln(1-\mu^2) + \ln\frac{q\sqrt p}{(p+q)^{\tfrac32}} \right] \right| \leq \frac{1}{12} ~ \frac{1}{n} ~ \frac{1}{\sqrt{p}\sqrt{q}{\sqrt{p+q}}}.
\]
Comparing this with the constants produced by the Taylor approximation on terms $(i)$ and $(ii),$ it is seen that only the $O(n^{-1})$ term remains.

\end{proof}

\end{appendices} 

\bibliographystyle{plain}
\bibliography{Jacobirefs}

\end{document}